\documentclass[a4paper,11pt,fleqn]{article}

\usepackage{amsmath,amsthm,amsfonts}
\usepackage{amssymb}
\usepackage{mathrsfs}
\newtheorem{theorem}{Theorem}[section]
\newtheorem{lemma}{Lemma}[section]
\newtheorem{corollary}{Corollary}[section]
\newtheorem{definition}{Definition}[section]
\numberwithin{equation}{section}

\begin{document}
\title{On preconditioned AOR method for solving linear systems}
\author{Yongzhong Song\\\\
\small{\it Jiangsu Key Lab for NSLSCS,
School of Mathematical Sciences,} \\
\small{\it  Nanjing Normal University,
Nanjing 210023, People's Republic of China}\\
\small{\it Email: yzsong@njnu.edu.cn}}
\date{}
\maketitle

\begin{abstract}
In this paper, we investigate the preconditioned AOR method for solving linear systems.
We study two general preconditioners and propose some lower triangular,
upper triangular and combination preconditioners. For $A$ being
an L-matrix, a nonsingular M-matrix, an irreducible L-matrix
and an irreducible nonsingular M-matrix, four types of comparison theorems are presented,
respectively. They contain a general comparison result, a strict comparison result and two
Stein-Rosenberg type comparison results. Our theorems include and are better than
almost all known corresponding results.


\end{abstract}

\section{Introduction}

Consider a system of $n$ equations
 \begin{eqnarray}\label{eqn1.1}
 Ax = b,
 \end{eqnarray}
where $A = (a_{i,j})\in{\mathscr R}^{n \times n}$, $b$, $x\in{\mathscr R}^{n}$
with $b$ known and $x$ unknown. In order to solve the system
(\ref{eqn1.1}) with iterative methods, the coefficient matrix $A$
is split into
 \begin{eqnarray}\label{eqn1.2}
 A = M - N,
 \end{eqnarray}
where $M$ is nonsingular and $N \not = 0$. Then a linear stationary
iterative method for solving (\ref{eqn1.1}) can be described as
 \begin{eqnarray}\label{eqn1.3}
 x^{k+1} = Tx^k + M^{-1}b, \quad k = 0,1,2,\cdots,
 \end{eqnarray}
where $T = M^{-1}N$ is the iteration matrix.

We decompose $A$ into
\begin{eqnarray*}
A = D - L - U,
 \end{eqnarray*}
where $D$ is a diagonal matrix, $L$ and $U$ are strictly lower and
upper triangular matrices, respectively, as usual.

For $\omega\in{\mathscr R}\setminus\{0\}$ and $\gamma\in{\mathscr R}$, let
 \begin{eqnarray}\label{eqn1.4}
 A = M_{\gamma,\omega} - N_{\gamma,\omega},
 \end{eqnarray}
 where
 \begin{eqnarray*}
 M_{\gamma,\omega} = \frac{1}{\omega}(D - \gamma L), \;
 N_{\gamma,\omega} = \frac{1}{\omega} \left[(1-\omega)D +
(\omega-\gamma)L + \omega U \right].
 \end{eqnarray*}
Then the AOR method for solving (\ref{eqn1.1}) is defined in \cite{Ha78} by
 \begin{eqnarray}\label{eqn1.5}
 x^{k+1} = {\mathscr L}_{\gamma,\omega}x^{k} +
 \omega(D - \gamma L)^{-1}b, \quad k = 0,1,2,\ldots,
 \end{eqnarray}
 where
\begin{eqnarray*}
{\mathscr L}_{\gamma,\omega} = (D - \gamma L)^{-1} [(1-\omega)D +
(\omega-\gamma)L + \omega U]
\end{eqnarray*}
 is the AOR iteration matrix. The splitting (\ref{eqn1.4}) is also called
 the AOR splitting of $A$.

When $(\gamma,\omega)$ is equal to $(\omega, \omega)$, (1, 1) and (0, 1),
the AOR method reduces respectively to the SOR method, Gauss-Seidel method and Jacobi method, whose iteration
matrices are represented by ${\mathscr L}_{\omega}$, ${\mathscr L}$ and ${\mathscr J}$.

In \cite{Ha78} it is pointed out that, for $\gamma \not = 0$, the AOR method is an extrapolated SOR (ESOR)
method with overrelaxation parameter $\gamma$ and extrapolation one $\omega/\gamma$, i.e.,
 \begin{eqnarray}\label{eqn1.6}
 {\mathscr L}_{\gamma,\omega} = \left(1 - \frac{\omega}{\gamma}\right)I
 + \frac{\omega}{\gamma}{\mathscr L}_{\gamma},
 \end{eqnarray}
and, hence, if $\eta$ is an eigenvalue of ${\mathscr L}_{\gamma}$ and $\lambda$,
the corresponding one of ${\mathscr L}_{\gamma,\omega}$, then we have
 \begin{eqnarray} \label{eqn1.7}
 \lambda = 1 - \frac{\omega}{\gamma} + \frac{\omega}{\gamma}\eta.
 \end{eqnarray}

It is well known that, if $A$ is
nonsingular, then the iterative method (\ref{eqn1.3}) is convergent
if and only if the spectral radius $\rho(T)$ of the iteration
matrix $T$ is less than 1. In this case, the matrix $T$ is also called
convergent. However, if $A$ is singular, then we
have $\rho(T) \ge 1$, so that we can only require the
semiconvergence of the splitting.
When the iterative method (semi)converges, the convergence
speed is determined by $\rho(T)$, and the smaller it is, the
faster the iterative method converges. Therefore $\rho(T)$ is
called convergence factor.

In order to decrease the spectral radius of the iteration matrix,
an effective method is to precondition the linear system
(\ref{eqn1.1}). It is well known that the term preconditioning
refers to transforming the system (\ref{eqn1.1}) into another
system with more favorable properties for iterative methods.

If $P$ is a nonsingular matrix, then the preconditioned linear
system
\begin{eqnarray*}
PAx = Pb
\end{eqnarray*}
has the same solution as (\ref{eqn1.1}). Here $P$ is called the
preconditioner.

Generally speaking, preconditioning attempts to improve the
spectral properties of the coefficient matrix. A good preconditioner $P$ should meet the following requirements:

\begin{itemize}
\item[$\bullet$] The preconditioned system should have more favorable properties for iterative methods,
in particular, the iterative methods can be convergent more faster.

\item[$\bullet$] The preconditioner should be cheap to construct.
\end{itemize}

To choose a good preconditioner $P$ is an interesting problem,
which has been investigated widely.
In a large number of papers, in particular for the AOR method, a special preconditioner $P$
is proposed by
\begin{eqnarray*}
P = I + Q, \; Q > 0.
 \end{eqnarray*}
Then we define a matrix splitting
\begin{eqnarray*}
PA = \hat{M} - \hat{N}.
 \end{eqnarray*}
A preconditioned iterative method can be defined by
\begin{eqnarray*}
 x^{k+1} = \hat{T}x^k + \hat{M}^{-1}b, \quad k = 0,1,2,\cdots,
 \end{eqnarray*}
where $\hat{T} = \hat{M}^{-1}\hat{N}$ is the iteration matrix.

When $A$ is an L-matrix, a nonsingular M-matrix, an irreducible L-matrix
or an irreducible nonsingular M-matrix, the preconditioned AOR,
SOR, Gauss-Seidel and Jacobi methods are constructed, generalized and
applied by
\cite{AH04,DHS11,DH11,DH14,EM95,EMT01,GJS91,HNT03,HS14,HDS10,HN01,HCC06,HCEC05,HWF07,
HWXC16,KO98,KKNU97,KN09,KHMN02,KNO96,KNO97,KNO972,Li11,Li12,LE94,LW14,LH05,LJ10,Li02,
Li03,Li05,LL07,LS00,LLW07,LLW072,LW04,LW042,LY08,LL12,Li08,LC09,LC10,LCC08,LHZ15,MEY01,Mi87,Mo10,
NA12,NHMS04,NK10,NKA09,NKM08,SE13,SE15,SER14,SSHL18,Su05,Su06,UKN95,UNK94,WZ11,WL09,Wa06,Wa062,
WS09,WH06,WWS07,WH07,WH09,YZ12,Yu07,Yu072,Yu08,Yu11,Yu12,YK08,YLK14,ZHL05,ZHL07,ZHLG07,ZM09,ZSWL09}.

In this paper, we investigate the preconditioned AOR method for solving linear systems.
We study two general preconditioners and propose some lower triangular,
upper triangular and combination preconditioners. For $A$ being
an L-matrix, a nonsingular M-matrix, an irreducible L-matrix
and an irreducible nonsingular M-matrix, four types of comparison theorems are presented,
respectively. They contain a general comparison result, a strict comparison result and two
Stein-Rosenberg type comparison results. Our theorems include and are better than
almost all known corresponding results. Some incorrect results are pointed out.

This paper is organized as follows. In Section \ref{se2} we give
some concepts and results, which will applied in the next section.
In Section \ref{se3}, we study two general preconditioners, some lower triangular,
upper triangular and combination preconditioners for the preconditioned AOR method.
Four types of comparison results are proved. In Section \ref{se4}, we give some
explanations and prospects.

\section{Some concepts and lemmas}\label{se2}

For convenience we recall and give some concepts and lemmas as follows.

A matrix $B\in{\mathscr R}^{n\times m}$ is called nonnegative, semi-positive, positive if
each element of $B$ is nonnegative, nonnegative but at least a positive element, positive,
which is denoted by $B \ge 0$, $B > 0$ and $B \gg 0$, respectively. When
$B_1 - B_2 \ge (>, \gg) 0$, we denote $B_1 \ge (>, \gg) B_2$ or $B_2 \le (<, \ll) B_1$.
Similarly, for $y\in{\mathscr R}^n$, by identifying it with $n\times 1$ matrix, we can
also define $y \ge (\le) 0$, $y > (<) 0$ and $y \gg (\ll) 0$. $B$ is called
monotone, if $B$ is invertible and $B^{-1} > 0$.
$B = (b_{i,j})\in{\mathscr R}^{n\times n}$ is called a
Z-matrix if $b_{i,j} \le 0$ for $i,j = 1, \cdots, n$, $i\neq j$; an L-matrix if it
is a Z-matrix with $a_{i,i} > 0$, $i = 1, \cdots, n$;
a nonsingular M-matrix if it is a Z-matrix and is monotone. It is well known that
a nonsingular M-matrix is an L-matrix.

A matrix $B\in{\mathscr R}^{n\times n}$
is called reducible if there is a permutation matrix $V$ such that
\begin{eqnarray*}
VBV^T = \left[\begin{array}{cc}
 B_{1,1} & B_{1,2} \\
0 & B_{2,2}
\end{array}\right],
\end{eqnarray*}
where $B_{1,1}\in{\mathscr R}^{r\times r}$, $B_{2,2}\in{\mathscr R}^{(n-r)\times (n-r)}$
with $1 \le r \le n-1$. Otherwise, $B$ is irreducible.
The directed graph of a matrix $B = (b_{i,j})\in{\mathscr R}^{n\times n}$ is denoted
by $G(B)$. A path in $G(B)$ which leads from the vertex $V_i$
to the vertex $V_j$ is denoted by $\sigma_{i,j}$, i.e.,
$\sigma_{i,j} = (j_0, j_1, \cdots, j_{l+1})$ with $i = j_0$, $j = j_{l+1}$, $l \ge 0$ and
$b_{j_kj_{k+1}} \not= 0$, $k=0,\cdots,l$. It is well known that
a matrix B is irreducible if and only if $G(B)$ is strongly connected, which
means that for any $i,j\in\{1,\cdots,n\}$ there exists a path $\sigma_{i,j}\in G(B)$.

\begin{definition} \upshape
The decomposition (\ref{eqn1.2}) is called a splitting of $A$ if $M$ is nonsingular. A
splitting is called:

\begin{itemize}
\item[(i)] Regular if $M^{-1} \ge 0$ and
$N \ge 0$ (cf. \cite[Definition 3.28]{Va00});

\item[(ii)] Weak regular if $M^{-1}
 \ge 0$ and $M^{-1}N \ge 0$ (cf. \cite[Definition 3.28]{Va00});

\item[(iii)] Nonnegative if $M^{-1}N \ge 0$ (cf.
\cite[Definition 1.1]{So91}).
 \end{itemize}
 \end{definition}

\begin{lemma} {\upshape\cite[Theorems 2.7 and 2.20]{Va00}} \label{lem1-1}

\begin{itemize}
\item[(i)] Let $B \ge 0$. Then $B$ has a nonnegative eigenvalue equal to $\rho(B)$, and there corresponds an
eigenvector $x > 0$.

\item[(ii)] Let $B \ge 0$ be irreducible. Then $B$ has a positive eigenvalue equal to $\rho(B)$, and there corresponds an
eigenvector $x \gg 0$.
\end{itemize}
\end{lemma}

\begin{lemma} {\upshape\cite[Theorem 2-1.11]{BP94}} \label{lem1-2}
Let $B \ge 0$.

\begin{itemize}
\item[(i)] If $Bx \ge \alpha x$ with $x > 0$, then $\rho(B) \ge \alpha$.

\item[(ii)] If $Bx \le \beta x$ with $x \gg 0$, then $\rho(B) \le \beta$.

\item[(iii)] If $B$ is irreducible and $Bx > \alpha x$ with
$x > 0$, then $\rho(B) > \alpha$.

\item[(iv)] If $B$ is irreducible and $Bx < \beta x$ with
$x > 0$, then $\rho(B) < \beta$ and $x \gg 0$.
\end{itemize}
\end{lemma}

Here we have made a minor modification to \cite[Theorem 2-1.11]{BP94}.
In fact, for ($iii$), $x \gg 0$ cannot be derived.

\begin{lemma} \label{lem1-3}
Let $B \ge 0$ and $x > 0$.

\begin{itemize}
\item[(i)] If $Bx \gg \alpha x$ then $\rho(B) > \alpha$.

\item[(ii)] If $Bx \ll \beta x$ then $\rho(B) < \beta$ and $x \gg 0$.
\end{itemize}
\end{lemma}

\begin{proof}
Since $B \ge 0$, then there exists $y > 0$ such that $B^Ty = \rho(B)y$.
Multiply $y^T$ on the left side of $Bx \gg \alpha x$ or $Bx \ll \beta x$
respectively, we can obtain $\rho(B)y^Tx > \alpha y^Tx$ or $\rho(B)y^Tx < \beta y^Tx$,
which derives $\rho(B) > \alpha$ or $\rho(B) < \beta$ directly. When
$Bx \ll \beta x$, $x \gg 0$ is obvious.
\end{proof}

\begin{lemma} {\upshape\cite[Theorem 6-2.7]{BP94}} \label{lem1-4}
Let $B$ be an irreducible Z-matrix. Then $B$ is a nonsingular M-matrix
if and only if $B^{-1} \gg 0$.
 \end{lemma}

\begin{lemma} {\upshape\cite[Theorem 6-2.3]{BP94}} \label{lem1-5}
Let $B$ be a Z-matrix. Then the following statements are
equivalent:

\begin{itemize}
\item[(i)] $B$ is a nonsingular M-matrix.

\item[(ii)] There is a vector $x \gg 0$ such that $Bx \gg 0$.

\item[(iii)] The weak regular splitting of $B$ is convergent.
 \end{itemize}
 \end{lemma}

\begin{lemma} {\upshape\cite[Theorem 3.37]{Va00}} \label{lem1-6}
Any weak regular splitting of $B$ is convergent if and only if $B$ is nonsingular with $B^{-1} > 0$.
\end{lemma}

\begin{lemma} {\upshape\cite[Exerxise 3.3-6]{Va00}} \label{lem1-7}
Let $A$ be an irreducible L-matrix.
Then $\rho({\mathscr L}) > 0$ and it has associated eigenvector $x \gg 0$.
\end{lemma}

\begin{lemma}\label{lem1-8}
Let $A$ be an irreducible L-matrix, and let $0 \le \gamma \le \omega$ and $\omega > 0$.

\begin{itemize}
\item[(i)] Then $\rho({\mathscr L}_{\gamma,\omega}) > 0$ holds.

\item[(ii)] Assume that $x > 0$ satisfies
${\mathscr L}_{\gamma,\omega} x = \rho({\mathscr L}_{\gamma,\omega}) x$. Then $x \gg 0$.

\item[(iii)] Assume that $y > 0$ satisfies
$y^T{\mathscr L}_{\gamma,\omega} = \rho({\mathscr L}_{\gamma,\omega}) y^T$.
Then $y^T(D - \gamma L)^{-1} \gg 0$.
 \end{itemize}
\end{lemma}

\begin{proof} Denote $\rho = \rho({\mathscr L}_{\gamma,\omega})$.
Clearly, $\rho = 0$ if and only if $N_{\gamma,\omega}x = 0$.
Since $A$ is irreducible L-matrix, then it gets that $N_{\gamma,\omega} > 0$.

When $\omega < 1$, we have $N_{\gamma,\omega}x \ge (1-\omega)/\omega x > 0$. When $\omega = 1$ and $\gamma < 1$,
we have that $N_{\gamma,\omega} = (1-\gamma)L + U$ is irreducible so that $N_{\gamma,\omega}x > 0$.
Hence, for these two cases we obtain $\rho > 0$.
When $\omega = \gamma = 1$, by Lemma \ref{lem1-7} we have also $\rho > 0$.
We have proved ($i$).

From $(D - \gamma L)^{-1}[(1-\omega)D + (\omega-\gamma)L + \omega U]x = \rho x$,
we obtain
\begin{eqnarray}\label{eqs2.1}
[(\omega + \rho -1)D - (\omega-\gamma+\gamma\rho)L - \omega U]x = 0.
\end{eqnarray}
Let $\hat{A} = (\omega + \rho -1)D - (\omega-\gamma+\gamma\rho)L - \omega U$.
Since $\omega > 0$ and $\omega-\gamma+\gamma\rho > 0$, then $\hat{A}$ is an irreducible Z-matrix.
If $x$ has some zero elements, without loss of generality, then we can assume that
\begin{eqnarray*}
x = \left(\begin{array}{c}
0\\
\hat{x}
\end{array}\right), \; \hat{x} \gg 0\in{\mathscr R}^{\;r}, \; 1 \le r \le n-1.
\end{eqnarray*}
We divide $\hat{A}$ accordingly as
\begin{eqnarray*}
\hat{A} = \left(\begin{array}{cc}
A_{1,1} & A_{1,2}\\
A_{2,1} & A_{2,2}
\end{array}\right), \; A_{2,2}\in{\mathscr R}^{r\times r}.
\end{eqnarray*}
Then $A_{1,2} \le 0$. From (\ref{eqs2.1}) we derive $A_{1,2}\hat{x} = 0$ so that
$A_{1,2} = 0$, which implies that $\hat{A}$ is reducible. This is a contradiction. Hence,
$x$ has no zero elements, i.e., $x \gg 0$. This has proved ($ii$).

When $y^T{\mathscr L}_{\gamma,\omega} = \rho y^T$, let $z^T = y^T(D - \gamma L)^{-1}$. Then $z > 0$.
Further, we have $z^T[(1-\omega)D + (\omega-\gamma)L + \omega U] = \rho z^T(D - \gamma L)$,
so that $[(\omega + \rho -1)D - (\omega-\gamma+\gamma\rho)L^T - \omega U^T]z = 0$.
Since $A^T$ ia also an irreducible L-matrix, then similar to the proof of ($ii$) we can prove $z \gg 0$.
Then ($iii$) is proved. \end{proof}

\begin{lemma}\label{lem1-9}
Let $P > 0$ be nonsingular, and let the splitting
\begin{eqnarray}\label{eqs2.2}
PA = \hat{M} - \hat{N}
 \end{eqnarray}
be weak regular. Then the following
three statements are equivalent:

\begin{itemize}
\item[(i)] $A^{-1} > 0$.

\item[(ii)] $(PA)^{-1} > 0$.

\item[(iii)] The splitting (\ref{eqs2.2}) is convergent.
\end{itemize}
\end{lemma}

\begin{proof}
By Lemma \ref{lem1-6}, ($ii$) and ($iii$) are equivalent, immediately.

The splitting (\ref{eqs2.2}) can be rewritten into
$A = P^{-1}\hat{M} - P^{-1}\hat{N}$.
Clearly, this splitting is weak regular and $(P^{-1}\hat{M})^{-1}(P^{-1}\hat{N})
 = \hat{M}^{-1}\hat{N}$. The equivalence between ($i$) and ($iii$) follows directly
by Lemma \ref{lem1-6} again.  \end{proof}

From this lemma, the following lemma is obvious.

\begin{lemma}\label{lem1-10}
Let $A$ and $PA$ be Z-matrices, where $P > 0$ is nonsingular. Then
$A$ is a nonsingular M-matrix if and only if $PA$ is a nonsingular M-matrix.
\end{lemma}

We prove two Stein-Rosenberg type comparison theorems.

\begin{lemma} \label{lem1-11}
Let the splittings $A = M_1 - N_1 = M_2 - N_2$ be respectively weak regular and nonnegative,
and let $x \gg 0$, $y > 0$ satisfy $M_2^{-1}N_2 x = \rho(M_2^{-1}N_2)x$,
$y^TM_1^{-1}N_1 = \rho(M_1^{-1}N_1)y^T$. Suppose that one of the following two conditions is satisfied:

\begin{itemize}
\item[(i)] $M_1^{-1}(N_2 - N_1)x \gg 0$.

\item[(ii)] $y^TM_1^{-1} \gg 0$ and $(N_2 - N_1)x > 0$.
 \end{itemize}
Then one of the following mutually exclusive relations holds:

\begin{itemize}
\item[(a)] $\rho(M_1^{-1}N_1) < \rho(M_2^{-1}N_2) < 1$.

\item[(b)] $\rho(M_1^{-1}N_1) = \rho(M_2^{-1}N_2) = 1$.

\item[(c)] $\rho(M_1^{-1}N_1) > \rho(M_2^{-1}N_2) > 1$.
 \end{itemize}
\end{lemma}

\begin{proof}
Since $Ax = (M_2-N_2)x = [1-\rho(M_2^{-1}N_2)]M_2x$, then by simple operation we have
 \begin{eqnarray}\label{eqs2.3} \nonumber
M_1^{-1}N_1x - \rho(M_2^{-1}N_2)x & = & (M_1^{-1}N_1 - M_2^{-1}N_2)x\\\nonumber
& = & (M_2^{-1}A - M_1^{-1}A)x\\
& = & M_1^{-1}(M_1 - M_2)M_2^{-1}Ax\\\nonumber
& = & [\rho(M_2^{-1}N_2)-1]M_1^{-1}(M_2 - M_1)x \\ \nonumber
& = & [\rho(M_2^{-1}N_2)-1]M_1^{-1}(N_2 - N_1)x
 \end{eqnarray}
and, therefore,
 \begin{eqnarray*}
[\rho(M_1^{-1}N_1) - \rho(M_2^{-1}N_2)]y^Tx = [\rho(M_2^{-1}N_2) - 1]y^TM_1^{-1}(N_2 - N_1)x.
 \end{eqnarray*}

When one the conditions ($i$) and ($ii$) is satisfied, it derives that $y^TM_1^{-1}(N_2 - N_1)x > 0$.
Since $y^Tx > 0$, then we obtain
 \begin{eqnarray*}
\rho(M_1^{-1}N_1) - \rho(M_2^{-1}N_2)
 \left\{ \begin{array}{llll}
 < & 0, & if & \rho(M_2^{-1}N_2) < 1,\\
 = & 0, & if & \rho(M_2^{-1}N_2) = 1,\\
 > & 0, & if & \rho(M_2^{-1}N_2) > 1.
 \end{array}\right.
 \end{eqnarray*}
The proof is completed.
  \end{proof}

By the definition of the AOR method, when $\omega = \gamma = 1$ in (\ref{eqn1.4}) we
derive the Gauss-Seidel method, whose iteration matrix is denoted by ${\mathscr L}$.
Now, let
\begin{eqnarray*}
PA = \tilde{D} - \tilde{L} - \tilde{U} = \tilde{M}_{1,1} - \tilde{N}_{1,1},
 \end{eqnarray*}
where $\tilde{D}$ is a diagonal matrix, $\tilde{L}$ and $\tilde{U}$ are strictly lower and
upper triangular matrices respectively and
$\tilde{M}_{1,1} = \tilde{D} - \tilde{L}$, $\tilde{N}_{1,1} = \tilde{U}$. Then the
preconditioned Gauss-Seidel iteration matrix can be defined as $\tilde{\mathscr L}
 = \tilde{M}_{1,1}^{-1}\tilde{N}_{1,1} = (\tilde{D} - \tilde{L})^{-1}\tilde{U}$.

\begin{lemma} \label{lem1-12}
Let $A$ and $PA$ be L-matrices,
and let ${\mathscr L}x = \rho({\mathscr L})x$ with $x \gg 0$.
Suppose that the second to $n$th elements of $\tilde{M}_{1,1}^{-1}(N_{1,1} - \tilde{N}_{1,1})x$
are positive. Then one of the following mutually exclusive relations holds:

\begin{itemize}
\item[(a)] $\rho(\tilde{\mathscr L}) < \rho({\mathscr L}) < 1$.

\item[(b)] $\rho(\tilde{\mathscr L}) = \rho({\mathscr L}) = 1$.

\item[(c)] $\rho(\tilde{\mathscr L}) > \rho({\mathscr L}) > 1$.
 \end{itemize}
\end{lemma}

\begin{proof}
Consider the splittings
\begin{eqnarray*}
PA = \tilde{M}_{1,1} - \tilde{N}_{1,1} = PM_{1,1} - PN_{1,1}.
 \end{eqnarray*}
Clearly, they are regular and nonnegative respectively.

Since the first column $\tilde{\mathscr L}$ is a zero vector, then it can be
decomposed as
 \begin{eqnarray*}
 \tilde{\mathscr L} = \left(\begin{array}{cc}
 0 & \psi_{1,2}\\
 0 & \Psi_{2,2}
 \end{array}\right), \quad \Psi_{2,2} \ge 0\in{\mathscr R}^{(n-1) \times (n-1)},
 \end{eqnarray*}
so that $\rho(\tilde{\mathscr L}) = \rho(\Psi_{2,2})$.

 Correspondingly, we decompose $x$
 and $\tilde{M}_{1,1}^{-1}(N_{1,1} - \tilde{N}_{1,1})x$ as
 \begin{eqnarray*}
x = \left(\begin{array}{c}
x_{1}\\
x_{2}
 \end{array}\right), \;\;
  \tilde{M}_{1,1}^{-1}(N_{1,1} - \tilde{N}_{1,1})x = \left(\begin{array}{c}
 \tilde{x}_{1}\\
\tilde{x}_{2}
 \end{array}\right),
 \quad
 x_2, \tilde{x}_{2}\in{\mathscr R}^{(n-1)}.
 \end{eqnarray*}
Then $x_2 \gg 0$ and $\tilde{x}_{2} \gg 0$.

Similar to the proof of Lemma \ref{lem1-11}, by (\ref{eqs2.3}) we can obtain that
\begin{eqnarray*}
 \left(\begin{array}{cc}
 0 & \psi_{1,2}\\
 0 & \Psi_{2,2}
 \end{array}\right)
 \left(\begin{array}{c}
x_{1}\\
x_{2}
 \end{array}\right)
  - \rho({\mathscr L}) \left(\begin{array}{c}
x_{1}\\
x_{2}
 \end{array}\right)
 = [\rho({\mathscr L})-1] \left(\begin{array}{c}
 \tilde{x}_{1}\\
\tilde{x}_{2}
 \end{array}\right),
 \end{eqnarray*}
so that
\begin{eqnarray*}
\Psi_{2,2}x_{2} - \rho({\mathscr L})x_{2}
 = [\rho({\mathscr L})-1]\tilde{x}_{2}.
 \end{eqnarray*}

Since $\Psi_{2,2} \ge 0$, then there exists $z > 0$ such that $z^T\Psi_{2,2}
 = \rho(\Psi_{2,2})z^T$. Hence, we have
\begin{eqnarray*}
[\rho(\Psi_{2,2}) - \rho({\mathscr L})]z^Tx_{2} = z^T\Psi_{2,2}x_{2} - \rho({\mathscr L})z^Tx_{2}
 = [\rho({\mathscr L})-1]z^T\tilde{x}_{2}.
 \end{eqnarray*}

Because of $x_{2} \gg 0$, $\tilde{x}_{2} \gg 0$ and $z > 0$, we drives $z^Tx_{2} > 0$
and $z^T\tilde{x}_{2} > 0$, so that
 \begin{eqnarray*}
\rho(\Psi_{2,2}) - \rho({\mathscr L})
 \left\{ \begin{array}{llll}
 < & 0, & if & \rho({\mathscr L}) < 1,\\
 = & 0, & if & \rho({\mathscr L}) = 1,\\
 > & 0, & if & \rho({\mathscr L}) > 1.
 \end{array}\right.
 \end{eqnarray*}
The proof is completed.
  \end{proof}

Similar to \cite[Theorem 3.4]{So97}, we prove a strictly comparison result.

\begin{lemma} \label{lem1-13}
Let the both splittings $A_1 = M_1 - N_1$ and $A_2 = M_2 - N_2$ be nonnegative and convergent with
$A_2^{-1} \gg 0$ and $A_2^{-1} \ge A_1^{-1}$.
Suppose that there exists $x > 0$ such that $M_2^{-1}N_2x = \rho(M_2^{-1}N_2)x$
and $M_2x > M_1x \ge 0$.
Then $\rho(M_1^{-1}N_1) < \rho(M_2^{-1}N_2)$.
\end{lemma}

\begin{proof}
We have that
\begin{eqnarray*}
A_2^{-1}M_2x \gg A_2^{-1}M_1x \ge A_1^{-1}M_1x,
  \end{eqnarray*}
  so that
\begin{eqnarray*}
\frac{1}{1 - \rho(M_2^{-1}N_2)}x = A_2^{-1}M_2x \gg (I - M_1^{-1}N_1)^{-1}x.
  \end{eqnarray*}
Clearly, $(I - M_1^{-1}N_1)^{-1} \ge 0$. It follows by Lemma \ref{lem1-3} that
\begin{eqnarray*}
\frac{1}{1 - \rho(M_2^{-1}N_2)} > \rho((I - M_1^{-1}N_1)^{-1}) = \frac{1}{1 - \rho(M_1^{-1}N_1)}.
  \end{eqnarray*}
The required result can be derived.
\end{proof}

\section{Preconditioned AOR method and comparison results}\label{se3}

In this section, without loss of generality, suppose that all of the diagonal elements of
$A$ are 1. In this case, $A$ is an L-matrix if and only if $A$ is a Z-matrix.

For convenience, if the matrix $Q$ is chosen as
$Q_\nu$, then we write $P_\nu = I + Q_\nu$, $Q_\nu = (q^{(\nu)}_{i,j})$ and
$A^{(\nu)} = P_\nu A = (a^{(\nu)}_{i,j})$. Let
\begin{eqnarray}\label{eqs3.1}
A^{(\nu)} = D_\nu - L_\nu - U_\nu = M^{(\nu)}_{\gamma,\omega} - N^{(\nu)}_{\gamma,\omega}
  \end{eqnarray}
with
  \begin{eqnarray*}
M^{(\nu)}_{\gamma,\omega} = \frac{1}{\omega}(D_\nu - \gamma L_\nu), \;
N^{(\nu)}_{\gamma,\omega} = \frac{1}{\omega} \left[(1-\omega)D_\nu +
(\omega-\gamma)L_\nu + \omega U_\nu \right],
  \end{eqnarray*}
where $D_\nu = diag(A^{(\nu)})$ is a diagonal matrix, $L_\nu$ and
$U_\nu$ are strictly lower and upper triangular matrices
respectively. Then the corresponding preconditioned AOR
method for solving (\ref{eqn1.1}) can be defined as
 \begin{eqnarray*}
 x^{k+1} = {\mathscr L}^{(\nu)}_{\gamma,\omega}x^{k} +
 \omega(D_\nu - \gamma L_\nu)^{-1}P_\nu b, \quad k = 0,1,2,\ldots,
 \end{eqnarray*}
 where
\begin{eqnarray*}
{\mathscr L}^{(\nu)}_{\gamma,\omega} = (D_\nu - \gamma L_\nu)^{-1}
[(1-\omega)D_\nu + (\omega-\gamma)L_\nu + \omega U_\nu]
\end{eqnarray*}
is the preconditioned AOR iteration matrix.

In the following, if there is no special explanation then we always assume that
\begin{eqnarray*}
0 < \omega \le 1, \;\; 0 \le \gamma \le 1.
\end{eqnarray*}

We will propose four types of comparison theorems.
They contain a general comparison result, a strict comparison result and two
Stein-Rosenberg type comparison results.

We first give them as follows.

\medskip\noindent
{\bf Theorem A} ({\rm Stein-Rosenberg Type Theorem I})

{\it Let $A$ be an L-matrix.
Then one of the following mutually exclusive relations is valid:

\begin{itemize}
\item[(i)] $\rho({\mathscr L}^{(\nu)}_{\gamma,\omega}) \le \rho({\mathscr L}_{\gamma,\omega}) < 1$.

\item[(ii)] $\rho({\mathscr L}^{(\nu)}_{\gamma,\omega}) = \rho({\mathscr L}_{\gamma,\omega}) = 1$.

\item[(iii)] $\rho({\mathscr L}^{(\nu)}_{\gamma,\omega}) \ge \rho({\mathscr L}_{\gamma,\omega}) > 1$.
  \end{itemize}
}

\noindent
{\bf Theorem B}

{\it Let $A$ be a nonsingular M-matrix. Then
\begin{eqnarray*}
\rho({\mathscr L}^{(\nu)}_{\gamma,\omega}) \le \rho({\mathscr L}_{\gamma,\omega}) < 1.
\end{eqnarray*}
}

\noindent
{\bf Theorem C} ({\rm Stein-Rosenberg Type Theorem II})

{\it Let $A$ be an irreducible L-matrix.
Then one of the following mutually exclusive relations is valid:

\begin{itemize}
\item[(i)] $\rho({\mathscr L}^{(\nu)}_{\gamma,\omega}) < \rho({\mathscr L}_{\gamma,\omega}) < 1$.

\item[(ii)] $\rho({\mathscr L}^{(\nu)}_{\gamma,\omega}) = \rho({\mathscr L}_{\gamma,\omega}) = 1$.

\item[(iii)] $\rho({\mathscr L}^{(\nu)}_{\gamma,\omega}) > \rho({\mathscr L}_{\gamma,\omega}) > 1$.
 \end{itemize}
}

\noindent
{\bf Theorem D}

{\it Let $A$ be an irreducible nonsingular M-matrix. Then
\begin{eqnarray*}
\rho({\mathscr L}^{(\nu)}_{\gamma,\omega}) < \rho({\mathscr L}_{\gamma,\omega}) < 1.
\end{eqnarray*}
}

\begin{lemma}\label{lem3-1}
Suppose that ${\mathscr L}_{\gamma,\omega} \ge 0$, ${\mathscr L}^{(\nu)}_{\gamma,\omega} \ge 0$.
Assume that one of Theorems A, B, C and D is valid for $0 \le \gamma \le \omega \le 1$,
$\omega > 0$. Then it is valid for $0 < \omega \le 1$ and
$0 \le \gamma \le 1$.
\end{lemma}

\begin{proof}
Assume that Theorem C is valid for $0 \le \gamma \le \omega \le 1$,
$\omega > 0$.

We just need to prove that it is also valid for $0 < \omega < \gamma \le 1$. From (\ref{eqn1.7}), it is easy to prpve that
\begin{eqnarray*}
 \rho({\mathscr L}_{\gamma,\omega}) = 1 - \frac{\omega}{\gamma} + \frac{\omega}{\gamma}\rho({\mathscr L}_{\gamma}),
\; \;
\rho({\mathscr L}^{(\nu)}_{\gamma,\omega}) = 1 - \frac{\omega}{\gamma} + \frac{\omega}{\gamma}\rho({\mathscr L}^{(\nu)}_{\gamma}).
 \end{eqnarray*}
Clearly,
 \begin{eqnarray*}
  \rho({\mathscr L}_{\gamma,\omega}) \left\{
  \begin{array}{l}
 < 1\\
 = 1\\
 > 1
  \end{array}\right.
  \; \Longleftrightarrow \;
   \rho({\mathscr L}_{\gamma}) \left\{\begin{array}{l}
 < 1\\
 = 1\\
 > 1
 \end{array}\right.
  \end{eqnarray*}
  and
  \begin{eqnarray*}
  \rho({\mathscr L}^{(\nu)}_{\gamma,\omega}) \left\{\begin{array}{l}
 < 1\\
 = 1\\
 > 1
  \end{array}\right. \; \Longleftrightarrow \;
   \rho({\mathscr L}^{(\nu)}_{\gamma}) \left\{\begin{array}{l}
 < 1\\
 = 1\\
 > 1.
 \end{array}\right.
   \end{eqnarray*}
Since Theorem C is valid for $\gamma \le \omega$, then
we have that
 \begin{eqnarray*}
   \rho({\mathscr L}^{(\nu)}_{\gamma}) \left\{\begin{array}{lcl}
 < \rho({\mathscr L}_{\gamma}) &  \hbox{if}  & \rho({\mathscr L}_{\gamma}) < 1,\\
 = \rho({\mathscr L}_{\gamma}) & \hbox{if} & \rho({\mathscr L}_{\gamma}) = 1, \\
 >\rho({\mathscr L}_{\gamma}) & \hbox{if} & \rho({\mathscr L}_{\gamma}) > 1,
 \end{array}\right.
  \end{eqnarray*}
so that
 \begin{eqnarray*}
  \rho({\mathscr L}^{(\nu)}_{\gamma,\omega}) \left\{\begin{array}{lcl}
 < \rho({\mathscr L}_{\gamma,\omega}) &  \hbox{if}  & \rho({\mathscr L}_{\gamma,\omega}) < 1,\\
 = \rho({\mathscr L}_{\gamma,\omega}) & \hbox{if} & \rho({\mathscr L}_{\gamma,\omega}) = 1, \\
> \rho({\mathscr L}_{\gamma,\omega}) & \hbox{if} & \rho({\mathscr L}_{\gamma,\omega}) > 1.
 \end{array}\right.
  \end{eqnarray*}

When one of Theorems A, B and D is valid for $0 \le \gamma \le \omega \le 1$,
$\omega > 0$, the proof is completely same.

The proof is completed.  \end{proof}

In the following, we will construct some $Q_\nu$ ($P_\nu$) to make the above four theorems hold.
For simplicity, when we provide the conditions for the establishment of Theorems A, B, C and D
we always assume that $A$ is an L-matrix, a nonsingular M-matrix, an irreducible L-matrix and
an irreducible nonsingular M-matrix, respectively. We will not elaborate on this point one by one below.

\subsection{General preconditioners}

In \cite{WS09} we have proposed some general preconditioners. A class of general constructions of $Q$ is given by
\begin{eqnarray*}
Q_1 = (q^{(1)}_{i,j})
 \end{eqnarray*}
with
\begin{eqnarray*}
q^{(1)}_{i,j}
 \left\{\begin{array}{ll}
 = 0, & i = j = 1, \cdots, n, \\
 \ge 0, & i,j = 1, \cdots, n, \; i\not = j,
\end{array}\right.
 \;\; \hbox{and} \;\;
  \sum\limits_{{i,j = 1}\atop{i\not = j}}^n q^{(1)}_{i,j} \not = 0.
\end{eqnarray*}
Some comparison theorems have been proved.

By direct operation we have
\begin{eqnarray} \label{eqs3.2}
\left. \begin{array}{l}
a^{(1)}_{i,j} = a_{i,j} + q^{(1)}_{i,j} + \sum\limits_{{k = 1}\atop{k\not = i,j}}^n
q^{(1)}_{i,k}a_{k,j}, \; i,j = 1,\cdots, n, i\not = j, \\
a^{(1)}_{i,i} = 1 + \sum\limits_{{k = 1}\atop{k\not = i}}^n
q^{(1)}_{i,k}a_{k,i}, \; i = 1,\cdots, n.
\end{array}
\right.
 \end{eqnarray}

We define several decompositions as
\begin{eqnarray*}
Q_1 = Q^{(l)} + Q^{(u)}, \;
Q^{(l)}U = E_1+F_1+G_1, \;
Q^{(u)}L = E_2+F_2+G_2,
\end{eqnarray*}
where $E_1$ and $E_2$ are diagonal matrices, $Q^{(l)}$, $F_1$ and $F_2$
are strictly lower triangular matrices, while $Q^{(u)}$, $G_1$ and $G_2$
are strictly upper triangular matrices. Then the
three matrices in (\ref{eqs3.1}) are given by
\begin{eqnarray*}
D_1 & = & I - E_1 - E_2,\\
L_1 & = & L + F_1 + F_2 + Q^{(l)}L - Q^{(l)},  \\
U_1 & = & U + G_1 + G_2 + Q^{(u)}U - Q^{(u)}.
\end{eqnarray*}

Similar to the proof of \cite[Theorem 2.6]{WS09}, we prove
a lemma.

\begin{lemma}\label{lem3-2}
\begin{itemize}
\item[(i)]
\begin{eqnarray*}
&& P_1N_{\gamma,\omega} - N^{(1)}_{\gamma,\omega}\\
 && = P_1M_{\gamma,\omega} - M^{(1)}_{\gamma,\omega}\\
 && = \frac{1}{\omega}[E_1 + E_2 + \gamma(F_1 + F_2) + (1-\gamma)Q^{(l)} + \omega Q^{(u)}M_{\gamma,\omega}].
\end{eqnarray*}

\item[(ii)] Assume that $\lambda$ is an eigenvalue of ${\mathscr L}_{\gamma,\omega}$ and $x \not = 0$
is its associated eigenvector. Then
\begin{eqnarray*}
\left. \begin{array}{l}
 {\mathscr L}^{(1)}_{\gamma,\omega}x - \lambda x \\
 = (\lambda-1)[M^{(1)}_{\gamma,\omega}]^{-1}[E_1+E_2
 +\gamma(F_1+F_2) +(1-\gamma)Q^{(l)} + \omega Q^{(u)}M_{\gamma,\omega}]x.
\end{array}\right.
\end{eqnarray*}
\end{itemize}
\end{lemma}

\begin{proof}
Since
\begin{eqnarray*}
P_1A = M^{(1)}_{\gamma,\omega} - N^{(1)}_{\gamma,\omega} = P_1M_{\gamma,\omega} - P_1N_{\gamma,\omega},
  \end{eqnarray*}
then by direct operation we have
\begin{eqnarray*}
 P_1N_{\gamma,\omega} - N^{(1)}_{\gamma,\omega}
 & = & P_1M_{\gamma,\omega} - M^{(1)}_{\gamma,\omega} \\
& = & \frac{1}{\omega}[(I+Q^{(l)}+Q^{(u)})(I - \gamma L) - (D_1 - \gamma L_1)] \\
& = & \frac{1}{\omega}[E_1 + E_2 + \gamma(F_1 + F_2) + (1-\gamma)Q^{(l)} + Q^{(u)}(I - \gamma L)],
\end{eqnarray*}
which shows ($i$).

Similar to (\ref{eqs2.3}), we can get
\begin{eqnarray*}
{\mathscr L}^{(1)}_{\gamma,\omega}x - \lambda x
 = (\lambda - 1)[M^{(1)}_{\gamma,\omega}]^{-1}(P_1M_{\gamma,\omega} - M^{(1)}_{\gamma,\omega})x.
\end{eqnarray*}
By ($i$), we derive ($ii$).  \end{proof}

Let
\begin{eqnarray*}
\Delta^{(1)}(\gamma) & = & (E_1 + E_2) + \gamma F_1 + \gamma F_2 + \gamma Q^{(u)}U + (1-\gamma)Q \\
& = & \Delta_{11} + \gamma\Delta_{12} + \gamma\Delta_{13} + \gamma\Delta_{14} + (1-\gamma)Q, \nonumber
\end{eqnarray*}
where $\Delta_{11} = E_1 + E_2$, $\Delta_{12} = F_1$, $\Delta_{13} = F_2$, $\Delta_{14} = Q^{(u)}U$.
Denote
\begin{eqnarray*}
\Delta^{(1)}(\gamma) = (\delta_{i,j}^{(1)}(\gamma)), \; \Delta_k = (\delta^{(1k)}_{i,j}), \; k = 1,2,3,4.
\end{eqnarray*}
By direct operation we can obtain that
\begin{eqnarray*}
\delta^{(11)}_{i,j} & = & \left\{\begin{array}{ll}
 -\sum\limits_{{k = 1}\atop{k\not = i}}^n q^{(1)}_{i,k}a_{k,i}, & i = j = 1, \cdots, n,\\
 0, & otherwise,
\end{array}\right. \\
\delta^{(12)}_{i,j} & = & \left\{\begin{array}{ll}
 -\sum\limits_{k = 1}^{j-1} q^{(1)}_{i,k}a_{k,j}, & i = 3,\cdots,n, j = 2,\cdots, i-1,\\
0, & otherwise,
\end{array}\right. \\
\delta^{(13)}_{i,j} & = & \left\{\begin{array}{ll}
-\sum\limits_{k = i+1}^n q^{(1)}_{i,k}a_{k,j}, & i = 2,\cdots,n-1, j = 1,\cdots,i-1,\\
0, & otherwise,
\end{array}\right. \\
\delta^{(14)}_{i,j} & = & \left\{\begin{array}{ll}
 -\sum\limits_{k = i+1}^{j-1} q^{(1)}_{i,k}a_{k,j}, & i = 1,\cdots,n-2, j = i+2,\cdots,n,\\
0, & otherwise,
\end{array}\right.
\end{eqnarray*}
so that
\begin{eqnarray*}
\delta_{i,j}^{(1)}(\gamma) & = & \delta^{(11)}_{i,j} + \gamma\delta^{(12)}_{i,j}
 + \gamma\delta^{(13)}_{i,j} + \gamma\delta^{(14)}_{i,j} + (1-\gamma)q^{(1)}_{i,j}\\
 & = & \left\{\begin{array}{ll}
 -\sum\limits_{{k = 1}\atop{k\not = i}}^n q^{(1)}_{i,k}a_{k,i}, & i = j = 1, \cdots, n;\\
  (1-\gamma)q^{(1)}_{i,j} - \gamma\sum\limits_{k = i+1}^{j-1} q^{(1)}_{i,k}a_{k,j}, &
      \begin{array}{l} i = 1,\cdots,n-1, \\
      j = i+1,\cdots,n;
      \end{array}\\
  (1-\gamma)q^{(1)}_{i,j} - \gamma\sum\limits_{{1 \le k \le j-1}\atop{i+1 \le k \le n}}q^{(1)}_{i,k}a_{k,j}, &
   \begin{array}{l} i = 2,\cdots,n, \\
   j = 1,\cdots, i-1
   \end{array}
\end{array}\right.
\end{eqnarray*}
and
\begin{eqnarray} \label{eqs3.3}
\delta_{i,j}^{(1)}(1) = \left\{\begin{array}{ll}
-\sum\limits_{{1 \le k \le j-1}\atop{i+1 \le k \le n}}q^{(1)}_{i,k}a_{k,j}, &
  \begin{array}{l} i = 1,\cdots,n, j = 1,\cdots, i, \\
  (i,j)\not = (n,1); \end{array}\\
 -\sum\limits_{k = i+1}^{j-1} q^{(1)}_{i,k}a_{k,j}, &
  i = 1,\cdots,n-2, j = i+2,\cdots,n;\\
  0, & i = 1,\cdots,n-1, j = i+1;\\
  0, & i = n, j = 1,
\end{array}\right.
\end{eqnarray}
where the sum is taken to be zero when the upper limit is less than the lower limit.

\begin{lemma}\label{lem3-3}
Let $A$ be an L-matrix.

\begin{itemize}
\item[(i)] Let $0 \le \gamma < 1$.

\begin{itemize}
\item[($i_1$)] If $q^{(1)}_{i,j} > 0$ for some $i \not = j$, then $\delta_{i,j}^{(1)}(\gamma) > 0$.

\item[($i_2$)] There exist $i,j\in\{1, \cdots, n\}$, $i\not = j$, such that $\delta_{i,j}^{(1)}(\gamma) > 0$.
\end{itemize}

\item[(ii)] Let $\gamma = 1$.

\begin{itemize}
\item[($ii_1$)]
Suppose that there exist $i,j\in\{1, \cdots, n\}$ such that
$q^{(1)}_{i,j}a_{j,i} < 0$. Then $\delta_{i,i}^{(1)}(1) > 0$.

\item[($ii_2$)]
Suppose that there exist $i\in\{1, \cdots, n-1\}$ and $j\in\{1, \cdots, i\}$ such that
$q^{(1)}_{i,n}a_{n,j} < 0$. Then $\delta_{i,j}^{(1)}(1) > 0$.

\item[($ii_3$)]
Suppose that there exist $i\in\{1, \cdots, n-1\}$ and $j\in\{i+1, \cdots, n\}$ such that
$q^{(1)}_{i,j}a_{j,1} < 0$. Then $\delta_{i,1}^{(1)}(1) > 0$.

\item[($ii_4$)]
Suppose that there exist $i\in\{1, \cdots, n\}$ and $j\in\{1, \cdots, n-1\}$ such that
$q^{(1)}_{i,j}a_{j,j+1} < 0$. Then $\delta_{i,j+1}^{(1)}(1) > 0$.

\item[($ii_5$)]
Suppose that there exist $i\in\{2, \cdots, n\}$ and $j\in\{2, \cdots, i\}$ such that
$q^{(1)}_{i,1}a_{1,j} < 0$. Then $\delta_{i,j}^{(1)}(1) > 0$.

In addition, suppose that $A$ is irreducible.

\item[($ii_6$)] If $q^{(1)}_{i,i+1} > 0$ for some $i\in\{1, \cdots, n-1\}$, then there exists
$j\in\{1, \cdots, n\}\setminus\{i+1\}$ such that $\delta_{i,j}^{(1)}(1) > 0$.

\item[($ii_7$)] If $q^{(1)}_{n,1} > 0$, then there exists
$j\in\{2, \cdots, n\}$ such that $\delta_{n,j}^{(1)}(1) > 0$.

\item[($ii_8$)]
If $a_{n,1} < 0$ and $a_{k,k+1} < 0$, $k = 1, \cdots, n-1$,
then there exist $i,j\in\{1, \cdots, n\}$ such that $\delta_{i,j}^{(1)}(1) > 0$.
\end{itemize}\end{itemize}
\end{lemma}

\begin{proof}
Since $A$ is an L-matrix, then $\delta^{(1k)}_{i,j} \ge 0$, $k = 1,2,3,4$, and
$\delta_{i,j}^{(1)}(\gamma) \ge 0$.

Assume that $\gamma < 1$. If $q^{(1)}_{i,j} > 0$, $i\not = j$,
then $\delta_{i,j}^{(1)}(\gamma) \ge (1-\gamma)q^{(1)}_{i,j} > 0$,
i.e., ($i_1$) holds.

By the definition of $Q_1$, there exist $i,j\in\{1, \cdots, n\}$ and $i\not = j$,
such that $q^{(1)}_{i,j} > 0$, it follows by ($i_1$) that ($i_2$) holds.

Now we prove ($ii$). From
\begin{eqnarray*}
\delta_{i,i}^{(1)}(1) = -\sum\limits_{{k = 1}\atop{k\not = i}}^nq^{(1)}_{i,k}a_{k,i}, \;
i = 1, \cdots, n,
\end{eqnarray*}
($ii_1$) is obvious.

By (\ref{eqs3.3}), when $i = 1,\cdots,n-1$, then we have
$\delta_{i,j}^{(1)}(1) \ge -q^{(1)}_{i,n}a_{n,j}$ for $j = 1,\cdots, i$,
which implies ($ii_2$), while
$\delta_{i,1}^{(1)}(1) \ge -q^{(1)}_{i,j}a_{j,1}$, for $j = i+1,\cdots, n$,
which implies ($ii_3$).

Similarly, when $i = 2,\cdots, n$, $j = 1,\cdots, i-1$
and $i = 1,\cdots, n-2$, $j = i+1,\cdots, n-1$, we get
$\delta_{i,j+1}^{(1)}(1) \ge -q^{(1)}_{i,j}a_{j,j+1}$,
which implies ($ii_4$).

While, when $i = 2,\cdots,n$,
$\delta_{i,j}^{(1)}(1) \ge -q^{(1)}_{i,1}a_{1,j}$, for $j = 2,\cdots, i$,
which implies ($ii_5$).

Assume that $A$ is irreducible.

By the irreducibility of $A$, we have $\sum_{j = 1,j\not = i+1}^na_{i+1,j} < 0$.
If $q^{(1)}_{i,i+1} > 0$ then
\begin{eqnarray*}
\sum\limits_{{j = 1}\atop{j\not = i+1}}^n\delta_{i,j}^{(1)}(1)
& = & -\sum\limits_{j = 1}^{i}\sum\limits_{{1 \le k \le j-1}\atop{i+1 \le k \le n}} q^{(1)}_{i,k}a_{k,j}
- \sum\limits_{j = i+2}^{n}\sum\limits_{k = i+1}^{j-1} q^{(1)}_{i,k}a_{k,j} \\
& \ge & -\sum\limits_{j = 1}^{i}q^{(1)}_{i,i+1}a_{i+1,j} - \sum\limits_{j = i+2}^{n}q^{(1)}_{i,i+1}a_{i+1,j}\\
& = & -q^{(1)}_{i,i+1}\sum\limits_{{j = 1}\atop{j\not = i+1}}^{n}a_{i+1,j}\\
& > & 0,
\end{eqnarray*}
which implies that there exists $j\in\{1, \cdots, n\}\setminus\{i+1\}$ such that
$\delta_{i,j}^{(1)}(1) > 0$. This proves ($ii_6$).

Similarly, we have
$\sum_{j = 2}^na_{1,j} < 0$. If $q^{(1)}_{n,1} > 0$ then
\begin{eqnarray*}
\sum\limits_{j = 2}^n\delta_{n,j}^{(1)}(1)
 = -\sum\limits_{j = 2}^n\sum\limits_{k = 1}^{j-1} q^{(1)}_{n,k}a_{k,j}
 \ge -\sum\limits_{j = 2}^n q^{(1)}_{n,1}a_{1,j} = -q^{(1)}_{n,1}\sum\limits_{j = 2}^na_{1,j} > 0,
\end{eqnarray*}
which implies that there exists $j\in\{2, \cdots, n\}$ such that
$\delta_{n,j}^{(1)}(1) > 0$. This proves ($ii_7$).

At last, assume that $a_{n,1} > 0$ and $a_{k,k+1} > 0$, $k = 1, \cdots, n-1$.

For $i = 1, \cdots, n-1$, we have
\begin{eqnarray*}
\delta_{i,1}^{(1)}(1) = -\sum\limits_{k = i+1}^nq^{(1)}_{i,k}a_{k,1}
 \ge -q^{(1)}_{i,n}a_{n,1}.
\end{eqnarray*}

Similarly, we have that for $i = 1, \cdots, n-1$, $j = i+2, \cdots, n$,
\begin{eqnarray*}
\delta_{i,j}^{(1)}(1) = -\sum\limits_{k = i+1}^{j-1} q^{(1)}_{i,k}a_{k,j} \ge -q^{(1)}_{i,j-1}a_{j-1,j},
\end{eqnarray*}
and for $i = 2, \cdots, n$, $j = 2, \cdots, i$,
\begin{eqnarray*}
\delta_{i,j}^{(1)}(1) = -\sum\limits_{{1 \le k \le j-1}\atop{i+1 \le k \le n}}q^{(1)}_{i,k}a_{k,j}
 \ge -q^{(1)}_{i,j-1}a_{j-1,j}.
\end{eqnarray*}
Hence it gets that
\begin{eqnarray*}
\delta_{i,j}^{(1)}(1) \ge -q^{(1)}_{i,j-1}a_{j-1,j}, \; i = 1, \cdots, n, j = 2, \cdots, n, j\not = i+1.
\end{eqnarray*}

Denote $\eta = \min\{-a_{n,1}; \; -a_{k,k+1}: k = 1, \cdots, n-1\}$. Then $\eta > 0$.
Now, we obtain
\begin{eqnarray*}
\sum\limits_{i,j = 1}^n\delta_{i,j}^{(1)}(1)
& = & \sum\limits_{i = 1}^{n}\sum\limits_{j = 2}^n\delta_{i,j}^{(1)}(1)
   + \sum\limits_{i = 1}^{n}\delta_{i,1}^{(1)}(1) \\
& \ge & \sum\limits_{i = 1}^{n}\sum\limits_{{j = 2}\atop{j\not = i+1}}^n\delta_{i,j}^{(1)}(1)
   + \sum\limits_{i = 1}^{n-1}\delta_{i,1}^{(1)}(1) \\
& \ge & -\sum\limits_{i = 1}^{n}\sum\limits_{{j = 2}\atop{j\not = i+1}}^nq^{(1)}_{i,j-1}a_{j-1,j}
   - \sum\limits_{i = 1}^{n-1}q^{(1)}_{i,n}a_{n,1} \\
& \ge & \eta\sum\limits_{i = 1}^{n}\sum\limits_{{j = 2}\atop{j\not = i+1}}^nq^{(1)}_{i,j-1}
   + \eta\sum\limits_{i = 1}^{n-1}q^{(1)}_{i,n} \\
& = & \eta \left(\sum\limits_{i = 1}^{n}\sum\limits_{{j = 1}\atop{j\not = i}}^{n-1}q^{(1)}_{i,j}
   + \sum\limits_{i = 1}^{n-1}q^{(1)}_{i,n}\right) \\
& = & \eta\sum\limits_{{i,j = 1}\atop{i\not = j}}^{n}q^{(1)}_{i,j}\\
& > & 0,
\end{eqnarray*}
which implies that there exist $i,j\in\{1, \cdots, n\}$ such that
$\delta_{i,j}^{(1)}(1) > 0$, i.e., ($ii_8$) holds.
\end{proof}

We first give the condition for the establishment of the Stein-Rosenberg Type Theorem I.

\begin{theorem}\label{thm3.1}
Suppose that $P_1A$ is an L-matrix.
Then Theorem A is valid for $\nu = 1$.
\end{theorem}

\begin{proof}
Denote $\rho = \rho({\mathscr L}_{\gamma,\omega})$.

Since $A$ is an L-matrix, then it gets that $E_i \ge 0$, $F_i \ge 0$, $i = 1,2$,
$M^{-1}_{\gamma,\omega} > 0$ and
\begin{eqnarray*}
{\mathscr L}_{\gamma,\omega} = M^{-1}_{\gamma,\omega}N_{\gamma,\omega}
= (1-\omega)I + \omega(I-\gamma L)^{-1} [(1-\gamma)L + U] \ge 0,
\end{eqnarray*}
which shows that the splitting (\ref{eqn1.4}) is weak regular.
By Lemma \ref{lem1-1}, $\rho$ is an eigenvalue of ${\mathscr L}_{\gamma,\omega}$
with associated eigenvector $x > 0$.

Similarly, the AOR splitting
$P_1A = M^{(1)}_{\gamma,\omega} - N^{(1)}_{\gamma,\omega}$
is weak regular, i.e., $[M^{(1)}_{\gamma,\omega}]^{-1} > 0$ and ${\mathscr L}^{(1)}_{\gamma,\omega} \ge 0$.

By Lemma \ref{lem3-2} we obtain
\begin{eqnarray} \label{eqs3.4}
&& {\mathscr L}^{(1)}_{\gamma,\omega}x - \rho x \\
& = & (\rho-1)[M^{(1)}_{\gamma,\omega}]^{-1}[E_1+E_2
 +\gamma(F_1+F_2) +(1-\gamma)Q^{(l)} +
\omega Q^{(u)}M_{\gamma,\omega}]x. \nonumber
\end{eqnarray}

By Lemma \ref{lem3-1} we just need to consider the case when $\gamma \le \omega$.
In this case $N_{\gamma,\omega} \ge 0$.

When $\rho \ge 1$, then $M_{\gamma,\omega}x = N_{\gamma,\omega}x/\rho \ge 0$.
Since $[M^{(1)}_{\gamma,\omega}]^{-1} > 0$, $Q^{(l)} \ge 0$ and $Q^{(u)} \ge 0$,
then from (\ref{eqs3.4}) it derives that ${\mathscr L}^{(1)}_{\gamma,\omega}x \ge \rho x$.
It follows by Lemma \ref{lem1-2} that $\rho({\mathscr L}^{(1)}_{\gamma,\omega}) \ge \rho$.

Assume that $\rho \le 1$.

When $A$ is irreducible, by Lemma \ref{lem1-8}, $\rho > 0$ and we can choose $x \gg 0$.
Since $M_{\gamma,\omega}x = N_{\gamma,\omega}x/\rho \ge 0$, then from (\ref{eqs3.4})
it derives that ${\mathscr L}^{(1)}_{\gamma,\omega}x \le \rho x$.
It follows by Lemma \ref{lem1-2} that $\rho({\mathscr L}^{(1)}_{\gamma,\omega}) \le \rho$.

If $A$ is reducible, then definite $\check{A} = (\check{a}_{i,j})$ with
\begin{eqnarray*}
\check{a}_{i,j} = \left\{
\begin{array}{lll}
0, & {\rm if} & a_{i,j}\neq 0,\\
1, & {\rm if} & a_{i,j} = 0,
\end{array}
\quad i,j = 1, \cdots, n.
\right.
\end{eqnarray*}
Let $A(\epsilon) = A - \epsilon\check{A}$ with $\epsilon > 0$.
Then $A(\epsilon)$ is an irreducible L-matrix. From
$P_1A(\epsilon) = P_1A - \epsilon P_1\check{A}$, it is easy to see that
$P_1A(\epsilon)$ is an L-matrix for sufficient small $\epsilon$,
since the matrix $P_1A$ is an L-matrix and $P_1\check{A} \ge 0$.
Denote the AOR iteration matrices corresponding to $A(\epsilon)$ and $P_1A(\epsilon)$
by ${\mathscr L}_{\gamma,\omega}(\epsilon)$ and ${\mathscr L}^{(1)}_{\gamma,\omega}(\epsilon)$,
respectively. By the proof above we have
 \begin{eqnarray*}
\rho({\mathscr L}^{(1)}_{\gamma,\omega}) = \lim_{\epsilon \rightarrow
0^{+}} \rho({\mathscr L}^{(1)}_{\gamma,\omega}(\epsilon))
    \le \lim_{\epsilon \rightarrow 0^{+}}\rho({\mathscr L}_{\gamma,\omega}(\epsilon))
    = \rho.
 \end{eqnarray*}

Now, we have proved that either $\rho({\mathscr L}^{(\nu)}_{\gamma,\omega}) \le \rho({\mathscr L}_{\gamma,\omega}) \le 1$ or
$\rho({\mathscr L}^{(\nu)}_{\gamma,\omega}) \ge \rho({\mathscr L}_{\gamma,\omega}) \ge 1$,
which implies that one of the three mutually exclusive relations ($i$), ($ii$) and ($iii$)
of Theorem A holds.

The proof is completed.  \end{proof}

This result is consistent with \cite[Theorem 2.6]{WS09}.

\begin{theorem}\label{thm3.2}
Suppose that $P_1A$ is a Z-matrix.
Then Theorem B is valid for $\nu = 1$.
\end{theorem}

\begin{proof}
Since $A$ is a nonsingular M-matrix, then the splitting (\ref{eqn1.4}) is weak regular.
By Lemma \ref{lem1-5}, the AOR method is convergent, i.e., $\rho({\mathscr L}_{\gamma,\omega}) < 1$.

On the other hand, by Lemma \ref{lem1-10}, $P_1A$ is a nonsingular M-matrix
so that it is an L-matrix.

Now, it follows by Theorem \ref{thm3.1} that Theorem B is valid.
 \end{proof}

Next, we give the Stein-Rosenberg Type Theorem II.

\begin{theorem}\label{thm3.3}
Suppose that $P_1A$ is an L-matrix.
Then Theorem C is valid for $\nu = 1$,
provided one of the following conditions is satisfied:

\begin{itemize}
\item[(i)] $0 \le \gamma < 1$ and $P_1A$ is irreducible.

\item[(ii)] $\gamma = 1$ and $P_1A$ is irreducible. One of the following conditions holds:

\begin{itemize}
\item[($ii_1$)]
There exist $i,j\in\{1, \cdots, n\}$ such that $\delta_{i,j}^{(1)}(1) > 0$.

\item[($ii_2$)]
$q^{(1)}_{n,1} > 0$.

\item[($ii_3$)]
There exists $k\in\{1, \cdots, n-1\}$ such that $q^{(1)}_{k,k+1} > 0$.

\item[($ii_4$)]
There exist $i,j\in\{1, \cdots, n\}$ such that $q^{(1)}_{i,j}a_{j,i} < 0$.

\item[($ii_5$)]
There exist $i\in\{1, \cdots, n-1\}$ and $j\in\{1, \cdots, i\}$ such that
$q^{(1)}_{i,n}a_{n,j} < 0$.

\item[($ii_6$)]
There exist $i\in\{1, \cdots, n-1\}$ and $j\in\{i+1, \cdots, n\}$ such that
$q^{(1)}_{i,j}a_{j,1} < 0$.

\item[($ii_7$)]
There exist $i\in\{1, \cdots, n\}$ and $j\in\{1, \cdots, n-1\}$ such that
$q^{(1)}_{i,j}a_{j,j+1} < 0$.

\item[($ii_8$)]
There exist $i\in\{2, \cdots, n\}$ and $j\in\{2, \cdots, i\}$ such that
$q^{(1)}_{i,1}a_{1,j} < 0$.

\item[($ii_9$)]
$a_{n,1} < 0$, $a_{k,k+1} < 0$, $k = 1, \cdots, n-1$.
\end{itemize}

\item[(iii)] $0 \le \gamma < 1$ and for each $i\in\{1, \cdots, n-1\}$
there exists $j(i)\in\{1, \cdots, n\}$ such that $q^{(1)}_{i,j(i)} > 0$.

\item[(iv)] $\gamma = 1$ and for each $i\in\{2, \cdots, n-1\}$
one of the following conditions holds:

\begin{itemize}
\item[($iv_1$)] There exists $j(i)\in\{1, \cdots, n\}$ such that
$\delta_{i,j(i)}^{(1)}(1) > 0$.

\item[($iv_2$)] $q^{(1)}_{i,i+1} > 0$.

\item[($iv_3$)] There exists $j_i\in\{1, \cdots, n\}$ such that $q^{(1)}_{i,j_i}a_{j_i,i} < 0$.

\item[($iv_4$)] There exists $j_i\in\{1, \cdots, i\}$ such that $q^{(1)}_{i,n}a_{n,j_i} < 0$.

\item[($iv_5$)] There exists $j_i\in\{i+1, \cdots, n\}$ such that $q^{(1)}_{i,j_i}a_{j_i,1} < 0$.

\item[($iv_6$)] There exists
$j_i\in\{1, \cdots, n-1\}$ such that $q^{(1)}_{i,j_i}a_{j_i,j_i+1} < 0$.

\item[($iv_7$)] There exists $j_i\in\{2, \cdots, i\}$ such that $q^{(1)}_{i,1}a_{1,j_i} < 0$.
\end{itemize}

At the same time, one of the following conditions also holds:

\begin{itemize}
\item[($iv^a$)]
There exist $j\in\{2, \cdots, n\}$ and $k\in\{1, \cdots, j-1\}$ such that
$q^{(1)}_{n,k}a_{k,j} < 0$.

\item[($iv^b$)] $q^{(1)}_{n,1} > 0$.

\item[($iv^c$)] There exists $j\in\{2, \cdots, n-1\}$ such that
\begin{eqnarray}\label{eqs3.5}
a_{n,j} + q^{(1)}_{n,j} + \sum\limits_{{k = 1}\atop{k\not = j}}^{n-1}q^{(1)}_{n,k}a_{k,j} < 0.
\end{eqnarray}

\item[($iv^d$)] One of the conditions ($iv_1$)-($iv_6$) holds for $i=1$ and
\begin{eqnarray}\label{eqs3.6}
a_{n,1} + q^{(1)}_{n,1} + \sum\limits_{k = 2}^{n-1}q^{(1)}_{n,k}a_{k,1} < 0.
\end{eqnarray}

\item[($iv^e$)] One of the conditions ($iv_1$)-($iv_6$) holds for $i=1$ and $a_{n,1} < 0$.
\end{itemize}\end{itemize}
\end{theorem}

\begin{proof}
Denote $\rho = \rho({\mathscr L}_{\gamma,\omega})$.
Assume that $x > 0$ is its associated eigenvector.

Consider two splittings of $P_1A$ given by
\begin{eqnarray}\label{eqs3.7}
P_1A = M^{(1)}_{\gamma,\omega} - N^{(1)}_{\gamma,\omega} = P_1M_{\gamma,\omega} - P_1N_{\gamma,\omega}.
 \end{eqnarray}
Since $A$ and $P_1A$ are L-matrices, then the splittings are respectively weak regular and nonnegative,
so that $[M^{(1)}_{\gamma,\omega}]^{-1} > 0$.

By Lemma \ref{lem3-1} we just need to consider the case when $\gamma \le \omega$.
By Lemma \ref{lem1-8} it follows that $\rho > 0$ and $x\gg 0$.

Furthermore, we have $1/(\omega - \gamma + \gamma\rho) > 0$.
From $N_{\gamma,\omega}x = \rho M_{\gamma,\omega}x$, it gets that
\begin{eqnarray*}
Lx = \frac{1}{\omega - \gamma + \gamma\rho}[(\omega + \rho -1)I -\omega U]x,
\end{eqnarray*}
so that
\begin{eqnarray*}
M_{\gamma,\omega}x = \frac{1}{\omega - \gamma + \gamma\rho}[(1-\gamma)I + \gamma U]x.
\end{eqnarray*}

By Lemma \ref{lem3-2}, we obtain
\begin{eqnarray}\label{eqs3.8} \nonumber
 (P_1N_{\gamma,\omega} - N^{(1)}_{\gamma,\omega})x
& = & \frac{1}{\omega} \left\{E_1+E_2
+\gamma(F_1+F_2) +(1-\gamma)Q^{(l)} \right.\\
 && \left. + \frac{\omega}{\omega - \gamma + \gamma\rho}Q^{(u)}[(1-\gamma)I + \gamma U]\right\}x\\ \nonumber
 & = & \Phi(\gamma,\omega)x,
\end{eqnarray}
where
\begin{eqnarray*}
\Phi(\gamma,\omega)
& = & \frac{1}{\omega}(E_1 + E_2 + \gamma F_1+ \gamma F_2) \\
&& + \frac{\gamma}{\omega - \gamma + \gamma\rho}Q^{(u)}U
+ (1-\gamma)\left[\frac{1}{\omega} Q^{(l)} + \frac{1}{\omega - \gamma + \gamma\rho}Q^{(u)}\right].
\end{eqnarray*}

Clearly, $\Phi(\gamma,\omega) \ge 0$, $\Delta^{(1)}(\gamma) \ge 0$, $\Delta_{1k} \ge 0$, $k = 1,2,3,4$, and
the positions of the positive elements of the both matrices $\Phi(\gamma,\omega)$ and
$\Delta^{(1)}(\gamma)$ are completely same, since $\omega > 0$ and $\omega - \gamma + \gamma\rho > 0$.

Since $P_1A$ is an irreducible L-matrix, then, by Lemma \ref{lem1-8}, we can obtain
$\rho({\mathscr L}^{(1)}_{\gamma,\omega}) > 0$ and
$y^T(D_1 - \gamma L_1)^{-1} \gg 0$ whenever $y$ satisfies $y > 0$ and
$y^T{\mathscr L}^{(1)}_{\gamma,\omega} = \rho({\mathscr L}^{(1)}_{\gamma,\omega})y^T$.

We first prove ($i$).

In this case by ($i_2$) in Lemma \ref{lem3-3} it follows that
$\Delta^{(1)}(\gamma) > 0$ so that $\Phi(\gamma,\omega) > 0$.
From (\ref{eqs3.8}), we can get
$(P_1N_{\gamma,\omega} - N^{(1)}_{\gamma,\omega})x > 0$.
This shows that the condition ($ii$) of Lemma \ref{lem1-11} is satisfied.
The required result follows by Lemma \ref{lem1-11} directly.

We prove ($ii$).

Since $\gamma = 1$, then $\omega = 1$. In this case,
the AOR method reduces to the Gauss-Seidel method.
The equality (\ref{eqs3.8}) reduces to
$P_1N_{1,1} - N^{(1)}_{1,1}
 = E_1+E_2 + F_1 + F_2 + Q^{(u)}U/\rho = \Phi(1,1)$.

If one of the conditions ($ii_2$)-($ii_9$) is satisfied, then by ($ii$) of Lemma \ref{lem3-3}
it is easy to prove that there exist $i,j\in\{1, \cdots, n\}$ such that
$\delta_{i,j}^{(1)}(1) > 0$, which shows that ($ii_1$) is satisfied.

If ($ii_1$) is satisfied, then $\Phi(1,1) > 0$,
so that $(P_1N_{1,1} - N^{(1)}_{1,1})x > 0$.
This shows that the condition ($ii$) of Lemma \ref{lem1-11} is satisfied.

We prove ($iii$).
Let
\begin{eqnarray*}
M^{(1)}_{\gamma,\omega} = (m^{(1)}_{i,j}) =
\left(\begin{array}{ll}
\bar{M}_{1,1} & \bar{m}_{1,2}\\
\bar{m}_{2,1} & m^{(1)}_{n,n}
\end{array}\right), \; \bar{M}_{1,1}\in {\mathscr R}^{(n-1) \times (n-1)}.
\end{eqnarray*}
 Then $\bar{m}_{1,2} = 0$, $\bar{m}_{2,1} = (m^{(1)}_{n,1}  \cdots m^{(1)}_{n,n-1})$
and for $j= 1, \cdots, n-1$, $m^{(1)}_{j,j} > 0$,
\begin{eqnarray}\label{eqs3.88}
m^{(1)}_{n,j} = a_{n,j} + q^{(1)}_{n,j} + \sum\limits_{{k = 1}\atop{k\not = j}}^{n-1}q^{(1)}_{n,k}a_{k,j} \le 0.
\end{eqnarray}
Furthermore, we have
\begin{eqnarray*}
[M^{(1)}_{\gamma,\omega}]^{-1} = (\hat{m}_{i,j}) =
\left(\begin{array}{cc}
\bar{M}_{1,1}^{-1} & 0\\
-\frac{1}{m^{(1)}_{n,n}}\bar{m}_{2,1}\bar{M}_{1,1}^{-1} & \frac{1}{m^{(1)}_{n,n}}
\end{array}\right) > 0,
\end{eqnarray*}
where $\hat{m}_{k,k} > 0$, $\hat{m}_{n,k} \ge -m^{(1)}_{n,k}\hat{m}_{k,k}/m^{(1)}_{n,n}$,
$k=1,\cdots, n-1$.

By ($i_1$) in Lemma \ref{lem3-3}, for each $i\in\{1, \cdots, n-1\}$,
$\delta_{i,j(i)}^{(1)}(\gamma) > 0$.
Hence, in this case, the every row of $\Phi(\gamma,\omega)$ has positive elements except
the last row, so that the first to $(n-1)$th elements of $\Phi(\gamma,\omega)x$ are positive.
Since $[M^{(1)}_{\gamma,\omega}]^{-1} > 0$ and $\hat{m}_{k,k} > 0$ for $k=1,\cdots, n-1$,
then the first to $(n-1)$th elements of $[M^{(1)}_{\gamma,\omega}]^{-1}\Phi(\gamma,\omega)x$
are also positive.

Since $A$ is irreducible, then
there exists $j_n\in\{1, \cdots, n-1\}$ such that $a_{n,j_n} < 0$.

If $q^{(1)}_{n,j_n} > 0$ then
$\delta_{n,j_n}^{(1)}(\gamma) \ge (1-\gamma)q^{(1)}_{n,j_n} > 0$. This shows that the last row of $\Phi(\gamma,\omega)$
has positive elements, so that the last element of $\Phi(\gamma,\omega)x$ is positive.
From (\ref{eqs3.8}) we have proved that
$(P_1N_{\gamma,\omega} - N^{(1)}_{\gamma,\omega})x \gg 0$
and, therefore,
\begin{eqnarray}\label{eqs3.9}
[M^{(1)}_{\gamma,\omega}]^{-1}(P_1N_{\gamma,\omega} - N^{(1)}_{\gamma,\omega})x \gg 0.
\end{eqnarray}

When $q^{(1)}_{n,j_n} = 0$ then from (\ref{eqs3.88}) $m^{(1)}_{n,j_n} \le a_{n,j_n} < 0$, so that $\hat{m}_{n,j_n} > 0$.
Hence the last element of $[M^{(1)}_{\gamma,\omega}]^{-1}\Phi(\gamma,\omega)x$ is positive,
which shows that (\ref{eqs3.9}) holds.

Now, we have proved that the condition ($i$) of Lemma \ref{lem1-11} is satisfied for the splittings
given in (\ref{eqs3.7}). By Lemma \ref{lem1-11}
we can prove that Theorem C is valid.

At last, we prove ($iv$).

In this case, the AOR method reduces to the Gauss-Seidel method.

If one of ($iv_2$)-($iv_7$) holds, then it follows by ($ii_1$)-($ii_6$) in Lemma \ref{lem3-3}
that ($iv_1$) is satisfied.

When ($iv_1$) holds, then the every row of $\Phi(1,1)$ has positive elements except
the first and last rows, so that the second to $(n-1)$th elements of $\Phi(1,1)x$ and
$[M^{(1)}_{1,1}]^{-1}\Phi(1,1)x$ are positive.

If ($iv^a$) holds then $\delta_{n,j}^{(1)}(1) \ge -q^{(1)}_{n,k}a_{k,j} > 0$. And if ($iv^b$)
is satisfied, then by ($ii_7$) in Lemma \ref{lem3-3} there exists
$j\in\{2, \cdots, n\}$ such that $\delta_{n,j}^{(1)}(1) > 0$.
Hence, for these two cases the $n$th row of $\Phi(1,1)$ has positive elements.
This has proved that the second to $n$th rows of $\Phi(1,1)$ has positive elements,
so that the second to $n$th elements of $\Phi(1,1)x$ and
$[M^{(1)}_{1,1}]^{-1}\Phi(1,1)x$
are all positive.

If ($iv^c$) holds, then $m^{(1)}_{n,j} < 0$, so that $\hat{m}_{n,j} > 0$.
Hence the last element of $[M^{(1)}_{\gamma,\omega}]^{-1}\Phi(1,1)x$ is positive,
which shows that its second to $n$th elements are all positive.

When ($iv^e$) holds, if $q^{(1)}_{n,1} > 0$ then the proof is given above. If $q^{(1)}_{n,1} = 0$ then
\begin{eqnarray*}
a_{n,1} + q^{(1)}_{n,1} + \sum\limits_{k = 2}^{n-1}q^{(1)}_{n,k}a_{k,1} \le a_{n,1} < 0,
\end{eqnarray*}
which implies that ($iv^d$) is satisfied.

Now, we consider that ($iv^d$) holds. Just as the proof above,
the first row of $\Phi(1,1)$ has positive elements,
so that the first to $(n-1)$th elements of $\Phi(1,1)x$ and
$[M^{(1)}_{1,1}]^{-1}\Phi(1,1)x$ are positive. the inequality (\ref{eqs3.6}) shows that
 $m^{(1)}_{n,1} < 0$, so that $\hat{m}_{n,1} > 0$.
Hence the last element of $[M^{(1)}_{1,1}]^{-1}\Phi(1,1)x$ is positive,
and therefore, its second to $n$th elements are all positive.

We have proved that if one of ($iv_1$)-($iv_7$) and one of ($iv^a$)-($iv^e$) hold at the same time,
then the condition of Lemma \ref{lem1-12} is satisfied. By Lemma \ref{lem1-12}
we can prove that Theorem C is valid.

The proof is completed.  \end{proof}

\begin{theorem}\label{thm3.4}
Suppose that $P_1A$ is a Z-matrix.
Then Theorem D is valid for $\nu = 1$,
provided one of the following conditions is satisfied:

\begin{itemize}
\item[(i)] One of the conditions ($i$)-($iv$) of Theorem \ref{thm3.3} holds.

\item[(ii)] For $i = 2,\cdots,n$, $j = 1,\cdots,i-1$, $a_{i,j} \ge a^{(1)}_{i,j}$.
And one of the following conditions holds:

\begin{itemize}
\item[($i_1$)] There exists $i_0\in\{1,\cdots,n\}$ such that $a^{(1)}_{i_0,i_0} < 1$.

\item[($ii_2$)] $\gamma > 0$ and there exist $i_0\in\{2,\cdots,n\}$, $j_0\in\{1,\cdots,i_0-1\}$
such that $a_{i_0,j_0} > a^{(1)}_{i_0,j_0}$.
\end{itemize}\end{itemize}
\end{theorem}

\begin{proof}
By Lemma \ref{lem1-10}, $PA$ ia a nonsingular M-matrix. Hence
both AOR splittings $A = M_{\gamma,\omega} - N_{\gamma,\omega}$ and
$P_1A = M^{(1)}_{\gamma,\omega} - N^{(1)}_{\gamma,\omega}$ are weak regular.

Denote $\rho = \rho({\mathscr L}_{\gamma,\omega})$ and $x > 0$ being its associated eigenvector.
By Theorem \ref{thm3.2}, we have
$\rho({\mathscr L}^{(1)}_{\gamma,\omega}) \le \rho({\mathscr L}_{\gamma,\omega}) < 1$.

When ($i$) holds, the proof is completely same as the proof of
Theorem \ref{thm3.2}, by Theorem \ref{thm3.3} we can prove the required result.

Now, we prove ($ii$).

Since $A$ is an irreducible nonsingular M-matrix, then, by Lemma \ref{lem1-4}, $A^{-1}\gg 0$.
While by Lemma \ref{lem1-9}, $(P_1A)^{-1} > 0$. From $A^{-1} - (P_1A)^{-1} = (P_1A)^{-1}(P_1 - I)
= (P_1A)^{-1}Q_1 > 0$, it gets that $A^{-1} > (P_1A)^{-1}$.

By Lemma \ref{lem3-1} we just need to consider the case when $\gamma \le \omega$.
Then $N_{\gamma,\omega} > 0$ and $N^{(1)}_{\gamma,\omega} \ge 0$.
By Lemma \ref{lem1-8}, $\rho > 0$ and $x\gg 0$.
Then it is easy to prove that $M_{\gamma,\omega} > M^{(1)}_{\gamma,\omega}$, so that
$M_{\gamma,\omega}x > M^{(1)}_{\gamma,\omega}x$. Noticing that $Ax = (1/\rho - 1)N_{\gamma,\omega}x > 0$,
we have $M^{(1)}_{\gamma,\omega}x = P_1Ax + N^{(1)}_{\gamma,\omega}x > 0$.
Now we have proved that the condition of
Lemma \ref{lem1-13} is satisfied. By Lemma \ref{lem1-13} it follows that
$\rho({\mathscr L}^{(1)}_{\gamma,\omega}) < \rho({\mathscr L}_{\gamma,\omega}) < 1$.
\end{proof}

By the definitions of L-matrix and Z-matrix, the following two corollaries can be
derived from Theorems \ref{thm3.1} and \ref{thm3.2} directly.

\begin{corollary} \label{coro3-1}
Suppose that
$a^{(1)}_{i,i} > 0$, $a^{(1)}_{i,j} \le 0$, $i,j = 1, \cdots, n$, $i\not = j$.
Then Theorem A is valid for $\nu = 1$.
\end{corollary}

\begin{corollary} \label{coro3-2}
Suppose that
$a^{(1)}_{i,j} \le 0$, $i,j = 1,\cdots, n$, $i\not = j$.
Then Theorem B is valid for $\nu = 1$.
\end{corollary}

In all of the following, for the case when $A$ is irreducible,
the symbol ``$\lesssim$" (``$\gtrsim$") indicates ``$\le$" (``$\ge$")
if $A^{(\nu)}$ is irreducible even when it appears ``$=$",
otherwise it is ``$<$" (``$>$").

\begin{corollary} \label{coro3-3}
Suppose that $a^{(1)}_{i,i} > 0$, $a^{(1)}_{i,j} \le 0$,
$i,j = 1, \cdots, n$, $i\not = j$.
Then Theorem C is valid for $\nu = 1$,
provided one of the following conditions is satisfied:

\begin{itemize}
\item[(i)] For $i,j = 1, \cdots, n$, $i\not = j$, $a^{(1)}_{i,j} \lesssim 0$ whenever $a_{i,j} < 0$.
One of the conditions $0 \le \gamma < 1$ and ($ii_1$)-($ii_9$) whenever $\gamma=1$
in Theorem \ref{thm3.3} holds.

\item[(ii)] One of the conditions ($iii$) and ($iv$) of Theorem \ref{thm3.3} holds.
\end{itemize}
\end{corollary}

\begin{proof}
Clearly, the matrix $P_1A$ is an L-matrix. The condition $a^{(1)}_{i,j} \lesssim 0$
whenever $a_{i,j} < 0$ ensures that $P_1A$ is irreducible, since $A$ is
irreducible. This shows that the condition of Theorem \ref{thm3.3} is satisfied,
so that Theorem C is valid. \end{proof}

Similarly, the following corollary can be
derived from Theorem \ref{thm3.4} directly.

\begin{corollary} \label{coro3-4}
Suppose that
$a^{(1)}_{i,j} \le 0$, $i,j = 1,\cdots, n$, $i\not = j$.
Then Theorem D is valid for $\nu = 1$,
provided one of the conditions ($ii$) of Theorem \ref{thm3.4},
($i$) and ($ii$) of Corollary \ref{coro3-3} is satisfied.
\end{corollary}

Furthermore, noticing (\ref{eqs3.2}), from Corollaries \ref{coro3-1}-\ref{coro3-4},
we give the following corollaries.

\begin{corollary} \label{coro3-5}
Suppose that $q^{(1)}_{i,j} \le -a_{i,j}$, $i,j = 1,\cdots, n$, $i\not = j$, and
\begin{eqnarray}\label{eqs3.11}
1 + \sum\limits_{{k = 1}\atop{k\not = i}}^nq^{(1)}_{i,k}a_{k,i} > 0, \; i = 1,\cdots, n.
 \end{eqnarray}
Then Theorem A is valid for $\nu = 1$.
\end{corollary}

\begin{proof} For $i,j = 1,\cdots, n$, we have
\begin{eqnarray*}
a^{(1)}_{i,j} = a_{i,j} + q^{(1)}_{i,j} + \sum\limits_{{k = 1}\atop{k\not = i,j}}^n q^{(1)}_{i,k}a_{k,j}
 \le a_{i,j} + q^{(1)}_{i,j} \le 0, \; i\not = j,
\end{eqnarray*}
and
\begin{eqnarray*}
a^{(1)}_{i,i} = 1 + \sum\limits_{{k = 1}\atop{k\not = i}}^nq^{(1)}_{i,k}a_{k,i} > 0, \; i = j.
\end{eqnarray*}
This shows that the condition of Corollary \ref{coro3-1} is
satisfied. Therefore Theorem A is valid.  \end{proof}

By Corollary \ref{coro3-2} and the proof of Corollary \ref{coro3-5}
we obtain the following corollary directly.

\begin{corollary} \label{coro3-6}
Suppose that $q^{(1)}_{i,j} \le -a_{i,j}$, $i,j = 1,\cdots, n$, $i\not = j$.
Then Theorem B is valid for $\nu = 1$.
\end{corollary}

\begin{corollary} \label{coro3-7}
Suppose that the condition of Corollary \ref{coro3-5} is satisfied.
Then Theorem C is valid for $\nu = 1$,
provided one of the following conditions is satisfied:

\begin{itemize}
\item[(i)] For $i,j = 1, \cdots, n$, $i\not = j$, $q^{(1)}_{i,j} \lesssim -a_{i,j}$ whenever $a_{i,j} < 0$.
One of the conditions
$0\le\gamma<1$ and ($ii_1$)-($ii_9$) whenever $\gamma=1$ in Theorem \ref{thm3.3} holds.

\item[(ii)] One of the conditions ($iii$) and ($iv$) of Theorem \ref{thm3.3} holds,
where the inequality (\ref{eqs3.6}) can be replaced by $q^{(1)}_{n,j} < -a_{n,j}$.
\end{itemize}
\end{corollary}

\begin{proof} Since $A$ is irreducible, then there exists $j\in\{1, \cdots, n-1\}$ such that
$a_{n,j} < 0$. So we can choose $Q_1$ such that $q^{(1)}_{n,j} < -a_{n,j}$ in ($ii$).

From the proof of Corollary \ref{coro3-5} it is easy to prove that
the condition of Corollary \ref{coro3-3} is satisfied. Therefore Theorem C is valid.
 \end{proof}

Similarly, by Corollary \ref{coro3-4} we can prove the following corollary directly.

\begin{corollary} \label{coro3-8}
Suppose that $q^{(1)}_{i,j} \le -a_{i,j}$, $i,j = 1,\cdots, n$, $i\not = j$.
Then Theorem D is valid for $\nu = 1$,
provided one of the conditions ($ii$) of Theorem \ref{thm3.4},
($i$) and ($ii$) of Corollary \ref{coro3-7} is satisfied.
\end{corollary}

As a special case, in \cite{WS09} we propose $q^{(2)}_{i,j} = -\alpha_{i,j}a_{i,j}$ and get
 \begin{eqnarray*}
Q_2 = (-\alpha_{i,j}a_{i,j})
 \end{eqnarray*}
 with
\begin{eqnarray*}
 \alpha_{i,i} = 0, \; \alpha_{i,j} \ge 0, \; i, j = 1, \cdots, n, \; i\not = j,
  \;  \hbox{and} \;
\sum\limits_{{i,j = 1}\atop{i\not = j}}^n \alpha_{i,j}a_{i,j} \not = 0.
\end{eqnarray*}

Of course, when $a_{i,j} = 0$, the choice of $\alpha_{i,j}$ is meaningless.

In \cite{HN01}, two special preconditioners are proposed for the preconditioned Gauss-Seidel method,
where one is $\alpha_{i,j} = 1$, the other is $\alpha_{i,j} = 1+\alpha$ for $\alpha \ge 0$,
$i, j = 1, \cdots, n$, $i \not= j$.
In \cite{Li08,WZ11}, for the preconditioned Gauss-Seidel and AOR methods respectively, the authors consider the case when
$\alpha_{i,j} = \alpha \ge 0$ for $i > j$, $\alpha_{i,j} = \beta \ge 0$ for $i < j$, $i, j = 1, \cdots, n$,
with $\alpha + \beta \not = 0$.

Denote
\begin{eqnarray*}
\delta_{i,j}^{(2)}(\gamma) = \left\{\begin{array}{ll}
 \sum\limits_{{k = 1}\atop{k\not = i}}^n \alpha_{i,k}a_{i,k}a_{k,i}, &\;\; i = j = 1, \cdots, n;\\
 (\gamma-1)\alpha_{i,j}a_{i,j} + \gamma\sum\limits_{k = i+1}^{j-1} \alpha_{i,k}a_{i,k}a_{k,j}, &
  \begin{array}{l} i = 1,\cdots,n-1, \\
  j = i+1,\cdots,n;
  \end{array}\\
 (\gamma-1)\alpha_{i,j}a_{i,j} + \gamma\sum\limits_{{1 \le k \le j-1}\atop{i+1 \le k \le n}}\alpha_{i,k}a_{i,k}a_{k,j}, &
   \begin{array}{l} i = 2,\cdots,n,\\
    j = 1,\cdots, i-1
\end{array}\end{array}\right.
\end{eqnarray*}
and
\begin{eqnarray*}
\delta_{i,j}^{(2)}(1) = \left\{\begin{array}{ll}
 \sum\limits_{{1 \le k \le j-1}\atop{i+1 \le k \le n}}\alpha_{i,k}a_{i,k}a_{k,j}; &
  \begin{array}{l} i = 1,\cdots,n, j = 1,\cdots, i, \\
   (i, j) \not = (n, 1); \end{array}\\
 \sum\limits_{k = i+1}^{j-1} \alpha_{i,k}a_{i,k}a_{k,j}, &
 i = 1,\cdots,n-2, j = i+2,\cdots,n;\\
  0, & i = 1,\cdots,n-1, j = i+1;\\
 0, & i = n, j = 1.
\end{array}\right.
\end{eqnarray*}

Using Corollaries \ref{coro3-1}-\ref{coro3-4}, we prove corresponding comparison
theorems.

\begin{theorem}\label{thm3.5}
Suppose that $\sum_{k = 1,k\not = i}^n\alpha_{i,k}a_{i,k}a_{k,i} < 1$, $i = 1,\cdots,n$, and
\begin{eqnarray} \label{eqs3.12}
 (1 - \alpha_{i,j})a_{i,j} - \sum\limits_{{k = 1}\atop{k\not = i,j}}^n \alpha_{i,k}a_{i,k}a_{k,j}
 \le 0, \; i,j = 1,\cdots,n, \; i\not = j.
\end{eqnarray}
Then Theorem A is valid for $\nu = 2$.
\end{theorem}

\begin{proof} The inequality (\ref{eqs3.12}) shows that $a^{(2)}_{i,j} \le 0$,
$i,j = 1,\cdots,n$, $i\not = j$, and the inequality $\sum_{k = 1,k\not = i}^n\alpha_{i,k}a_{i,k}a_{k,i} < 1$ shows that
$a^{(2)}_{i,i} > 0$, $i = 1,\cdots,n$.

It has proved that the condition of Corollary \ref{coro3-1} is satisfied so that
Theorem A is valid.
  \end{proof}

\begin{theorem}\label{thm3.6}
Suppose that (\ref{eqs3.12}) holds.
Then Theorem B is valid for $\nu = 2$.
\end{theorem}

\begin{proof}
From the proof of Theorem \ref{thm3.5} it can prove that the condition of
Corollary \ref{coro3-2} is satisfied. Therefore Theorem B is valid.  \end{proof}

\begin{theorem}\label{thm3.7}
Suppose that $\sum_{k = 1,k\not = i}^n\alpha_{i,k}a_{i,k}a_{k,i} < 1$, $i = 1,\cdots,n$.
Then Theorem C is valid for $\nu = 2$,
provided one of the following conditions is satisfied:

\begin{itemize}
\item[(i)] $0 \le \gamma < 1$ and
\begin{eqnarray} \label{eqs3.14} \nonumber
&& (1 - \alpha_{i,j})a_{i,j} - \sum\limits_{{k = 1}\atop{k\not = i,j}}^n \alpha_{i,k}a_{i,k}a_{k,j}
 \lesssim 0 \;\; \hbox{whenever} \;\; a_{i,j} < 0,\\
&& i,j = 1, \cdots, n, \; i\not = j.
\end{eqnarray}

\item[(ii)] $\gamma = 1$, the inequality (\ref{eqs3.14}) holds and one of the following conditions holds:

\begin{itemize}
\item[($ii_1$)]
There exist $i,j\in\{1, \cdots, n\}$ such that $\delta_{i,j}^{(2)}(1) > 0$.

\item[($ii_2$)]
$a_{n,1} < 0$ and $\alpha_{n,1} > 0$.

\item[($ii_3$)]
There exists $k\in\{1, \cdots, n-1\}$ such that
$a_{k,k+1} < 0$ and $\alpha_{k,k+1} > 0$.

\item[($ii_4$)]
There exist $i,j\in\{1, \cdots, n\}$ such that $\alpha_{i,j}a_{i,j}a_{j,i} > 0$.

\item[($ii_5$)]
There exist $i\in\{1, \cdots, n-1\}$ and $j\in\{1, \cdots, i\}$ such that
$\alpha_{i,n}a_{i,n}a_{n,j}$ $ > 0$.

\item[($ii_6$)]
There exist $i\in\{1, \cdots, n-1\}$ and $j\in\{i+1, \cdots, n\}$ such that
$\alpha_{i,j}a_{i,j}a_{j,1} < 0$.

\item[($ii_7$)]
There exist $i\in\{1, \cdots, n\}$ and $j\in\{1, \cdots, n-1\}$ such that
$\alpha_{i,j}a_{i,j}a_{j,j+1}$ $ > 0$.

\item[($ii_8$)]
There exist $i\in\{2, \cdots, n\}$ and $j\in\{2, \cdots, i\}$ such that
$\alpha_{i,1}a_{i,1}a_{1,j} > 0$.

\item[($ii_9$)]
$a_{n,1} < 0$ and $a_{k,k+1} < 0$, $k = 1, \cdots, n-1$.
\end{itemize}

\item[(iii)] $0 \le \gamma < 1$, the inequality (\ref{eqs3.12}) holds and
for each $i\in\{1, \cdots, n-1\}$
there exists $j(i)\in\{1, \cdots, n\}$ such that $\alpha_{i,j(i)}a_{i,j(i)} < 0$.

\item[(iv)] $\gamma = 1$, the inequality (\ref{eqs3.12}) holds and for each $i\in\{2, \cdots, n-1\}$
one of the following conditions holds:

\begin{itemize}
\item[($iv_1$)] There exists $j(i)\in\{1, \cdots, n\}$ such that
$\delta_{i,j(i)}^{(2)}(1) > 0$.

\item[($iv_2$)] $a_{i,i+1} < 0$ and $\alpha_{i,i+1} > 0$.

\item[($iv_3$)] There exists $j_i\in\{1, \cdots, n\}$ such that $\alpha_{i,j_i}a_{i,j_i}a_{j_i,i} > 0$.

\item[($iv_4$)] There exists $j_i\in\{1, \cdots, i\}$ such that $\alpha_{i,n}a_{i,n}a_{n,j_i} > 0$.

\item[($iv_5$)] There exists $j_i\in\{i+1, \cdots, n\}$ such that $\alpha_{i,j_i}a_{i,j_i}a_{j_i,1} > 0$.

\item[($iv_6$)] There exists
$j_i\in\{1, \cdots, n-1\}$ such that $\alpha_{i,j_i}a_{i,j_i}a_{j_i,j_i+1} > 0$.

\item[($iv_7$)] There exists $j_i\in\{2, \cdots, i\}$ such that $\alpha_{i,1}a_{i,1}a_{1,j_i} > 0$.
\end{itemize}

At the same time, one of the following conditions also holds:

\begin{itemize}
\item[($iv^a$)]
There exist $j\in\{2, \cdots, n\}$ and $k\in\{1, \cdots, j-1\}$ such that
$\alpha_{n,k}a_{n,k}a_{k,j}$ $ > 0$.

\item[($iv^b$)] $a_{n,1} < 0$ and $\alpha_{n,1} > 0$.

\item[($iv^c$)] There exists $j\in\{2, \cdots, n-1\}$ such that
\begin{eqnarray}\label{eqs3.15}
(1 - \alpha_{n,j})a_{n,j} - \sum\limits_{{k = 1}\atop{k\not = j}}^{n-1}\alpha_{n,k}a_{n,k}a_{k,j} < 0.
\end{eqnarray}

\item[($iv^d$)] One of the conditions ($iv_1$)-($iv_6$) holds for $i=1$ and
\begin{eqnarray}\label{eqs3.16}
(1 - \alpha_{n,1})a_{n,1} - \sum\limits_{k = 2}^{n-1}\alpha_{n,k}a_{n,k}a_{k,1} < 0.
\end{eqnarray}

\item[($iv^e$)] One of the conditions ($iv_1$)-($iv_6$) holds for $i=1$ and $a_{n,1} < 0$.
\end{itemize}\end{itemize}
\end{theorem}

\begin{proof}
The inequality $\sum_{k = 1,k\not = i}^n\alpha_{i,k}a_{i,k}a_{k,i} < 1$ shows that $a^{(2)}_{i,i} > 0$,
$i = 1,\cdots,n$.
Now, $\delta_{i,j}^{(1)}(\gamma)$ reduces to $\delta_{i,j}^{(2)}(\gamma)$, (\ref{eqs3.5}) and
(\ref{eqs3.6}) reduce to (\ref{eqs3.15}) and (\ref{eqs3.16}), respectively.

For $i\not = j$, if $a_{i,j} = 0$ then
\begin{eqnarray*}
a^{(2)}_{i,j} = -\sum\limits_{{k = 1}\atop{k\not = i,j}}^n \alpha_{i,k}a_{i,k}a_{k,j}
 \le 0.
\end{eqnarray*}
When $a_{i,j} < 0$, the inequality (\ref{eqs3.14}) implies $a^{(2)}_{i,j} \lesssim 0$.

Now, we have proved that the condition of Corollary \ref{coro3-3} is satisfied.
Hence Theorem C is valid.
  \end{proof}

By Corollary \ref{coro3-4} it is easy to prove the following theorem.

\begin{theorem}\label{thm3.8}
Theorem D is valid for $\nu = 2$,
provided one of the following conditions is satisfied:

\begin{itemize}
\item[(i)] One of the conditions ($i$)-($iv$) of Theorem \ref{thm3.7} holds.

\item[(ii)] The inequality (\ref{eqs3.12}) holds. For $i = 2,\cdots,n$, $j = 1,\cdots,i-1$,
\begin{eqnarray*}
 \alpha_{i,j}a_{i,j} + \sum\limits_{{k = 1}\atop{k\not = i,j}}^n \alpha_{i,k}a_{i,k}a_{k,j} \ge 0.
\end{eqnarray*}
And one of the following conditions holds:

\begin{itemize}
\item[($ii_1$)] There exists $i_0\in\{1,\cdots,n\}$ such that
\begin{eqnarray*}
\sum\limits_{{k = 1}\atop{k\not = i_0}}^n \alpha_{i_0,k}a_{i_0,k}a_{k,i_0} > 0
 \end{eqnarray*}

\item[($ii_2$)] $\gamma > 0$ and there exist $i_0\in\{2,\cdots,n\}$, $j_0\in\{1,\cdots,i_0-1\}$ such that
 \begin{eqnarray*}
\alpha_{i_0,j_0}a_{i_0,j_0} + \sum\limits_{{k = 1}\atop{k\not = i_{0},j_0}}^n \alpha_{i_0,k}a_{i_0,k}a_{k,j_0} > 0.
\end{eqnarray*}
\end{itemize}\end{itemize}
\end{theorem}

Since
\begin{eqnarray*}
(1 - \alpha_{i,j})a_{i,j} - \sum\limits_{{k = 1}\atop{k\not = i,j}}^n \alpha_{i,k}a_{i,k}a_{k,j}
 \le (1 - \alpha_{i,j})a_{i,j}, \; i\not = j,
\end{eqnarray*}
then from Theorems \ref{thm3.5}-\ref{thm3.8}, we can prove the following corollaries, directly.

\begin{corollary}\label{coro3-9}
Suppose that $0 \le \alpha_{i,j} \le 1$, $i,j = 1,\cdots,n$, $i\not = j$,
and \\$\sum_{k = 1,k\not = i}^n\alpha_{i,k}a_{i,k}a_{k,i} < 1$,
$i = 1,\cdots,n$. Then Theorem A is valid for $\nu = 2$.
\end{corollary}

For the special case when
$\alpha_{i,j} = \alpha \ge 0$ for $i > j$, $\alpha_{i,j} = \beta \ge 0$ for $i < j$, $i = 1, \cdots, n$,
the result is better than the corresponding one given by \cite[Theorem 3.1, Corollaries 3.2, 3.3]{WZ11},
where the assumption that $A$ is irreducible is redundant.

\begin{corollary}\label{coro3-10}
Suppose that $0 \le \alpha_{i,j} \le 1$, $i,j = 1,\cdots,n$, $i\not = j$.
Then Theorem B is valid for $\nu = 2$.
\end{corollary}

The result is consistent with \cite[Theorem 2.7]{WS09} and
it is better than the corresponding one given by \cite[Theorem 2.2]{WZ11},
where there are problems in the expression.

\begin{corollary}\label{coro3-11}
Suppose that $\sum_{k = 1,k\not = i}^n\alpha_{i,k}a_{i,k}a_{k,i} < 1$, $i = 1,\cdots,n$.
Then Theorem C is valid for $\nu = 2$,
provided one of the following conditions is satisfied:

\begin{itemize}
\item[(i)]
For $i,j = 1, \cdots, n$, $i\not = j$, $0 \le \alpha_{i,j} \lesssim 1$.
One of the conditions
$0\le\gamma<1$ and ($ii_1$)-($ii_9$) whenever $\gamma=1$ in Theorem \ref{thm3.7} holds.

\item[(ii)]
One of the conditions ($iii$) and ($iv$) of Theorem \ref{thm3.7} holds, where the inequality
(\ref{eqs3.12}) will be replaced by
$0 \le \alpha_{i,j} \le 1$, $i,j = 1, \cdots, n$, $i\not = j$.
\end{itemize}
\end{corollary}

\begin{corollary}\label{coro3-12}
Theorem D is valid for $\nu = 2$,
provided one of the conditions ($ii$) of Theorem \ref{thm3.8},
($i$) and ($ii$) of Corollary \ref{coro3-11} is satisfied.
\end{corollary}

\subsection{Lower triangular preconditioners}

Let
\begin{eqnarray*}
\alpha_{i,j} = 0, \; i = 1, \cdots, n, j \ge i.
\end{eqnarray*}
Then $Q_2$ reduces to
\begin{eqnarray*}
Q_3 = \left(\begin{array}{ccccc}
0 & 0 &  \cdots & 0 & 0\\
-\alpha_{2,1}a_{2,1} & 0  & \cdots & 0 & 0\\
-\alpha_{3,1}a_{3,1} & -\alpha_{3,2}a_{3,2}   & \ddots & 0 & 0 \\
\vdots & \vdots   & \ddots & \ddots & \vdots\\
-\alpha_{n,1}a_{n,1} & -\alpha_{n,2}a_{n,2} &   \cdots &
-\alpha_{n,n-1}a_{n,n-1} & 0
\end{array}\right)
 \end{eqnarray*}
with $\alpha_{i,j} \ge 0$, $i = 2, \cdots, n$, $j < i$, and
\begin{eqnarray*}
\sum\limits_{i = 2}^n\sum\limits_{j = 1}^{i-1}\alpha_{i,j}a_{i,j} \not = 0.
\end{eqnarray*}

\begin{theorem}\label{thm3.9}
Suppose that $\sum_{k = 1}^{i-1}\alpha_{i,k}a_{i,k}a_{k,i} < 1$, $i = 2,\cdots,n$, and
\begin{eqnarray} \label{eqs3.17}
(1 - \alpha_{i,j})a_{i,j} - \sum\limits_{{k = 1}\atop{k\not = j}}^{i-1}\alpha_{i,k}a_{i,k}a_{k,j}
 \le 0, \; i = 2,\cdots,n, j < i.
\end{eqnarray}
Then Theorem A is valid for $\nu = 3$.
\end{theorem}

\begin{proof}
For $i,j = 1, \cdots, n$, we have
\begin{eqnarray*}
\sum\limits_{{k = 1}\atop{k\not = i,j}}^n\alpha_{i,k}a_{i,k}a_{k,j}
 = \sum\limits_{{k = 1}\atop{k\not = j}}^{i-1}\alpha_{i,k}a_{i,k}a_{k,j}.
\end{eqnarray*}

If $j > i$ then
\begin{eqnarray*}
(1 - \alpha_{i,j})a_{i,j} - \sum\limits_{{k = 1}\atop{k\not = i,j}}^n\alpha_{i,k}a_{i,k}a_{k,j}
 = a_{i,j} - \sum\limits_{k = 1}^{i-1}\alpha_{i,k}a_{i,k}a_{k,j}
 \le a_{i,j} \le 0.
\end{eqnarray*}

This proves that the condition of Theorem \ref{thm3.5} is satisfied, so that Theorem A is valid. \end{proof}

Similarly, by Theorem \ref{thm3.6} we can prove the following theorem.

\begin{theorem}\label{thm3.10}
Suppose that (\ref{eqs3.17}) holds.
Then Theorem B is valid for $\nu = 3$.
\end{theorem}

In this case, since $\alpha_{i,j} = 0$ for $i \le j$, then $\delta_{1,j}^{(2)}(\gamma) = 0$,
$j = 1, \cdots, n$, so that the conditions ($ii_3$), ($ii_5$),
($ii_6$), ($iii$), ($iv_2$), ($iv_4$), ($iv_5$), ($iv^d$) and ($iv^e$) in Theorem \ref{thm3.7} can be not satisfied.

Now, $\delta_{i,j}^{(2)}(1)$ reduces to
\begin{eqnarray*}
\delta_{i,j}^{(3)}(1) = \left\{\begin{array}{ll}
 \sum\limits_{k = 1}^{j-1} \alpha_{i,k}a_{i,k}a_{k,j}, &
 i = 2,\cdots,n, j = 2,\cdots,i;\\
 0, & otherwise.
\end{array}\right.
\end{eqnarray*}

\begin{theorem}\label{thm3.11}
Suppose that $\sum_{k = 1}^{i-1}\alpha_{i,k}a_{i,k}a_{k,i} < 1$, $i = 2,\cdots,n$.
Then Theorem C is valid for $\nu = 3$,
provided one of the following conditions is satisfied:

\begin{itemize}
\item[(i)] $0 \le \gamma < 1$ and
\begin{eqnarray} \label{eqs3.19}\nonumber
&& (1 - \alpha_{i,j})a_{i,j} - \sum\limits_{{k = 1}\atop{k\not = j}}^{i-1} \alpha_{i,k}a_{i,k}a_{k,j}
 \lesssim 0  \;\; \hbox{whenever} \;\; a_{i,j} < 0,\\
 && i = 2,\cdots,n, \; j < i.
\end{eqnarray}

\item[(ii)] $\gamma = 1$, (\ref{eqs3.19}) holds and
one of the following conditions holds:

\begin{itemize}
\item[($ii_1$)] There exist $i\in\{2, \cdots, n\}$, $j\in\{2, \cdots, i\}$ and
$k\in\{1, \cdots, j-1\}$ such that $\alpha_{i,k}a_{i,k}a_{k,j} > 0$.

\item[($ii_2$)]
$a_{n,1} < 0$ and $\alpha_{n,1} > 0$.

\item[($ii_3$)]
$a_{k,k+1} < 0$, $k = 1, \cdots, n-1$.
\end{itemize}

\item[(iii)] $\gamma = 1$, the inequality (\ref{eqs3.17}) holds. For each $i\in\{2, \cdots, n-1\}$ there exist
$j(i)\in\{2, \cdots, i\}$ and $k(i)\in\{1, \cdots, j(i)-1\}$ such that $\alpha_{i,k(i)}a_{i,k(i)}a_{k(i),j(i)} > 0$.
And one of the following conditions holds:

\begin{itemize}
\item[($iii_1$)]
There exist $j\in\{2, \cdots, n\}$ and $k\in\{1, \cdots, j-1\}$ such that
$\alpha_{n,k}a_{n,k}a_{k,j}$ $ > 0$.

\item[($iii_2$)] $a_{n,1} < 0$ and $\alpha_{n,1} > 0$.

\item[($iii_3$)] There exists $j\in\{2, \cdots, n-1\}$ such that
\begin{eqnarray*}
(1 - \alpha_{n,j})a_{n,j} - \sum\limits_{{k = 1}\atop{k\not = j}}^{n-1}\alpha_{n,k}a_{n,k}a_{k,j} < 0.
\end{eqnarray*}
\end{itemize}\end{itemize}
\end{theorem}

\begin{proof}
By Theorem \ref{thm3.7}, ($i$) and ($ii_2$) are obvious.

When ($ii_3$) holds, we have $\max\limits_{1 \le k \le n-1}\{a_{k,k+1}\} < 0$ and so that
\begin{eqnarray*}
\sum\limits_{i = 2}^{n}\sum\limits_{j = 2}^{i}\sum\limits_{k = 1}^{j-1}\alpha_{i,k}a_{i,k}a_{k,j}
& \ge & \sum\limits_{i = 2}^{n}\sum\limits_{j = 2}^{i}\alpha_{i,j-1}a_{i,j-1}a_{j-1,j} \\
& \ge & \max\limits_{1 \le k \le n-1}\{a_{k,k+1}\}\sum\limits_{i = 2}^{n}\sum\limits_{j = 1}^{i-1}\alpha_{i,j}a_{i,j} \\
& > & 0,
\end{eqnarray*}
which implies that ($ii_1$) holds.

When ($ii_1$) holds we have that $\delta_{i,j}^{(3)}(1) \ge \alpha_{i,k}a_{i,k}a_{k,j} > 0$,
i.e., the conditions ($ii_1$) in Theorem \ref{thm3.7} holds, so that Theorem C is valid.

Now, we prove ($iii$). Clearly, $\delta_{i,j(i)}^{(3)}(1) \ge
\alpha_{i,k(i)}a_{i,k(i)}a_{k(i),j(i)} > 0$, which implies
that the condition ($iv_1$) in Theorem \ref{thm3.7} holds.
The required result follows by ($iv^a$), ($iv^b$) and ($iv^c$) in Theorem \ref{thm3.7}, immediately.
\end{proof}

The later part of the condition ($ii_1$) is equivalent to there exist positive elements in
lower triangular part of the matrix $Q_3U$ except the first column.

For ($ii_1$) we can choose some special $\{i,j,k\}$ to construct $Q_3$, e.g., $j = i$,
$k = 1$, $k = j-1$, etc.

Similarly, by Theorem \ref{thm3.8} we can prove the following result immediately.

\begin{theorem}\label{thm3.12}
Theorem D is valid for $\nu = 3$,
provided one of the conditions ($i$), ($ii$) and ($iii$) of Theorem \ref{thm3.11} is satisfied.
\end{theorem}

Similar to Corollaries \ref{coro3-9}-\ref{coro3-12}, from Theorems \ref{thm3.9}-\ref{thm3.12}
we have the following corollaries, immediately.

\begin{corollary}\label{coro3-13}
Suppose that $0 \le \alpha_{i,j} \le 1$ and $\sum_{k = 1}^{i-1}\alpha_{i,k}a_{i,k}a_{k,i} < 1$,
$i = 2,\cdots,n$, $j < i$.
Then Theorem A is valid for $\nu = 3$.
\end{corollary}

\begin{corollary}\label{coro3-14}
Suppose that $0 \le \alpha_{i,j} \le 1$, $i = 2,\cdots,n$, $j < i$.
Then Theorem B is valid for $\nu = 3$.
\end{corollary}

\begin{corollary}\label{coro3-15}
Suppose that $\sum_{k = 1}^{i-1}\alpha_{i,k}a_{i,k}a_{k,i} < 1$,
$i = 2,\cdots,n$. Then Theorem C is valid for $\nu = 3$,
provided one of the following conditions is satisfied:

\begin{itemize}
\item[(i)]
For $i = 2,\cdots,n$, $j < i$, $0 \le \alpha_{i,j} \lesssim 1$.
One of the conditions $0\le\gamma<1$ and ($ii_1$), ($ii_2$), ($ii_3$) whenever $\gamma=1$ in
Theorem \ref{thm3.11} holds.

\item[(ii)] The condition ($iii$) of Theorem \ref{thm3.11} holds,
where the inequality (\ref{eqs3.17})
will be replaced by $0 \le \alpha_{i,j} \le 1$, $i = 2,\cdots,n$, $j < i$.
\end{itemize}
\end{corollary}

\begin{corollary}\label{coro3-16}
Theorem D is valid for $\nu = 3$, provided one of the conditions ($i$) and ($ii$)
of Corollary \ref{coro3-15} is satisfied.
\end{corollary}

Many known corresponding results about the preconditioned AOR method
proposed in the references are the special cases of Theorems \ref{thm3.9}-\ref{thm3.12} and
Corollaries \ref{coro3-13}-\ref{coro3-16}, i.e., they
can be derived from these theorems, immediately.

As a special case of $Q_3$, let
\begin{eqnarray*}
\alpha_{i,j} = \alpha, \; i = 2, \cdots, n, \; j < i,
\end{eqnarray*}
with $\alpha > 0$. Then in \cite{WZ11} $Q$ is defined as
\begin{eqnarray*}
Q_4 = \alpha L,
\end{eqnarray*}
which is studied in \cite{YLK14}. When $\alpha = 1$ it is given in \cite{Li08}
for the preconditioned Gauss-Seidel method.

From Theorems \ref{thm3.9}-\ref{thm3.12} and Corollaries \ref{coro3-13}-\ref{coro3-16},
we have the following comparison results.

\begin{theorem}\label{thm3.13}
Suppose that $\alpha\sum_{k = 1}^{i-1}a_{i,k}a_{k,i} < 1$, $i = 2,\cdots,n$, and
\begin{eqnarray} \label{eqs3.20}
(1 - \alpha)a_{i,j} - \alpha\sum\limits_{{k = 1}\atop{k\not = j}}^{i-1}a_{i,k}a_{k,j} \le 0, \;
 i = 2,\cdots,n, j < i.
\end{eqnarray}
Then Theorem A is valid for $\nu = 4$.
\end{theorem}

\begin{theorem}\label{thm3.14}
Suppose that (\ref{eqs3.20}) holds.
Then Theorem B is valid for $\nu = 4$.
\end{theorem}

This theorem is better than the corresponding one given by \cite[Theorem 2.1]{WZ11},
where there are problems in the expression.

\begin{theorem}\label{thm3.15}
Suppose that $\alpha\sum_{k = 1}^{i-1}a_{i,k}a_{k,i} < 1$, $i = 2,\cdots,n$.
Then Theorem C is valid for $\nu = 4$,
provided one of the following conditions is satisfied:

\begin{itemize}
\item[(i)] $0 \le \gamma < 1$ and
\begin{eqnarray} \label{eqs3.22}
&& (1 - \alpha)a_{i,j} - \alpha\sum\limits_{{k = 1}\atop{k\not = j}}^{i-1}a_{i,k}a_{k,j} \lesssim 0
\;\; \hbox{whenever} \;\; a_{i,j} < 0,\\ \nonumber
&& i = 2,\cdots,n, \; j < i.
\end{eqnarray}

\item[(ii)] $\gamma = 1$, (\ref{eqs3.22}) holds and
one of the following conditions holds:

\begin{itemize}
\item[($ii_1$)] There exist $i\in\{2, \cdots, n\}$, $j\in\{2, \cdots, i\}$
and $k\in\{1, \cdots, j-1\}$ such that $a_{i,k}a_{k,j} > 0$.

\item[($ii_2$)]
$a_{n,1} < 0$.

\item[($ii_3$)]
$a_{1,2} < 0$.
\end{itemize}

\item[(iii)] $\gamma = 1$ and (\ref{eqs3.20}) holds. For each $i\in\{2, \cdots, n-1\}$ there exist
$j(i)\in\{2, \cdots, i\}$ and $k(i)\in\{1, \cdots, j(i)-1\}$ such that $a_{i,k(i)}a_{k(i),j(i)} > 0$.
And one of the following conditions holds:

\begin{itemize}
\item[($iii_1$)]
There exist $j\in\{2, \cdots, n\}$ and $k\in\{1, \cdots, j-1\}$ such that
$a_{n,k}a_{k,j} > 0$.

\item[($iii_2$)]
$a_{n,1} < 0$.

\item[($iii_3$)] There exists $j\in\{2, \cdots, n-1\}$ such that
\begin{eqnarray*}
(1 - \alpha)a_{n,j} - \alpha\sum\limits_{{k = 1}\atop{k\not = j}}^{n-1}a_{n,k}a_{k,j} < 0.
\end{eqnarray*}
\end{itemize}\end{itemize}
\end{theorem}

\begin{proof}
By Theorem \ref{thm3.11} we just need to prove ($ii_3$). In fact, if $a_{n,1} < 0$ then
($ii_2$) holds. If $a_{n,1} = 0$ then, from the irreducibility of $A$,
$\sum_{i = 2}^{n-1}a_{i,1} = \sum_{i = 2}^{n}a_{i,1} < 0$ so that there exists
$i_0\in\{2, \cdots, n-1\}$ such that $a_{i_0,1} < 0$. Hence $a_{i_0,1}a_{1,2} > 0$. This shows that
($ii_1$) holds for $i = i_0$, $j = 2$ and $k = 1$.  \end{proof}

\begin{theorem}\label{thm3.16}
Theorem D is valid for $\nu = 4$,
provided one of the conditions ($i$), ($ii$) and ($iii$) of Theorem \ref{thm3.15} is satisfied.
\end{theorem}

From Theorems \ref{thm3.13}-\ref{thm3.16}, the following results are directly.

\begin{corollary}
Suppose that $0 < \alpha \le 1$ and $\alpha\sum_{k = 1}^{i-1}a_{i,k}a_{k,i} < 1$, $i = 2,\cdots,n$.
Then Theorem A is valid for $\nu = 4$.
\end{corollary}

\begin{corollary}
Suppose that $0 < \alpha \le 1$.
Then Theorem B is valid for $\nu = 4$.
\end{corollary}

\begin{corollary}\label{coro3-19}
Suppose that $\alpha\sum_{k = 1}^{i-1}a_{i,k}a_{k,i} < 1$, $i = 2,\cdots,n$.
Then Theorem C is valid for $\nu = 4$, provided one of the following conditions is satisfied:

\begin{itemize}
\item[(i)] $0 < \alpha \lesssim 1$.
One of the conditions $0\le\gamma<1$ and ($ii_1$), ($ii_2$), ($ii_3$) whenever $\gamma=1$ in
Theorem \ref{thm3.15} holds.

\item[(ii)] The condition ($iii$) of Theorem \ref{thm3.15} holds,
where the inequality (\ref{eqs3.20})
will be replaced by $0 < \alpha \le 1$.
\end{itemize}
\end{corollary}

The result when ($i$) holds is better than the corresponding one given by \cite[Theorem 4.2]{YLK14}.

If $\sum_{k = 1}^{i-1}a_{i,k}a_{k,i} > 0, \; i = 2,\cdots,n$,
then for each $i\in\{2, \cdots, n\}$, there exists
$k(i)\in\{1, \cdots, i-1\}$ such that
$a_{i,k(i)}a_{k(i),i} > 0$, which implies that ($iii_1$) in Theorem \ref{thm3.15}
holds. Hence, Corollary \ref{coro3-19} when ($ii$) holds is better than the corresponding
one given by \cite[Theorem 4.1]{YLK14}.

\begin{corollary}
Theorem D is valid for $\nu = 4$,
provided one of the conditions ($i$) and ($ii$)
of Corollary \ref{coro3-19} is satisfied.
\end{corollary}

Specially, for some $r$, $2 \le r \le n$, $\alpha_{r,j} = \alpha_j \ge 0$, $j = 1, \cdots, r-1$, and $\alpha_{i,j} = 0$ otherwise,
in \cite{Wa06} the matrix $Q$ is defined as
\begin{eqnarray*}
Q_5 = \left(\begin{array}{ccccccc}
0 & \quad &\cdots & 0 & 0 &  \quad\cdots \quad & 0\\
\vdots & & \ddots & \vdots & \vdots & \vdots & \vdots\\
0 & &\cdots & 0 & 0 & \cdots & 0\\
-\alpha_1a_{r,1} & &\cdots & -\alpha_{r-1}a_{r,r-1} & 0 & \cdots & 0\\
0 & &\cdots & 0 & 0 & \cdots & 0\\
\vdots & & \vdots &\vdots & \vdots & \vdots & \vdots\\
0 & &\cdots & 0 & 0 & \cdots & 0
\end{array}\right)
 \end{eqnarray*}
 with
 \begin{eqnarray*}
\sum\limits_{k = 1}^{r-1}\alpha_ka_{r,k} \not = 0.
 \end{eqnarray*}

When $r=n$, it is proposed in \cite{NKM08} for the preconditioned Gauss-Seidel method.

In this case, for $i = 2, \cdots, n$, if $i \not = r$, then
\begin{eqnarray*}
\sum\limits_{k = 1}^{i-1}\alpha_{i,k}a_{i,k}a_{k,i} = 0
 \end{eqnarray*}
and
\begin{eqnarray*}
(1 - \alpha_{i,j})a_{i,j} - \sum\limits_{{k = 1}\atop{k\not = i,j}}^{i-1}\alpha_{i,k}a_{i,k}a_{k,j}
 = a_{i,j}, \;\; \hbox{for} \;\; j < i.
 \end{eqnarray*}

Now, $\delta_{i,j}^{(3)}(1)$ reduces to
\begin{eqnarray*}
\delta_{i,j}^{(5)}(1) = \left\{\begin{array}{ll}
 \sum\limits_{k = 1}^{j-1} \alpha_ka_{r,k}a_{k,j}, &
 i = r, j = 2,\cdots,r;\\
 0, & otherwise.
\end{array}\right.
\end{eqnarray*}

Hence, from Theorems \ref{thm3.9}-\ref{thm3.12} and Corollaries \ref{coro3-13}-\ref{coro3-16},
we can obtain the following comparison results, directly.

\begin{theorem}
Suppose that $\sum_{k = 1}^{r-1}\alpha_{k}a_{r,k}a_{k,r} < 1$ and
\begin{eqnarray} \label{eqs3.23}
(1 - \alpha_j)a_{r,j} - \sum\limits_{{k = 1}\atop{k\not = j}}^{r-1}\alpha_{k}a_{r,k}a_{k,j}
 \le 0, \; j = 1, \cdots, r-1.
\end{eqnarray}
Then Theorem A is valid for $\nu = 5$.
\end{theorem}

\begin{theorem}
Suppose that (\ref{eqs3.23}) holds.
Then Theorem B is valid for $\nu = 5$.
\end{theorem}

When $\alpha_j = 1$, $j = 1, \cdots, r-1$, the inequality (\ref{eqs3.23}) is trivial.
Hence the result is better than \cite[Theorem 2.9]{NKM08}, where the convergence hypothesis of
two Gauss-Seidel methods is unnecessary and the proof is insufficient, which is pointed
out by \cite{LC09}. While the condition $\rho(\mathscr{L}) > 0$ in \cite[Theorem 3.2]{LC09}
is unnecessary.

\begin{theorem}\label{thm3.19}
Suppose that $\sum_{k = 1}^{r-1}\alpha_{k}a_{r,k}a_{k,r} < 1$ and
\begin{eqnarray}\label{eqs3.25}
&& (1 - \alpha_j)a_{r,j} - \sum\limits_{{k = 1}\atop{k\not = j}}^{r-1}\alpha_{k}a_{r,k}a_{k,j}
 \lesssim 0 \;\; \hbox{whenever} \;\; a_{r,j} < 0, \\\nonumber
 && j = 1, \cdots, r-1.
\end{eqnarray}
Then Theorem C is valid for $\nu = 5$,
provided one of the following conditions is satisfied:

\begin{itemize}
\item[(i)] $0 \le \gamma < 1$.

\item[(ii)] $\gamma = 1$ and one of the following conditions holds:

\begin{itemize}
\item[($ii_1$)] There exist $j\in\{2, \cdots, r\}$ and
$k\in\{1, \cdots, j-1\}$ such that $\alpha_{k}a_{r,k}a_{k,j}$ $ > 0$.

\item[($ii_2$)] $a_{k,k+1} < 0$, $k = 1, \cdots, r-1$.

\item[($ii_3$)] $a_{k,r} < 0$, $k = 1, \cdots, r-1$.

\item[($ii_4$)] $r = n$, $a_{n,1} < 0$ and $\alpha_1 > 0$.
\end{itemize}\end{itemize}
\end{theorem}

\begin{proof}
By ($i$) and ($ii$) of Theorem \ref{thm3.11}, ($i$), ($ii_1$) and ($ii_4$) are derived directly.

By the definition of $Q_5$, there exists $k_0 \in \{1, \cdots, r-1\}$ such that
$\alpha_{k_0}a_{r,k_0} < 0$.

If ($ii_2$) holds then $\alpha_{k_0}a_{r,k_0}a_{k_0,k_0+1} > 0$, which shows that
($ii_1$) holds for $j = k_0+1$ and $k = k_0$.

Similarly, if ($ii_3$) holds then $\alpha_{k_0}a_{r,k_0}a_{k_0,r} > 0$, which shows that
($ii_1$) holds for $j = r$ and $k = k_0$.
\end{proof}

\begin{theorem}
Suppose that (\ref{eqs3.25}) holds.
Then Theorem D is valid for $\nu = 5$,
provided one of the conditions ($i$) and ($ii$)
of Theorem \ref{thm3.19} is satisfied.
\end{theorem}

\begin{corollary}
Suppose that $0 \le \alpha_k \le 1$, $k = 1, \cdots, r-1$,
and $\sum_{k = 1}^{r-1}\alpha_{k}a_{r,k}a_{k,r}$ $ < 1$.
Then Theorem A is valid for $\nu = 5$.
\end{corollary}

\begin{corollary}
Suppose that
$0 \le \alpha_k \le 1$, $k = 1, \cdots, r-1$.
Then Theorem B is valid for $\nu = 5$.
\end{corollary}

The result includes the corresponding one given in \cite[Corollary 2.3]{Wa06}.

\begin{corollary}
Suppose that
$0 \le \alpha_j \lesssim 1$, $j = 1, \cdots, r-1$, and $\sum_{k = 1}^{r-1}\alpha_{k}a_{r,k}a_{k,r}$ $ < 1$.
Then Theorem C is valid for $\nu = 5$,
provided one of the conditions ($i$) and ($ii$)
of Theorem \ref{thm3.19} is satisfied.
\end{corollary}

\begin{corollary}
Suppose that $0 \le \alpha_k \lesssim 1$, $j = 1, \cdots, r-1$.
Then Theorem D is valid for $\nu = 5$,
provided one of the conditions ($i$) and ($ii$)
of Theorem \ref{thm3.19} is satisfied.
\end{corollary}

Similarly, for some $r$, $2 \le r \le n$, $\alpha_{i,r-1} = \alpha_i \ge 0$,
$i = r, \cdots, n$, and $\alpha_{i,j} = 0$ otherwise, the matrix $Q_3$ reduces to
\begin{eqnarray*}
Q_6 = \left(\begin{array}{cccccccc}
0 & \quad \cdots \quad & 0 & 0 & 0 & \quad\cdots \quad & 0\\
\vdots & \vdots & \vdots & \vdots & \vdots & \vdots & \vdots\\
0 & \cdots & 0 &  0 & 0 & \cdots & 0\\
0 &  \cdots & 0 & -\alpha_{r}a_{r,r-1} & 0 & \cdots & 0\\
\vdots &   \vdots & \vdots &\vdots & \vdots & \ddots & \vdots\\
0 & \cdots & 0 & -\alpha_{n}a_{n,r-1} & 0 & \cdots & 0
\end{array}\right)
 \end{eqnarray*}
 with
 \begin{eqnarray*}
\sum\limits_{k = r}^{n}\alpha_ka_{k,r-1} \not = 0.
 \end{eqnarray*}

When $r = 2$, it is investigated in \cite{LY08,LC10,Yu11},
in \cite{HCC06} for the preconditioned SOR method and in \cite{HNT03}
for the preconditioned Gauss-Seidel and Jacobi methods, respectively.
When $r = 2$ and $\alpha_i = 1$, $i = 2, \cdots, n$, it is a special case in \cite{Mi87}
for the preconditioned Gauss-Seidel and Jacobi methods, and it is used
to the preconditioned AOR method in \cite{LLW07}.

In this case, $\delta_{i,j}^{(3)}(1)$ reduces to
\begin{eqnarray*}
\delta_{i,j}^{(6)}(1) = \left\{\begin{array}{ll}
 \alpha_ia_{i,r-1}a_{r-1,j}, &
 i = r,\cdots,n, j = r,\cdots,i;\\
 0, & otherwise.
\end{array}\right.
\end{eqnarray*}

\begin{theorem}\label{thm3.21}
Suppose that $0 \le \alpha_k \le 1$ and
$\alpha_{k}a_{k,r-1}a_{r-1,k} < 1$, $k = r, \cdots, n$.
Then Theorem A is valid for $\nu = 6$.
\end{theorem}

\begin{proof} It is easy to prove that the condition of Corollary \ref{coro3-13} is satisfied,
so that Theorem A is valid. \end{proof}

The result includes the corresponding one given by \cite[Theorem 2.2-(a)]{LY08}.
The result for $\omega = \gamma$ includes the corresponding one given by
\cite[Theorem 3.3]{Yu072}, where the condition is too strong.

Similarly, by Corollary \ref{coro3-14} we can prove the following theorem.

\begin{theorem}\label{thm3.22}
Suppose that $0 \le \alpha_k \le 1$, $k = r, \cdots, n$.
Then Theorem B is valid for $\nu = 6$.
\end{theorem}

In order to give the Stein-Rosenberg Type Theorem II,
we prove a lemma.

\begin{lemma}\label{lem3-4}
Let $A$ be an irreducible Z-matrix.
Assume that $r = 2$, $0 < \alpha_k \le 1$,
$k = 2, \cdots, n$ and $A^{(6)}$ has the block form
\begin{eqnarray*}
A^{(6)} = \left(\begin{array}{ccc}
1 & & \bar{a}_{1,2}\\
\bar{a}^{(6)}_{2,1} & & A^{(6)}_{2,2}
\end{array}\right), \; A^{(6)}_{2,2}\in{\mathscr R}^{(n-1)\times(n-1)}.
 \end{eqnarray*}
 Then

\begin{itemize}
\item[(i)] $A^{(6)}_{2,2}$ is an irreducible Z-matrix.

\item[(ii)] $A^{(6)}$ is an irreducible Z-matrix if and only if there exists
$i_0\in\{2, \cdots, n\}$ such that $(1-\alpha_{i_0})a_{i_0,1} \not = 0$.
\end{itemize}
\end{lemma}

\begin{proof}
Since
\begin{eqnarray*}
a^{(6)}_{i,j} = \left\{\begin{array}{lll}
a_{1,j}, & i = 1, j = 1,\cdots n,\\
(1 - \alpha_{i})a_{i,1}, & i = 2,\cdots, n, j = 1, \\
a_{i,j} - \alpha_{i}a_{i,1}a_{1,j}, & i,j = 2,\cdots, n,
\end{array}\right.
\end{eqnarray*}
then it is clearly that $a^{(6)}_{1,k} = a_{1,j} \le 0$,
$a^{(6)}_{k,1} = (1 - \alpha_{i})a_{i,1} \le 0$ for $k = 2,\cdots,n$ and
\begin{eqnarray} \label{eqs3.26}
a^{(6)}_{i,j} = a_{i,j} - \alpha_{i}a_{i,1}a_{1,j} \le a_{i,j} \le 0, \;
i,j = 2,\cdots, n, \; i\not = j.
\end{eqnarray}
Hence, $A^{(6)}$ and $A^{(6)}_{2,2}$ are Z-matrices.

For any $i,j\in\{2, \cdots, n\}$, $i\not = j$, since $A$ is irreducible, then there exists a path
$\sigma_{i,j} = (j_0, j_1, \cdots, j_{l+1})\in G(A)$ with $i = j_0$ and $j = j_{l+1}$.

If $j_k\in\{2, \cdots, n\}$ for $k = 1, \cdots, l$, then, by (\ref{eqs3.26}), it gets that
$a^{(6)}_{j_k,j_{k+1}} \le a_{j_k,j_{k+1}} < 0$ and $a^{(6)}_{j_0,j_{1}} \le a_{j_0,j_{1}} < 0$, so that
$\sigma_{i,j}\in G(A^{(6)}_{2,2})$.

For the case when there exists $s\in\{1, \cdots, l\}$ such that $j_s = 1$, we have
$j_{s-1} > 1$, $j_{s+1} > 1$, $a_{j_{s-1},1} < 0$ and $a_{1,j_{s+1}} < 0$.
By (\ref{eqs3.26}), it gets that $a^{(6)}_{j_{s-1},j_{s+1}} = a_{j_{s-1},j_{s+1}}
- \alpha_{j_{s-1}}a_{j_{s-1},1}a_{1,j_{s+1}} \le -\alpha_{j_{s-1}}a_{j_{s-1},1}a_{1,j_{s+1}} < 0$.
It follows that $\tilde{\sigma}_{i,j} = (j_0, \cdots, j_{s-1}, j_{s+1},$  $ \cdots, j_{l+1}) \in G(A^{(6)}_{2,2})$.

We have proved ($i$).

The necessity of ($ii$) is obvious. Now we prove the sufficiency.

For any $i,j\in\{1, \cdots, n\}$, $i\not = j$, if $i,j\in\{2, \cdots, n\}$
then, by ($i$), there exists a path $\sigma_{i,j}$ such that
$\sigma_{i,j}\in G(A^{(6)}_{2,2}) \subseteq G(A^{(6)})$.

For the case when $i = 1$, since $A$ is an irreducible Z-matrix, then there exists
$j_0\in\{2, \cdots, n\}$ such that $a^{(6)}_{1,j_0} = a_{1,j_0} < 0$. By ($i$),
there exists a path $\sigma_{j_0,j}$ such that
$\sigma_{j_0,j}\in G(A^{(6)}_{2,2})$ so that $(1,\sigma_{j_0,j})\in G(A^{(6)})$.

For the case when $j = 1$, there exists a path $\sigma_{i,i_0}$ such that
$\sigma_{i,i_0}\in G(A^{(6)}_{2,2})$. Since $a^{(6)}_{i_0,i} = (1-\alpha_{i_0})a_{i_0,1} \not = 0$,
it follows that $(\sigma_{i,i_0},1)\in G(A^{(6)})$.

We have proved ($ii$).
\end{proof}

This lemma improves \cite[Theorem 3]{Ed18}.

\begin{theorem}\label{thm3.23}
Suppose that $\alpha_{k}a_{k,r-1}a_{r-1,k} < 1$, $k = r, \cdots, n$.
Then Theorem C is valid for $\nu = 6$,
provided one of the following conditions is satisfied:

\begin{itemize}
\item[(i)]
$0 \le \gamma < 1$ and $0 \le \alpha_k \lesssim 1$, $k = r, \cdots, n$.

\item[(ii)]
$\gamma = 1$ and $0 \le \alpha_k \lesssim 1$, $k = r, \cdots, n$.
And one of the following conditions holds:

\begin{itemize}
\item[($ii_1$)] There exist $i\in\{r, \cdots, n\}$ and $j\in\{r, \cdots, i\}$ such that
$\alpha_ia_{i,r-1}a_{r-1,j}$ $ > 0$.

\item[($ii_2$)]
$a_{r-1,r} < 0$.

\item[($ii_3$)]
$r = 2$, $a_{n,1} < 0$ and $\alpha_n > 0$.
\end{itemize}

\item[(iii)] $r = 2$, $0 \le \gamma < 1$ and $0 < \alpha_k \le 1$, $k = 2, \cdots, n$.

\item[($iv$)] $r = 2$, $\gamma = 1$ and $0 < \alpha_k \le 1$, $k = 2, \cdots, n$.
And one of the following conditions holds:

\begin{itemize}
\item[($iv_1$)] There exist $i\in\{2, \cdots, n\}$ and $j\in\{2, \cdots, i\}$ such that
$a_{i,1}a_{1,j} > 0$.

\item[($iv_2$)]
$a_{1,2} < 0$.

\item[($iv_3$)]
$a_{n,1} < 0$.
\end{itemize}
\end{itemize}
\end{theorem}

\begin{proof} By Corollary \ref{coro3-15}, ($i$), ($ii_1$) and ($ii_3$) are obvious.

Assume that ($ii_2$) holds. By the definition of $Q_6$, there exists
$k_0\in\{r, \cdots, n\}$ such that $\alpha_{k_0}a_{k_0,r-1} < 0$, so that
$\alpha_{k_0}a_{k_0,r-1}a_{r-1,r} > 0$, which implies that ($ii_1$) holds
for $i = k_0$ and $j = r$.

We now prove ($iii$) and ($iv$).

If there exists
$i_0\in\{2, \cdots, n\}$ such that $(1-\alpha_{i_0})a_{i_0,1} \not = 0$, then,
by Lemma \ref{lem3-4}, $A^{(6)}$ is irreducible. By ($i$), ($ii_8$), ($ii_2$) in Theorem \ref{thm3.7}
and Corollary \ref{coro3-12} we can derive ($iii$), ($iv_1$) and ($iv_3$), directly.
When ($iv_2$) holds, then from the irreducibility of $A$, there exists
$i_1\in\{2, \cdots, n\}$ such that $a_{i_1,1} < 0$,
so that $a_{i_1,1}a_{1,2} > 0$, which implies that ($iv_1$) holds for
$i=i_1$ and $j=2$.

For the case when $(1-\alpha_{k})a_{k,1} = 0$, $k = 2,\cdots,n$,
then it gets that $\alpha_{k} = 1$ whenever $a_{k,1} < 0$. In this case the matrix $A^{(6)}$
can be partitioned as
\begin{eqnarray*}
A^{(6)} = \left(\begin{array}{ccc}
1 & & \bar{a}_{1,2}\\
0 & & A^{(6)}_{2,2}
\end{array}\right),
 \end{eqnarray*}
where, by Lemma \ref{lem3-4}, $A^{(6)}_{2,2}\in {\mathscr R}^{(n-1) \times (n-1)}$ is
an irreducible Z-matrix, so that it is also an L-matrix, since
$a^{(6)}_{k,k} = 1 - \alpha_{k}a_{k,1}a_{1,k} > 0$ for $k = 2, \cdots, n$.

Denote $\rho = \rho({\mathscr L}_{\gamma,\omega})$.
Let $x > 0$ be its associated eigenvector.

By Lemma \ref{lem3-1} we just need to consider the case when $\gamma \le \omega$.
By Lemma \ref{lem1-8} it follows that $\rho > 0$ and $x\gg 0$.

Let
\begin{eqnarray*}
A^{(6)}_{2,2} = \bar{M}_{\gamma,\omega} - \bar{N}_{\gamma,\omega}
 \end{eqnarray*}
be the AOR splitting of $A^{(6)}_{2,2}$. Then
\begin{eqnarray*}
M^{(6)}_{\gamma,\omega} = \left(\begin{array}{ccc}
\frac{1}{\omega} & & 0\\
0 & & \bar{M}_{\gamma,\omega}
\end{array}\right), \;
N^{(6)}_{\gamma,\omega} = \left(\begin{array}{ccc}
\frac{1-\omega}{\omega} & & -\bar{a}_{1,2}\\
0 & & \bar{N}_{\gamma,\omega}
\end{array}\right)
 \end{eqnarray*}
 and
\begin{eqnarray*}
[M^{(6)}_{\gamma,\omega}]^{-1} = \left(\begin{array}{ccc}
\omega & & 0\\
0 & & \bar{M}_{\gamma,\omega}^{-1}
\end{array}\right), \;
 {\mathscr L}^{(6)}_{\gamma,\omega}
 = \left(\begin{array}{ccc}
1-\omega & & -\omega\bar{a}_{1,2}\\
0 & & \bar{M}_{\gamma,\omega}^{-1}\bar{N}_{\gamma,\omega}
\end{array}\right).
 \end{eqnarray*}
Let $E_1$ and $F_1$ be diagonal part and strictly lower triangular part of $Q_6U$
with block forms
\begin{eqnarray*}
E_1 = \left(\begin{array}{ccc}
e_{1,1} && 0 \\
0 && E_{2,2}
\end{array}\right), \;
F_1 = \left(\begin{array}{ccc}
f_{1,1} && 0 \\
f_{1,2} && F_{2,2}
\end{array}\right), \; E_{2,2}, F_{2,2}\in{\mathscr R}^{(n-1)\times(n-1)}.
\end{eqnarray*}
Then $e_{1,1} = f_{1,1} = 0$ and $f_{1,2} = 0$. Let
\begin{eqnarray*}
&& Q_6 = \left(\begin{array}{cc}
q^{(6)}_{1,1} & \bar{q}_{1,2}\\
\bar{q}_{2,1} & Q_{2,2}
\end{array}\right), \; Q_{2,2}\in{\mathscr R}^{(n-1)\times (n-1)}, \\
&& x = \left(\begin{array}{c}
\bar{x}_{1}\\
\bar{x}_{2}
\end{array}\right), \; \bar{x}_{1}\in\Re, \; \bar{x}_{2}\in{\mathscr R}^{n-1}.
 \end{eqnarray*}
 Then $q^{(6)}_{1,1} = 0$, $\bar{q}_{1,2} = 0$, $\bar{q}_{2,1} > 0$, $Q_{2,2} = 0$,
 $\bar{x}_{1} > 0$ and $\bar{x}_{2} \gg 0$.
 Now, by Lemma \ref{lem3-2}, we obtain
\begin{eqnarray*}
{\mathscr L}^{(6)}_{\gamma,\omega}x - \rho x
& = & \left(\begin{array}{c}
(1-\omega)\bar{x}_{1} - \omega\bar{a}_{1,2}\bar{x}_{2} - \rho\bar{x}_{1}\\
\bar{M}_{\gamma,\omega}^{-1}\bar{N}_{\gamma,\omega}\bar{x}_{2} -\rho\bar{x}_{2}
\end{array}\right) \\
& = & (\rho-1)[E_1 +\gamma F_1 +(1-\gamma)Q_6]x \\
& = &
(\rho-1)\left(\begin{array}{cc}
\omega & 0\\
0 & \bar{M}_{\gamma,\omega}^{-1}
\end{array}\right)
\left[\left(\begin{array}{cc}
0 & 0 \\
0 & E_{2,2}
\end{array}\right)
+ \gamma\left(\begin{array}{cc}
0 & 0 \\
0 & F_{2,2}
\end{array}\right) \right.\\
&&\left. + (1-\gamma)\left(\begin{array}{cc}
0 & 0\\
\bar{q}_{2,1} & 0
\end{array}\right)\right]\left(\begin{array}{c}
\bar{x}_{1}\\
\bar{x}_{2}
\end{array}\right) \\
& = &
(\rho-1) \left(\begin{array}{c}
0 \\
\bar{M}_{\gamma,\omega}^{-1}[(E_{2,2} + \gamma F_{2,2})\bar{x}_{2} + (1-\gamma)\bar{q}_{2,1}\bar{x}_{1}]
\end{array}\right). \nonumber
\end{eqnarray*}
Hence, we have
\begin{eqnarray} \label{eqs3.27}
(1-\omega)\bar{x}_{1} - \rho\bar{x}_{1}
 = \omega\bar{a}_{1,2}\bar{x}_{2}
\end{eqnarray}
and
\begin{eqnarray} \label{eqs3.28}
\bar{M}_{\gamma,\omega}^{-1}\bar{N}_{\gamma,\omega}\bar{x}_{2} -\rho\bar{x}_{2}
 = (\rho-1)\bar{M}_{\gamma,\omega}^{-1}[(E_{2,2} + \gamma F_{2,2})\bar{x}_{2} + (1-\gamma)\bar{q}_{2,1}\bar{x}_{1}].
\end{eqnarray}

Since $A$ is an irreducible L-matrix, then $\bar{a}_{1,2} < 0$ so that $\omega\bar{a}_{1,2}\bar{x}_{2} < 0$.
From (\ref{eqs3.27}), it gets that $1-\omega < \rho$.

When ($iii$) holds, i.e., $\gamma < 1$, since $\bar{A}_{2,2}$ is an irreducible L-matrix, then it can derive that
$\bar{M}_{\gamma,\omega}^{-1}\bar{N}_{\gamma,\omega} \ge 0$ is irreducible. Furthermore,
$\bar{M}_{\gamma,\omega}^{-1}[(E_{2,2} + \gamma F_{2,2})\bar{x}_{2} + (1-\gamma)\bar{q}_{2,1}\bar{x}_{1}]
\ge (1-\gamma)\bar{M}_{\gamma,\omega}^{-1}\bar{q}_{2,1}\bar{x}_{1} > 0$,
since $\bar{q}_{2,1} > 0$ and $\bar{x}_{1} > 0$.

Now, Theorem \ref{thm3.21} has shown that $\rho = 1$ if and only if
$\rho({\mathscr L}^{(6)}_{\gamma,\omega}) = 1$.
If $\rho < 1$ then, from (\ref{eqs3.28}), it gets that $\bar{M}_{\gamma,\omega}^{-1}\bar{N}_{\gamma,\omega}\bar{x}_{2}
 < \rho\bar{x}_{2}$ so that $\rho(\bar{M}_{\gamma,\omega}^{-1}\bar{N}_{\gamma,\omega}) < \rho$. Hence,
we derive $\rho({\mathscr L}^{(6)}_{\gamma,\omega}) = \max\{1-\omega, \;
\rho(\bar{M}_{\gamma,\omega}^{-1}\bar{N}_{\gamma,\omega})\} < \rho$. If
$\rho > 1$ then, from (\ref{eqs3.28}), it gets that $\bar{M}_{\gamma,\omega}^{-1}\bar{N}_{\gamma,\omega}\bar{x}_{2}
 > \rho\bar{x}_{2}$ so that $\rho(\bar{M}_{\gamma,\omega}^{-1}\bar{N}_{\gamma,\omega}) > \rho$. Hence,
we derive $\rho({\mathscr L}^{(6)}_{\gamma,\omega}) = \max\{1-\omega, \;
\rho(\bar{M}_{\gamma,\omega}^{-1}\bar{N}_{\gamma,\omega})\} = \rho(\bar{M}_{\gamma,\omega}^{-1}\bar{N}_{\gamma,\omega})
 > \rho$.

When ($iv$) holds, then $\omega = 1$ and the AOR method reduces to the Gauss-Seidel method.
In this case, we have
\begin{eqnarray*}
{\mathscr L}^{(6)}
 = \left(\begin{array}{ccc}
0 & & -\bar{a}_{1,2}\\
0 & & \bar{M}_{1,1}^{-1}\bar{N}_{1,1}
\end{array}\right)
 \end{eqnarray*}
 and therefore $\rho({\mathscr L}^{(6)}) = \rho(\bar{M}_{1,1}^{-1}\bar{N}_{1,1})$.

Above we have proved that $\bar{A}^{(6)}_{2,2}$ is an irreducible
L-matrix. By Lemma \ref{lem1-8} it gets that
$\bar{y}^T\bar{M}_{1,1}^{-1} \gg 0$ whenever $\bar{y}$ satisfies $\bar{y} > 0$ and
$\bar{y}^T\bar{M}_{1,1}^{-1}\bar{N}_{1,1} = \rho(\bar{M}_{1,1}^{-1}\bar{N}_{1,1})\bar{y}^T$.
Multiply $\bar{y}^T$ on the left side of (\ref{eqs3.28}), we can derive
\begin{eqnarray*}
\rho(\bar{M}_{1,1}^{-1}\bar{N}_{1,1})\bar{y}^T\bar{x}_{2} -\rho\bar{y}^T\bar{x}_{2}
 = (\rho-1)\bar{y}^T\bar{M}_{\gamma,\omega}^{-1}(E_{2,2} + F_{2,2})\bar{x}_{2}.
\end{eqnarray*}
When ($iv_1$) holds, i.e., there exist $i\in\{2,\cdots,n\}$ and $j\in\{2,\cdots, i\}$ such that
$a_{i,1}a_{1,j} > 0$,
it is easy to prove that $E_{2,2} > 0$ or $F_{2,2} > 0$ so that $(E_{2,2}+F_{2,2})\bar{x}_{2} > 0$
and $\bar{y}^T\bar{M}_{\gamma,\omega}^{-1}(E_{2,2} + F_{2,2})\bar{x}_{2} > 0$. Since $\bar{y}^T\bar{x}_{2} > 0$,
then we can get
\begin{eqnarray*}
\rho({\mathscr L}^{(6)}) = \rho(\bar{M}_{1,1}^{-1}\bar{N}_{1,1})
 \left\{ \begin{array}{lcl}
 < \rho, & \; \hbox{if} \; & \rho < 1 \\
 = \rho, & \hbox{if} & \rho = 1 \\
 > \rho, & \hbox{if} & \rho > 1.
 \end{array}\right.
  \end{eqnarray*}

When ($iv_2$) holds, the irreducibility of $A$ or the definition of $Q_6$ ensures that
 there exists $i\in\{2, \cdots, n\}$ such that $a_{i,1} < 0$, we have
$a_{i,1}a_{1,2} > 0$, which implies that ($iv_1$) holds.

Similarly, when ($iv_3$) holds, the irreducibility of $A$ ensures that
 there exists $j\in\{2, \cdots, n\}$ such that $a_{1,j} < 0$, we have
$a_{n,1}a_{1,j} > 0$, which also implies that ($iv_1$) holds.
\end{proof}

The result for the case when $r=2$ is better than the corresponding
one given by \cite[Theorem 1, Corollary 1]{LLW07}, \cite[Theorems 3.3, 3.4, 3.5]{LC10}
and \cite[Theorems 3.11, 3.13, 3.14 and 3.15]{Yu11}. The proof of
\cite[Theorem 1]{LLW07} is insufficient, which is pointed out by
\cite{YK08} and \cite{Ed18}. When $r=2$
and $\gamma = \omega$, the result is better than \cite[Theorems 2.1 and 2.2]{HCC06},
where the condition $a_{k,k+1}a_{k+1,k} > 0$, $k = 1, \cdots, n-1$, implies that
$A$ is irreducible. While the proofs in \cite{HCC06} are insufficient, which is pointed out by
\cite{Yu072}. The comparison result \cite[Theorem 2.2-(b)]{LY08} is problematic, because
\cite[Lemma 2.1]{LY08} is wrong, which has been shown by \cite[Example 3.1]{YK08}.

From Theorem \ref{thm3.23}, we can prove the following theorem.

\begin{theorem} \label{thm3.24}
Theorem D is valid for $\nu = 6$,
provided one of the conditions ($i$)-($iv$)
of Theorem \ref{thm3.23} is satisfied.
\end{theorem}

When $r=2$, $\alpha_k=1$, $k=2, \cdots, n$, the results given in Theorems \ref{thm3.22} and
 \ref{thm3.24} are better than \cite[Theorem 3.4]{SE13} and \cite[Theorem 3.4]{SER14}.

As a special case of $Q_6$($Q_5$), for some $r > s$ with $a_{r,s} < 0$ and $\alpha > 0$,
the matrix $Q_6$($Q_5$) reduces to
\begin{eqnarray*}
Q_7 & = & \left(\begin{array}{ccccc}
0 & \quad \cdots\quad & 0 & \quad \cdots \quad & 0\\
\vdots & \ddots & \vdots & \cdots & \vdots\\
\vdots & \cdots & 0 & \cdots & 0\\
\vdots &  -\frac{a_{r,s}}{\alpha} & \vdots & \ddots & \vdots\\
0 & \cdots & 0 & \cdots & 0\\
\end{array}\right),
 \end{eqnarray*}
which is proposed in \cite{HCEC05,LLW072,Yu11} for $r = n$ and $s = 1$.
And for $r = n$, $s = 1$ and $\alpha = 1$ in \cite{EMT01}.
It is given in \cite{ZHL05} to replace $-a_{r,s}/\alpha$ with a constant $\beta$.

Now, $\delta_{i,j}^{(6)}(1)$ reduces to
\begin{eqnarray*}
\delta_{i,j}^{(7)}(1) = \left\{\begin{array}{ll}
 \frac{1}{\alpha}a_{r,s}a_{s,j}, \; & i = r, j = s+1, \cdots, r;\\
 0, & otherwise.
\end{array}\right.
\end{eqnarray*}

From Theorems \ref{thm3.21} and \ref{thm3.22},
we can obtain the following comparison results, directly.

\begin{theorem}
Suppose that $\alpha \ge 1$ and $\alpha > a_{r,s}a_{s,r}$.
Then Theorem A is valid for $\nu = 7$.
\end{theorem}

This result is better than that given by \cite[Theorem 3.7]{Yu11},
where $A$ is assumed to be irreducible.

\begin{theorem}
Suppose that $\alpha \ge 1$.
Then Theorem B is valid for $\nu = 7$.
\end{theorem}

This result includes \cite[Theorem 3.4]{Yu12} and the corresponding one given in \cite[Theorem 3.1]{Li02}.

In order to give the Stein-Rosenberg Type Theorem II,
we prove a lemma.

\begin{lemma}\label{lem3-5}
Let $A$ be an irreducible Z-matrix.
Assume that $r = n$, $\alpha \ge 1$ and $A^{(7)}$ has the block form
\begin{eqnarray*}
A^{(7)} = \left(\begin{array}{ccc}
1 & & \bar{a}_{1,2}\\
\bar{a}^{(7)}_{2,1} & & A^{(7)}_{2,2}
\end{array}\right), \; A^{(7)}_{2,2}\in{\mathscr R}^{(n-1)\times(n-1)}.
 \end{eqnarray*}
 Then one of the following two mutually exclusive relations holds:

\begin{itemize}
\item[(i)] $A^{(7)}$ is an irreducible Z-matrix.

\item[(ii)] $A^{(7)}$ is a reducible Z-matrix,
but $A^{(7)}_{2,2}$ is an irreducible Z-matrix and
$a_{k,1} = a^{(7)}_{k,1} = a^{(7)}_{n,1} = 0$, $k = 2,\cdots,n-1$.
\end{itemize}
\end{lemma}

\begin{proof}
Since
\begin{eqnarray} \label{eqs3.29}
a^{(7)}_{i,j} = \left\{\begin{array}{lll}
a_{i,j}, & i = 1,\cdots n-1, j = 1,\cdots n,\\
(1 - \frac{1}{\alpha})a_{n,1}, & i = n, j = 1, \\
a_{n,j} - \frac{1}{\alpha}a_{n,1}a_{1,j}, & i = n, j = 2,\cdots, n,
\end{array}\right.
\end{eqnarray}
then it is clearly that, $a^{(7)}_{i,j} = a_{i,j} \le 0$ for $i = 1,\cdots,n-1$,
$j = 1,\cdots,n$, $i\not = j$,
$a^{(7)}_{n,1} = (1 - 1/\alpha)a_{n,1} \le 0$ and
\begin{eqnarray} \label{eqs3.30}
a^{(7)}_{n,j} = a_{n,j} - \frac{1}{\alpha}a_{n,1}a_{1,j} \le a_{n,j} \le 0, \;
j = 2,\cdots, n.
\end{eqnarray}
Hence, $A^{(7)}$ and $A^{(7)}_{2,2}$ are Z-matrices.

Let
\begin{eqnarray*}
A = \left(\begin{array}{ccc}
\hat{a}_{1,1} & & \hat{A}_{1,2}\\
a_{n,1} & & \hat{a}_{2,2}
\end{array}\right), \;
\hat{A} = \left(\begin{array}{ccc}
\hat{a}_{1,1} & & \hat{A}_{1,2}\\
0 & & \hat{a}_{2,2}
\end{array}\right), \;
\hat{A}_{1,2}\in{\mathscr R}^{(n-1)\times(n-1)}.
 \end{eqnarray*}
Clearly, if $\hat{A}$ is irreducible then $A^{(7)}$ is irreducible,
since $A$ is irreducible.
When $\alpha > 1$ then $A^{(7)}$ is also irreducible.

Assume that $A^{(7)}$ is reducible. Then $\hat{A}$ is reducible and $\alpha = 1$.
The latter implies $a^{(7)}_{n,1} = 0$. In this case, there must be $i^*,j^*\in\{1,\cdots,n\}$
such that there is no path from $i^*$ to $j^*$ in $G(A^{(7)})$.

It is easy to see that, for any
$i,j\in\{1,\cdots,n\}$, $\sigma_{i,j} = (j_0, j_1, \cdots, j_{l+1}) \notin G(A^{(7)})$
with $i = j_0$ and $j = j_{l+1}$ if and only if there is $s\in\{0,\cdots,l\}$ such that
$i_s = n$ and $i_{s+1} = 1$.

Since $A$ is irreducible, then there exists $\sigma_{i^*,j^*} = (i_0, i_1, \cdots, i_{t+1})\in G(A)$
with $i^* = i_0$ and $j^* = i_{t+1}$.

If $j^* > 1$ then there is $\mu\in\{0,\cdots,t-1\}$ such that
$i_\mu = n$ and $i_{\mu+1} = 1$, which implies $a_{1,i_{\mu+2}} < 0$.
Hence, by (\ref{eqs3.30}), we have
$a^{(7)}_{n,i_{\mu+2}} = a_{n,j} - a_{n,1}a_{1,i_{\mu+2}} \le -a_{n,1}a_{1,i_{\mu+2}} < 0$,
so that $\sigma = (i_0, \cdots, i_{\mu-1},n,i_{\mu+2},\cdots, i_{t+1}) \in G(A^{(7)})$.
This is a contradiction. Therefore, $j^* = 1$.
This shows that, for any $k\in\{2,\cdots,n-1\}$, there exists $\sigma_{i^*,k}\in G(A^{(7)})$.
If $a^{(7)}_{k,1} < 0$ then $\tilde{\sigma}_{i^*,j^*} = (\sigma_{i^*,k}, 1)\in G(A^{(7)})$.
This is also a contradiction. Therefore, $a^{(7)}_{k,1} = a_{k,1} = 0$.

Now, let us prove the irreducibility of $A^{(7)}_{2,2}$.
For any $i,j\in\{2, \cdots, n\}$, $i\not = j$, since $A$ is irreducible, then there exists a path
$\sigma_{i,j} = (\tau_0, \tau_1, \cdots, \tau_{\upsilon+1})\in G(A)$
with $i = \tau_0$ and $j = \tau_{\upsilon+1}$.

If $\tau_k\in\{2, \cdots, n\}$ for $k = 1, \cdots, \upsilon$, then it gets that
$\sigma_{i,j}\in G(A^{(7)}_{2,2})$.

For the case when there exists $s\in\{1, \cdots, \upsilon\}$ such that $\tau_s = 1$, we have
$a_{\tau_{s-1},1} < 0$ and $a_{1,\tau_{s+1}} < 0$. By (\ref{eqs3.29}), $\tau_{s-1}$ must be $n$,
since $a_{k,1} = a^{(7)}_{k,1} = 0$ for $k = 2,\cdots,n-1$.
By (\ref{eqs3.30}), it gets that $a^{(7)}_{n,\tau_{s+1}} = a_{n,\tau_{s+1}}
- a_{n,1}a_{1,\tau_{s+1}} \le -a_{n,1}a_{1,\tau_{s+1}} < 0$. This shows that $\tilde{\sigma}_{i,j} =
(\tau_0, \cdots, \tau_{s-2}, n, \tau_{s+1},  \cdots, \tau_{\upsilon+1})\in G(A^{(7)}_{2,2})$.

This has proved that $ G(A^{(7)}_{2,2})$ is irreducible.

We have proved ($ii$).
\end{proof}

Using this lemma, completely similar to the proof of Theorem \ref{thm3.23}, we can prove the following theorem.

\begin{theorem}\label{thm3.27}
Suppose that $\alpha > a_{r,s}a_{s,r}$.
Then Theorem C is valid for $\nu = 7$,
provided one of the following conditions is satisfied:

\begin{itemize}
\item[(i)] $\alpha \gtrsim 1$. And one of the following conditions holds:

\begin{itemize}
\item[($i_1$)] $0 \le \gamma < 1$.

\item[($i_2$)] $\gamma = 1$ and there exists $k\in\{s+1, \cdots, r\}$ such that $a_{s,k} < 0$.
\end{itemize}

\item[(ii)] $r = n$, $s = 1$ and $\alpha \ge 1$.
\end{itemize}
\end{theorem}

The result when ($ii$) holds includes \cite[Theorems 3.8 and 3.9]{Yu11} and
 \cite[Theorem 1]{LLW072}, where the proof is insufficient, which is pointed out by \cite{Yu08}.
We also have to point out that there exist some mistakes in \cite[Theorems 4 and 5]{HCEC05}.
For $\alpha = 1$, the result is better than the corresponding one given by
\cite[Theorem 2.1, Corollaries 2.1, 2.2]{LW04},
where the condition $a_{k,k+1}a_{k+1,k} > 0$, $k = 1, \cdots, n-1$, implies that
$A$ is irreducible and so that the condition $a_{1,n}a_{n,1} > 0$ is unnecessary.

From Theorem \ref{thm3.27}, we can prove the following theorem.

\begin{theorem}
Theorem D is valid for $\nu = 7$, provided one of the conditions ($i$) and ($ii$)
of Theorem \ref{thm3.27} is satisfied.
\end{theorem}

This theorem when ($ii$)
in Theorem \ref{thm3.27} holds is better than \cite[Theorem 2.2]{EMT01} and
the corresponding one given in \cite[Theorem 3.1, Corollaries 3.1, 3.2, 3.3]{Li02}. In \cite[Theorem 2.2]{EMT01},
the condition $a_{k,k+1}a_{k+1,k} > 0$, $k = 1, \cdots, n-1$, implies that
$A$ is irreducible and $\rho(\mathscr{L}_{\gamma,\omega}) < 1$ implies that $A$ is a nonsingular M-matrix.
The condition $a_{n,1}a_{1,n} > 0$ is unnecessary.

In \cite{LCC08}, for the preconditioned Gauss-Seidel method, a special case of
the matrix $Q_3$ is proposed as
\begin{eqnarray*}
Q_8 = \left(\begin{array}{ccccc}
0 & 0 & \cdots & 0 & 0\\
-\alpha_1a_{2,1} & 0 & \cdots & 0 & 0\\
0 & -\alpha_2a_{3,2} &\ddots & 0 & 0 \\
\vdots & \vdots & \ddots & \ddots & \vdots\\
0 & 0 & \cdots & -\alpha_{n-1}a_{n,n-1} & 0
\end{array}\right)
 \end{eqnarray*}
with $\alpha_k \ge 0$, $k = 1, \cdots, n-1$, and
\begin{eqnarray*}
\sum\limits_{k = 1}^{n-1}\alpha_ka_{k+1,k} \not = 0.
 \end{eqnarray*}

It is used to the preconditioned AOR method in \cite{HDS10,Li12}.

In this case, for $i = 2, \cdots, n$,
\begin{eqnarray*}
\sum\limits_{k = 1}^{i-1}\alpha_{i,k}a_{i,k}a_{k,i} = \alpha_{i-1}a_{i,i-1}a_{i-1,i}.
 \end{eqnarray*}

Now, $\delta_{i,j}^{(3)}(1)$ reduces to
\begin{eqnarray*}
\delta_{i,j}^{(8)}(1) = \left\{\begin{array}{ll}
 \alpha_{i-1}a_{i,i-1}a_{i-1,i}, & i = j = 2,\cdots,n; \\
 0, & otherwise.
\end{array}\right.
\end{eqnarray*}

When $n \ge 3$, the conditions ($ii_2$) and ($iii_2$) in Theorem \ref{thm3.11} can be not satisfied.
By Corollaries \ref{coro3-13} and \ref{coro3-14}, it is easy to prove the following comparison results.

\begin{theorem}
Suppose that $0 \le \alpha_k \le 1$ and $\alpha_k a_{k,k+1}a_{k+1,k} < 1$, $k = 1, \cdots, n-1$.
Then Theorem A is valid for $\nu = 8$.
\end{theorem}

\begin{theorem}
Suppose that $0 \le \alpha_k \le 1$, $k = 1, \cdots, n-1$.
Then Theorem B is valid for $\nu = 8$.
\end{theorem}

In order to give the Stein-Rosenberg Type Theorem II,
we prove a lemma.

\begin{lemma}\label{lem3-44}
Let $A$ be a Z-matrix. Then $A$ and $A^{(8)}$ are irreducible Z-matrices,
provided one of the following conditions is satisfied:

\begin{itemize}
\item[(i)] $a_{k,k+1}a_{k+1,k} > 0$, $0 < \alpha_k \le 1$,
$k = 1, \cdots, n-1$.

\item[(ii)] $n \ge 3$, $a_{n,1} < 0$, $a_{k,k+1} < 0$, $0 \le \alpha_k \le 1$,
$k = 1, \cdots, n-1$.
\end{itemize}
\end{lemma}

\begin{proof}
The condition ($i$) implies
$a_{k,k+1} < 0$ and $a_{k+1,k} > 0$, $k = 1, \cdots, n-1$. Hence, if one of ($i$) and
($ii$) holds, then it is easy to prove that $A$ is irreducible.

Since
\begin{eqnarray} \label{eqs3.31}
a^{(8)}_{i,j} = \left\{\begin{array}{lll}
a_{1,j} \le 0, & i = 1, j = 2,\cdots n,\\
(1 - \alpha_{i-1})a_{i,i-1} \le 0, & i = 2,\cdots, n, j = i-1, \\
a_{i,j} - \alpha_{i-1}a_{i,i-1}a_{i-1,j} \le a_{i,j} \le 0, & i = 2,\cdots, n, j = 1,\cdots, n, \\
 & j \not= i,i-1,
\end{array}\right.
\end{eqnarray}
then $A^{(8)}$ is a Z-matrix.

When ($i$) holds, for any $i,j\in\{1, \cdots, n\}$, $i\not = j$, since $A$ is irreducible, then there exists a path
$\sigma_{i,j} = (j_{(0)}, j_{(1)}, \cdots, j_{(l+1)})\in G(A)$ with $i = j_{(0)}$ and $j = j_{(l+1)}$.

By (\ref{eqs3.31}), it is obviously that either if there is no $j_{(k+1)} = j_{(k)}-1$,
$k \in \{0, 1, \cdots, l\}$, or if there exists some $s \in \{0, 1, \cdots, l\}$
such that $j_{(s+1)} = j_{(s)}-1$ but $\alpha_{j_{(s)}-1} < 1$, then $\sigma_{i,j}\in G(A^{(8)})$.

For the case when $j_{(s+1)} = j_{(s)}-1$ and $\alpha_{j_{(s)}-1} = 1$ for some $s \in \{0, 1, \cdots, l\}$,
then $\sigma_{i,j}\notin G(A^{(8)})$, since $a^{(8)}_{j_{(s)},j_{(s+1)}} = a^{(8)}_{j_{(s)},j_{(s)}-1} = 0$.
If $j_{(s)} < n$, then it gets that $a^{(8)}_{j_{(s)},j_{(s)}+1} =
a_{j_{(s)},j_{(s)}+1} - a_{j_{(s)},j_{(s)}-1}a_{j_{(s)}-1,j_{(s)}+1}
\le a_{j_{(s)},j_{(s)}+1} < 0$ and $a^{(8)}_{j_{(s)}+1,j_{(s)}-1} = a_{j_{(s)}+1,j_{(s)}-1}
 - \alpha_{j_{(s)}}a_{j_{(s)}+1,j_{(s)}}a_{j_{(s)},j_{(s)}-1}
\le $\\ $-\alpha_{j_{(s)}}a_{j_{(s)}+1,j_{(s)}}a_{j_{(s)},j_{(s)}-1} < 0$.
While, if $j_{(s)} = n$, then $j_{(s+1)} = n-1$. In this case it gets that $a^{(8)}_{n,n-2}
= a_{n,n-2} - a_{n,n-1}a_{n-1,n-2} \le -a_{n,n-1}a_{n-1,n-2} < 0$ and $a^{(8)}_{n-2,n-1}
= a_{n-2,n-1} - \alpha_{n-3}a_{n-2,n-3}a_{n-3,n-1} \le a_{n-2,n-1} < 0$.
Now we can construct a path
$\sigma^{(1)}_{i,j} = (j_{(0)}, \cdots, j_{(s)}, j_{(s)}+1, j_{(s+1)}, \cdots,  j_{(l+1)})$ whenever $j_{(s)} < n$
or $\sigma^{(1)}_{i,j} = (j_{(0)}, \cdots, j_{(s-1)}, n, n-2, n-1, j_{(s+2)}, \cdots, j_{(l+1)})$
whenever $j_s = n$. To continue this process, we can eventually construct a path
$\sigma^{(t)}_{i,j}$, $t \le l$, such that $\sigma^{(t)}_{i,j}\in G(A^{(8)})$.
We have proved that $A^{(8)}$ is irreducible.

When ($ii$) holds, then by (\ref{eqs3.31}) we can obtain
$a^{(8)}_{1,2} = a_{1,2} < 0$, $a^{(8)}_{k,k+1} \le
a_{k,k+1} - \alpha_{k-1}a_{k,k-1}a_{k-1,k+1} \le a_{k,k+1} < 0$,
$k = 2, \cdots, n-1$, $a^{(8)}_{n,1} \le a_{n,1} - \alpha_{n-1}a_{n,n-1}a_{n-1,1} \le a_{n,1} < 0$.
From this it is easy to see that $A^{(8)}$ is irreducible.
\end{proof}

\begin{theorem}\label{thm3.31}
Suppose that $\alpha_ka_{k,k+1}a_{k+1,k} < 1$, $k = 1, \cdots, n-1$.
Then Theorem C is valid for $\nu = 8$,
provided one of the following conditions is satisfied:

\begin{itemize}
\item[(i)]
$0 \le \alpha_k \lesssim 1$, $k = 1, \cdots, n-1$.
And one of the following conditions holds:

\begin{itemize}
\item[($i_1$)] $0 \le \gamma < 1$.

\item[($i_2$)] $\gamma = 1$ and there exists $k\in\{1, \cdots, n-1\}$ such that
$\alpha_{k}a_{k,k+1}a_{k+1,k}$ $ > 0$.

\item[($i_3$)] $\gamma = 1$ and $a_{k,k+1} < 0$, $k = 1, \cdots, n-1$.
\end{itemize}

\item[(ii)] $a_{k,k+1}a_{k+1,k} > 0$ and
$0 < \alpha_k \le 1$, $k = 1, \cdots, n-1$.

\item[(iii)] $n \ge 3$, $a_{n,1} < 0$, $a_{k,k+1} < 0$, $0 \le \alpha_k \le 1$,
$k = 1, \cdots, n-1$.
\end{itemize}
\end{theorem}

\begin{proof} By ($i$) of Corollary \ref{coro3-15}, ($i$), ($ii_1$) and ($ii_3$)
in Theorem \ref{thm3.11}, ($i$) follows directly.

If ($ii$) holds, then by Lemma \ref{lem3-44} $A^{(8)}$ is an irreducible L-matrix.
Now, for $0 \le \gamma < 1$ the condition ($i$) of Theorem \ref{thm3.3} is satisfied.
By the definition of $Q_8$, there exists some
$k_0 \in \{1,\cdots, n-1\}$ such that $\alpha_{k_0} > 0$. Hence, for $\gamma = 1$,
we can prove that ($ii_4$) in Theorem \ref{thm3.3}
holds, since $\alpha_{k_0}a_{k_0+1,k_0}a_{k_0,k_0+1} > 0$.

When ($iii$) holds, the proof is completely same.
\end{proof}

Obviously, from Lemma \ref{lem3-44},
if ($ii$) or ($iii$) holds, then the assumption that
$A$ is irreducible is redundant.

This theorem when ($ii$) holds is better than the results in \cite{Li12}.
The corresponding result in \cite[Theorem 3.2]{HDS10} is problematic, because
\cite[Lemma 3.1]{HDS10} is wrong. In fact, Let
 \begin{eqnarray*}
A = \left(\begin{array}{rrrr}
1    & -0.5 & 0    & -1 \\
-1   & 1    & 0    & 0   \\
0    & -1   & 1    & 0   \\
0    & 0    & -1   & 1
\end{array}\right).
 \end{eqnarray*}
Then it is easy to prove that $A$ is an irreducible L-matrix and it satisfies
the assumption of \cite[Lemma 3.1]{HDS10}. But
the iteration matrices of the preconditioned AOR methods
are reducible when we choose $\alpha_3=1$.

The following result is easy to prove.

\begin{theorem}
Theorem D is valid for $\nu = 8$, provided one of the conditions ($i$), ($ii$) and ($iii$)
of Theorem \ref{thm3.31} is satisfied.
\end{theorem}

Different from $Q_5$, a special $Q$ is proposed as
\begin{eqnarray*}
Q_{9} = \left(\begin{array}{ccccccccc}
0     & 0      & \cdots & 0  & 0  \\
\vdots & \vdots &  \ddots & \vdots & \vdots \\
0      & 0      &\cdots & 0 & 0  \\
-\alpha_1 a_{n,1}+\beta_1  & -\alpha_2 a_{n,2}+\beta_2  & \cdots & -\alpha_{n-1} a_{n,n-1}+\beta_{n-1} & 0
\end{array}\right)
 \end{eqnarray*}
 with $\alpha_k \ge 0$, $-\alpha_k a_{n,k}+\beta_k \ge 0$, $k = 1, \cdots, n-1$, and
 $\sum_{k = 1}^{n-1}(-\alpha_k a_{n,k}+\beta_k) \not = 0$. For
 $\alpha_k = \alpha \ge 0$, $\beta_k = \beta \ge 0$ and $\alpha + \beta \not = 0$, it is given in \cite{HS14}.

In this case, we have that
\begin{eqnarray*}
a^{(9)}_{i,j} & = & a_{i,j}, \; i = 1, \cdots, n-1, \; j = 1, \cdots, n,\\
a^{(9)}_{n,j} & = & (1-\alpha_j)a_{n,j} + \beta_j + \sum\limits_{{k = 1}\atop{k\not = j}}^{n-1}(-\alpha_k a_{n,k} + \beta_k)a_{k,j}, \; j = 1,\cdots, n-1,\\
a^{(9)}_{n,n} & = & 1 + \sum\limits_{k = 1}^{n-1}(-\alpha_k a_{n,k} + \beta_k)a_{k,n},
\end{eqnarray*}
so that
\begin{eqnarray*}
 q^{(9)}_{i,j} & = & q^{(9)}_{n,n} = 0, \; i = 1, \cdots, n-1, \; j = 1, \cdots, n,\\
 q^{(9)}_{n,j} & = & -\alpha_j a_{n,j} + \beta_j, \; j = 1, \cdots, n-1
\end{eqnarray*}
and
\begin{eqnarray*}
\sum\limits_{{k = 1}\atop{k\not = i}}^nq^{(9)}_{i,k}a_{k,i} = \left\{\begin{array}{ll}
 0, & i = 1, \cdots, n-1,\\
 \sum\limits_{k = 1}^{n-1}(-\alpha_{k}a_{n,k} + \beta_k)a_{k,n}, & i = n.
\end{array} \right.
\end{eqnarray*}

Now, $\delta_{i,j}^{(1)}(1)$ reduces to
\begin{eqnarray*}
\delta_{i,j}^{(9)}(1) = \left\{\begin{array}{ll}
 \sum\limits_{k = 1}^{j-1} (\alpha_ka_{n,k} - \beta_k)a_{k,j}, &
 i = n, j = 2,\cdots,n;\\
 0, & otherwise.
\end{array}\right.
\end{eqnarray*}

Hence, by Corollaries \ref{coro3-1}-\ref{coro3-4}, we can prove the following comparison theorems
 directly, where the proof of Theorem \ref{thm3.35} is similar to that of Theorem \ref{thm3.19}.

\begin{theorem}\label{thm3.33}
Suppose that $\sum_{k = 1}^{n-1}(\alpha_ka_{n,k} - \beta_k)a_{k,n} < 1$ and
 \begin{eqnarray} \label{eqs3.32}
&&(1-\alpha_j)a_{n,j} + \beta_j + \sum\limits_{{k = 1}\atop{k\not = j}}^{n-1}(-\alpha_ka_{n,k} + \beta_k)a_{k,j} \le 0,\\
&& j = 1,\cdots, n-1.\nonumber
\end{eqnarray}
Then Theorem A is valid for $\nu = 9$.
\end{theorem}

\begin{theorem}\label{thm3.34}
Suppose that (\ref{eqs3.32}) holds.
Then Theorem B is valid for $\nu = 9$.
\end{theorem}

The results given by Theorems \ref{thm3.33} and \ref{thm3.34} include the corresponding
ones given in \cite[Theorem 2.3]{HS14},
where $1\le j\le n$ should be $1\le j\le n-1$.

\begin{theorem}\label{thm3.35}
Suppose that $\sum_{k = 1}^{n-1}(\alpha_ka_{n,k} - \beta_k)a_{k,n} < 1$ and
\begin{eqnarray}\label{eqs3.34}\nonumber
&& (1-\alpha_j)a_{n,j} + \beta_j + \sum\limits_{{k = 1}\atop{k\not = j}}^{n-1}(-\alpha_ka_{n,k} + \beta_k)a_{k,j}
 \left\{\begin{array}{ll}
 \le 0,\\
 \lesssim 0 \; \hbox{whenever} \; a_{n,j} < 0,
\end{array}\right. \\
&& j = 1,\cdots, n-1.
\end{eqnarray}
Then Theorem C is valid for $\nu = 9$,
provided one of the following conditions is satisfied:

\begin{itemize}
\item[(i)] $0 \le \gamma < 1$.

\item[(ii)] $\gamma = 1$ and one of the following conditions holds:

\begin{itemize}
\item[($ii_1$)] There exist $i\in\{2, \cdots, n\}$ and $j\in\{1, \cdots, i-1\}$ such that
$(\alpha_ja_{n,j} - \beta_j)a_{j,i} > 0$.

\item[($ii_2$)] $\alpha_1a_{n,1} - \beta_1 < 0$.

\item[($ii_3$)] $a_{k,k+1} < 0$, $k = 1, \cdots, n-1$.

\item[($ii_4$)] $a_{k,n} < 0$, $k = 1, \cdots, n-1$.
\end{itemize}\end{itemize}
\end{theorem}

\begin{theorem}\label{thm3.36}
Suppose that (\ref{eqs3.34}) holds.
Then Theorem D is valid for $\nu = 9$,
provided one of the conditions ($i$) and ($ii$)
of Theorem \ref{thm3.35} is satisfied.
\end{theorem}

From Theorems \ref{thm3.33}-\ref{thm3.36}, the following results are directly.

\begin{corollary}
Suppose that $\sum_{k = 1}^{n-1}(\alpha_ka_{n,k} - \beta_k)a_{k,n} < 1$ and $(1-\alpha_j)a_{n,j} + \beta_j \le 0$,
$j = 1,\cdots, n-1$.
Then Theorem A is valid for $\nu = 9$.
\end{corollary}

\begin{corollary}
Suppose that $(1-\alpha_k)a_{n,k} + \beta_k \le 0$, $k = 1,\cdots, n-1$.
Then Theorem B is valid for $\nu = 9$.
\end{corollary}

\begin{corollary}
Suppose that $\sum_{k = 1}^{n-1}(\alpha_ka_{n,k} - \beta_k)a_{k,n} < 1$ and
\begin{eqnarray}\label{eqs3.35}
(1-\alpha_j)a_{n,j} + \beta_j
 \left\{\begin{array}{ll}
 \le 0,\\
 \lesssim 0 \; \hbox{whenever} \; a_{n,j} < 0,
\end{array}\right.
\; j = 1,\cdots, n-1.
\end{eqnarray}
Then Theorem C is valid for $\nu = 9$,
provided one of the conditions ($i$) and ($ii$)
of Theorem \ref{thm3.35} is satisfied.
\end{corollary}

\begin{corollary}
Suppose that (\ref{eqs3.35}) holds.
Then Theorem D is valid for $\nu = 9$, provided one of the conditions ($i$) and ($ii$)
of Theorem \ref{thm3.35} is satisfied.
\end{corollary}

Similarly, different from $Q_6$, a special $Q$ is proposed in \cite{LL12} as
\begin{eqnarray*}
Q_{10} = \left(\begin{array}{ccccc}
0 && 0 & \quad \cdots \quad & 0\\
 -\alpha_2 a_{2,1}+\beta_2 && 0 & \cdots & 0\\
\vdots && \vdots & \ddots & \vdots\\
-\alpha_n a_{n,1}+\beta_n && 0 & \cdots & 0\\
\end{array}\right)
 \end{eqnarray*}
 with $\alpha_k \ge 0$, $-\alpha_k a_{k,1}+\beta_k \ge 0$, $k = 2, \cdots, n$, and
$\sum_{k = 2}^{n}(-\alpha_k a_{k,1}+\beta_k) \not = 0$.

It is given in \cite{DH14} for $\alpha_k = 1$,
$k = 2, \cdots, n$, in \cite{HS14} for $\alpha_k = \alpha \ge 0$, $\beta_k = \beta \ge 0$,
$k = 2, \cdots, n$ with $\alpha + \beta \not = 0$,
and in \cite{DH11} for the preconditioned SOR method, where $\alpha_k = 1$, $k = 2, \cdots, n$.

In this case, we have that
\begin{eqnarray*}
a^{(10)}_{i,j}  = \left\{\begin{array}{ll}
 a_{1,j}, & i =1, j = 1, \cdots, n,\\
 (1- \alpha_i)a_{i,1} + \beta_i, & i = 2,\cdots,n, j=1,\\
 a_{i,j} + (-\alpha_i a_{i,1} + \beta_i)a_{1,j}, & i,j = 2, \cdots, n,
 \end{array}\right.
\end{eqnarray*}
so that
\begin{eqnarray*}
 q^{(10)}_{i,1} & = & -\alpha_i a_{i,1} + \beta_i, \; i = 2, \cdots, n,\\
 q^{(10)}_{i,j} & = & q^{(10)}_{1,1} = 0, \; i = 1, \cdots, n, \; j = 2, \cdots, n
\end{eqnarray*}
and
\begin{eqnarray*}
\sum\limits_{{k = 1}\atop{k\not = i}}^nq^{(10)}_{i,k}a_{k,i} = \left\{\begin{array}{ll}
 0, & i = 1, \\
 (-\alpha_{i}a_{i,1} + \beta_i)a_{1,i}, & i = 2,\cdots,n.
\end{array} \right.
\end{eqnarray*}

Now, $\delta_{i,j}^{(1)}(1)$ reduces to
\begin{eqnarray*}
\delta_{i,j}^{(10)}(1) = \left\{\begin{array}{ll}
(\alpha_ia_{i,1} - \beta_i)a_{1,j}, &
 i = 2,\cdots,n, j = 2,\cdots,i;\\
 0, & otherwise.
\end{array}\right.
\end{eqnarray*}

Hence, by Corollaries \ref{coro3-5}-\ref{coro3-8}, we can prove the following comparison theorems directly.

\begin{theorem}\label{thm3.37}
Suppose that
$(1-\alpha_k)a_{k,1} + \beta_k \le 0$ and $(\alpha_k a_{k,1} - \beta_k)a_{1,k} < 1$,
 $k = 2, \cdots, n$.
Then Theorem A is valid for $\nu = 10$.
\end{theorem}

\begin{theorem} \label{thm3.38}
Suppose that
$(1-\alpha_k)a_{k,1} + \beta_k \le 0$, $k = 2, \cdots, n$.
Then Theorem B is valid for $\nu = 10$.
\end{theorem}

The results given by Theorems \ref{thm3.37} and \ref{thm3.38} include the corresponding
ones given in \cite[Theorem 3.3]{HS14}.

The result given by \ref{thm3.38} is better than the corresponding one given in \cite[Theorem 4.2]{LL12},
where the condition that $A$ is irreducible is unnecessary.

In order to give the Stein-Rosenberg Type Theorem II,
we give a lemma, whose proof is completely same as that of Lemma \ref{lem3-4}.

\begin{lemma}\label{lem3-6}
Let $A$ be an irreducible Z-matrix.
Assume that $0 < -\alpha_ka_{k,1}+ \beta_k \le -a_{k,1}$,
$k = 2, \cdots, n$ and $A^{(10)}$ has the block form
\begin{eqnarray*}
A^{(10)} = \left(\begin{array}{ccc}
1 & & \bar{a}_{1,2}\\
\bar{a}^{(10)}_{2,1} & & A^{(10)}_{2,2}
\end{array}\right), \; A^{(10)}_{2,2}\in{\mathscr R}^{(n-1)\times(n-1)}.
 \end{eqnarray*}
 Then

\begin{itemize}
\item[(i)] $A^{(10)}_{2,2}$ is an irreducible Z-matrix.

\item[(ii)] $A^{(10)}$ is an irreducible Z-matrix if and only if there exists
$k_0\in\{2, \cdots,  n\}$ such that $(1-\alpha_{k_0})a_{k_0,1} + \beta_{k_0}\not = 0$.
\end{itemize}
\end{lemma}

Using this lemma, similar to the proof of Theorem \ref{thm3.23} we prove the following theorem.

\begin{theorem}\label{thm3.39}
Suppose that
$(1-\alpha_k)a_{k,1} + \beta_k \le 0$
and $(\alpha_k a_{k,1} - \beta_k)a_{1,k} < 1$, $k = 2, \cdots, n$.
Then Theorem C is valid for $\nu = 10$,
provided one of the following conditions is satisfied:

\begin{itemize}
\item[(i)] $0 \le \gamma < 1$ and $(1-\alpha_k)a_{k,1} + \beta_k \lesssim 0$ whenever $a_{k,1} < 0$,
 $k = 2, \cdots, n$.

\item[(ii)] $\gamma = 1$ and $(1-\alpha_k)a_{k,1} + \beta_k \lesssim 0$ whenever $a_{k,1} < 0$,
 $k = 2, \cdots, n$.
And one of the following conditions holds:

\begin{itemize}
\item[($ii_1$)] There exist $i\in\{2, \cdots, n\}$ and $j\in\{2, \cdots, i\}$ such that
$(\alpha_{i}a_{i,1} - \beta_i)a_{1,j} > 0$.

\item[($ii_2$)] $-\alpha_{n}a_{n,1} + \beta_n > 0$.

\item[($ii_3$)] $a_{1,2} < 0$.
\end{itemize}

\item[(iii)] $-\alpha_ka_{k,1} + \beta_k > 0$, $k = 2, \cdots, n$.
\end{itemize}
\end{theorem}

\begin{proof}
By Theorem \ref{thm3.3}, we just need to prove ($ii_3$) and ($iii$).

For ($ii_3$), by the definition of $Q_{10}$, there exists
$i_0\in\{2, \cdots, n\}$ such that $-\alpha_{i_0}a_{i_0,1} + \beta_{i_0} > 0$, so that
$(\alpha_{i_0}a_{i_0,1} - \beta_{i_0})a_{1,2} > 0$,
which implies that ($ii_1$) holds for $i = i_0$ and $j = 2$.

For ($iii$), since $A$ is irreducible, then there exists $j\in\{2, \cdots, n\}$ such that $a_{1,j} < 0$,
so that $\delta_{n,j}^{(10)}(1) = (\alpha_na_{n,1} - \beta_n)a_{1,j} > 0$. Using Lemma \ref{lem3-6},
the rest proof is completely similar to that of ($iii$) and ($iv_1$) in Theorem \ref{thm3.23}.
\end{proof}

When ($iii$) holds it includes \cite[Theorem 3.1]{DH14}.
The result for $\gamma = \omega$ is better that the corresponding ones
given by \cite[Theorem 3.1, Corollary 3.1]{DH11}.

\begin{theorem}\label{thm3.40}
Suppose that
$(1-\alpha_k)a_{k,1} + \beta_k \le 0$, $k = 2, \cdots, n$.
Then Theorem D is valid for $\nu = 10$,
provided one of the conditions ($i$), ($ii$) and ($iii$)
of Theorem \ref{thm3.39} is satisfied.
\end{theorem}

As a special case of $Q_{9}$ and $Q_{10}$, $Q$ is proposed in \cite{WL09} as
\begin{eqnarray*}
Q_{11} = \left(\begin{array}{ccccccc}
0      && 0      && \cdots && 0   \\
\vdots && \vdots &&  \cdots && \vdots  \\
0      && 0      &&\cdots && 0   \\
-\frac{1}{\alpha}a_{n,1} - \beta && 0 && \cdots && 0
\end{array}\right)
 \end{eqnarray*}
with $\alpha > 0$ and $a_{n,1}/\alpha + \beta < 0$. It is discussed in \cite{LW042}
for $\alpha = 1$.

Now, $\delta_{i,j}^{(10)}(1)$ reduces to
\begin{eqnarray*}
\delta_{i,j}^{(11)}(1) = \left\{\begin{array}{ll}
 (\frac{1}{\alpha}a_{n,1} + \beta)a_{1,j}, \; & i = n, j = 2, \cdots, n; \\
 0, & otherwise.
\end{array}\right.
\end{eqnarray*}

By Theorems \ref{thm3.37} and \ref{thm3.38}, the following comparison results are obtained, directly.

\begin{theorem}\label{thm3.41}
Suppose that $\beta \ge (1 - 1/\alpha)a_{n,1}$
 and $(a_{n,1}/\alpha + \beta)a_{1,n} < 1$.
Then Theorem A is valid for $\nu = 11$.
\end{theorem}

\begin{theorem}
Suppose that $\beta \ge (1 - 1/\alpha)a_{n,1}$.
Then Theorem B is valid for $\nu = 11$.
\end{theorem}

In order to give the Stein-Rosenberg Type Theorem II,
completely similar to Lemma \ref{lem3-5}, we can prove the following lemma.

\begin{lemma} \label{lem3-7}
Let $A$ be an irreducible Z-matrix.
Assume that $\beta \ge (1 - 1/\alpha)a_{n,1}$ and $A^{(11)}$ has the block form
\begin{eqnarray*}
A^{(11)} = \left(\begin{array}{ccc}
1 & & \bar{a}_{1,2}\\
\bar{a}^{(11)}_{2,1} & & A^{(11)}_{2,2}
\end{array}\right), \; A^{(11)}_{2,2}\in{\mathscr R}^{(n-1)\times(n-1)}.
 \end{eqnarray*}
 Then one of the following two mutually exclusive relations holds:

\begin{itemize}
\item[(i)] $A^{(11)}$ is an irreducible Z-matrix.

\item[(ii)] $A^{(11)}$ is a reducible Z-matrix,
but $A^{(11)}_{2,2}$ is an irreducible Z-matrix and
$a_{k,1} = a^{(11)}_{k,1} = a^{(11)}_{n,1} = 0$, $k = 2,\cdots,n-1$.
\end{itemize}
\end{lemma}

Using this lemma, similar to the proof of Theorem \ref{thm3.27}, we prove the following theorem.

\begin{theorem}
Suppose that
$\beta \ge (1 - 1/\alpha)a_{n,1}$ and $(a_{n,1}/\alpha + \beta)a_{1,n} < 1$.
Then Theorem C is valid for $\nu = 11$.
\end{theorem}

\begin{proof} Since $(a_{n,1}/\alpha + \beta)a_{1,n} < 1$, then, by Lemma \ref{lem3-7},
$A^{(11)}$ is an L-matrix.

If $A^{(11)}$ is irreducible then it follows that by ($i$) and ($ii_2$) in Theorem \ref{thm3.3}
that Theorem C is valid, since $q^{(11)}_{n,1} = -a_{n,1}/\alpha - \beta > 0$.

For the case when $A^{(11)}$ is reducible, then the irreducibility of $A$ ensures that
there exists $j\in\{2, \cdots, n\}$ such that $a_{1,j} < 0$ so that $-(a_{n,1}/\alpha + \beta)a_{1,j} > 0$.
Now, using Lemma \ref{lem3-7}, the rest proof
is completely similar to that of ($iii$) and ($iv_1$) in Theorem \ref{thm3.23}.
\end{proof}

The result is better than \cite[Theorem 1]{WL09}, where the condition
$0 < a_{1,n}a_{n,1} < \alpha (\alpha > 1)$ is unnecessary.
For $\alpha = 1$, it also better than the corresponding one given by
\cite[Theorem 8, Corollaries 10, 11]{LW042}.

\begin{theorem}
Suppose that $\beta \ge (1 - 1/\alpha)a_{n,1}$.
Then Theorem D is valid for $\nu = 11$.
\end{theorem}

By the definition of $Q_{11}$, $a_{n,1}/\alpha + \beta < 0$. While
in the comparison theorems above we need the condition
$\beta \ge (1 - 1/\alpha)a_{n,1}$. Hence it implies $a_{n,1}<0$.

\subsection{Upper triangular preconditioners}

Corresponding to $Q_3$, we let
\begin{eqnarray*}
\alpha_{i,j} = 0, \; i = 1, \cdots, n, j \le i.
\end{eqnarray*}
Then $Q_2$ reduces to
\begin{eqnarray*}
Q_{12} = \left(\begin{array}{ccccc}
0 & -\alpha_{1,2}a_{1,2} & -\alpha_{1,3}a_{1,3} & \cdots & -\alpha_{1,n}a_{1,n}\\
0 & 0 & -\alpha_{2,3}a_{2,3} & \cdots & -\alpha_{2,n}a_{2,n}\\
\vdots & \vdots & \ddots & \ddots & \vdots\\
0 & 0 & 0 & \cdots & -\alpha_{n-1,n}a_{n-1,n} \\
0 & 0 & 0 & \cdots & 0
\end{array}\right)
 \end{eqnarray*}
with $\alpha_{i,j} \ge 0$, $i = 1, \cdots, n-1$, $j > i$, and
\begin{eqnarray*}
\sum\limits_{i = 1}^{n-1}\sum\limits_{j = i+1}^n\alpha_{i,j}a_{i,j} \not = 0.
\end{eqnarray*}
In \cite{KNO972} it is proposed for the preconditioned Gauss-Seidel method,
where $\alpha_{i,j} = \alpha_{i} \ge 0$, $i = 1, \cdots, n-1$.

In this case, for $i,j = 1, \cdots, n$, we have that
\begin{eqnarray*}
\sum\limits_{{k = 1}\atop{k\not = i,j}}^n \alpha_{i,k}a_{i,k}a_{k,j}
 = \sum\limits_{{k = i+1}\atop{k\not = j}}^n\alpha_{i,k}a_{i,k}a_{k,j},
\end{eqnarray*}
so that if $j < i$, then
\begin{eqnarray*}
(1-\alpha_{i,j})a_{i,j} - \sum\limits_{{k = 1}\atop{k\not = i,j}}^n \alpha_{i,k}a_{i,k}a_{k,j}
 = a_{i,j} - \sum\limits_{k = i+1}^n\alpha_{i,k}a_{i,k}a_{k,j} \le a_{i,j}.
\end{eqnarray*}

Since $\alpha_{n,j} = 0$ for $j = 1, \cdots, n$, then $\delta_{i,j}^{(2)}(\gamma)$
and $\delta_{i,j}^{(2)}(1)$ reduce respectively to
\begin{eqnarray*}
\delta_{i,j}^{(12)}(\gamma) = \left\{\begin{array}{ll}
 \sum\limits_{k = i+1}^n \alpha_{i,k}a_{i,k}a_{k,i}, & i = j = 1, \cdots, n-1;\\
 (\gamma-1)\alpha_{i,j}a_{i,j} + \gamma\sum\limits_{k = i+1}^{j-1} \alpha_{i,k}a_{i,k}a_{k,j}, &
 \begin{array}{l} i = 1,\cdots,n-1, \\
  j = i+1,\cdots,n; \end{array}\\
 \gamma\sum\limits_{k = i+1}^n\alpha_{i,k}a_{i,k}a_{k,j}, &
  \begin{array}{l}  i = 2,\cdots,n-1, \\
 j = 1,\cdots, i-1; \end{array}\\
  0, & i = n, j = 1,\cdots, n
\end{array}\right.
\end{eqnarray*}
and
\begin{eqnarray*}
\delta_{i,j}^{(12)}(1) = \left\{\begin{array}{ll}
 \sum\limits_{k = i+1}^n\alpha_{i,k}a_{i,k}a_{k,j}, &
  i = 1,\cdots,n-1, j = 1,\cdots, i;\\
 \sum\limits_{k = i+1}^{j-1} \alpha_{i,k}a_{i,k}a_{k,j}, &
 i = 1,\cdots,n-1, j = i+2,\cdots,n;\\
 0, & i = 1,\cdots,n-1, j = i+1;\\
 0, & i = n, j = 1,\cdots, n.
\end{array}\right.
\end{eqnarray*}

In this case, the conditions ($ii_2$), ($ii_8$), ($iv_7$),
($iv^a$) and ($iv^b$) in Theorem \ref{thm3.7} can be not satisfied.
While the inequality (\ref{eqs3.15}) or (\ref{eqs3.16}) is trivial because that
$A$ is irreducible and
\begin{eqnarray*}
(1 - \alpha_{n,j})a_{n,j} - \sum\limits_{{k = 1}\atop{k\not = j}}^{n-1}\alpha_{n,k}a_{n,k}a_{k,j}
= a_{n,j}.
\end{eqnarray*}

Completely similar to Theorems \ref{thm3.9}-\ref{thm3.12} and Corollaries
\ref{coro3-13}-\ref{coro3-16}, by Theorems \ref{thm3.5}-\ref{thm3.8},
 we can prove the following comparison results, immediately.

\begin{theorem}\label{thm3.45}
Suppose that $\sum_{k = i+1}^n\alpha_{i,k}a_{i,k}a_{k,i} < 1$, $i = 1,\cdots,n-1$, and
\begin{eqnarray}\label{eqs3.36}
(1-\alpha_{i,j})a_{i,j} - \sum\limits_{{k = i+1}\atop{k\not = j}}^n \alpha_{i,k}a_{i,k}a_{k,j} \le 0,
 \; i = 1,\cdots,n-1, j > i.
\end{eqnarray}
Then Theorem A is valid for $\nu = 12$.
\end{theorem}

\begin{theorem}\label{thm3.46}
Suppose that (\ref{eqs3.36}) holds.
Then Theorem B is valid for $\nu = 12$.
\end{theorem}

\begin{theorem}\label{thm3.47}
Suppose that $\sum_{k = i+1}^n\alpha_{i,k}a_{i,k}a_{k,i} < 1$, $i = 1,\cdots,n-1$.
Then Theorem C is valid for $\nu = 12$,
provided one of the following conditions is satisfied:

\begin{itemize}
\item[(i)] $0 \le \gamma < 1$ and
\begin{eqnarray}\label{eqn3.377}\nonumber
&& (1 - \alpha_{i,j})a_{i,j} - \sum\limits_{{k = i+1}\atop{k\not = j}}^n \alpha_{i,k}a_{i,k}a_{k,j}
 \lesssim 0 \;\; \hbox{whenever} \;\; a_{i,j} < 0,\\
&& i = 1, \cdots, n-1, \; j > i.
\end{eqnarray}

\item[(ii)] $\gamma = 1$, (\ref{eqn3.377}) holds and
one of the following conditions holds:

\begin{itemize}
\item[($ii_1$)]
There exist $i\in\{1, \cdots, n-1\}$, $j\in\{1, \cdots, i\}$ and
$k\in\{i+1, \cdots, n\}$ such that
$\alpha_{i,k}a_{i,k}a_{k,j} > 0$.

\item[($ii_2$)]
There exist $i\in\{1, \cdots, n-1\}$, $j\in\{i+2, \cdots, n\}$ and
$k\in\{i+1, \cdots, j-1\}$ such that
$\alpha_{i,k}a_{i,k}a_{k,j} > 0$.

\item[($ii_3$)]
There exists $k\in\{1, \cdots, n-1\}$ such that
$a_{k,k+1} < 0$ and $\alpha_{k,k+1} > 0$.

\item[($ii_4$)]
$a_{n,1} < 0$ and $a_{k,k+1} < 0$, $k = 1, \cdots, n-1$.

\item[($ii_5$)]
$a_{n,1} < 0$ and $a_{k,n} < 0$, $k = 2, \cdots, n-1$.

\item[($ii_6$)]
$a_{k,1} < 0$, $k = 2, \cdots, n$.
\end{itemize}

\item[(iii)] $0 \le \gamma < 1$ and (\ref{eqs3.36}) holds. For each $i\in\{1, \cdots, n-1\}$,
there exists $j(i)\in\{i+1, \cdots, n\}$ such that $\alpha_{i,j(i)}a_{i,j(i)} < 0$.

\item[(iv)] $\gamma = 1$  and (\ref{eqs3.36}) holds. For each $i\in\{1, \cdots, n-1\}$,
one of the following conditions holds:

\begin{itemize}
\item[($iv_1$)] There exist $j(i)\in\{1, \cdots, i\}$ and
$k(i)\in\{i+1, \cdots, n\}$ such that\\
$\alpha_{i,k(i)}a_{i,k(i)}a_{k(i),j(i)} > 0$.

\item[($iv_2$)] There exist $j(i)\in\{i+2, \cdots, n\}$ and $k(i)\in\{i+1, \cdots, j-1\}$ such that
$\alpha_{i,k(i)}a_{i,k(i)}a_{k(i),j(i)} > 0$.

\item[($iv_3$)] $a_{i,i+1} < 0$ and $\alpha_{i,i+1} > 0$.
\end{itemize}\end{itemize}
\end{theorem}

\begin{proof}
By Theorem \ref{thm3.7}, we just need to prove ($ii_5$), ($ii_6$) and ($iv$).

By the definition of $Q_{12}$, there exist $i_0\in\{1, \cdots, n-1\}$ and
$j_0\in\{i+1, \cdots, n\}$ such that $\alpha_{i_0,j_0}a_{i_0,j_0} < 0$.

When ($ii_5$) holds, if $j_0 < n$ then
$\alpha_{i_0,j_0}a_{i_0,j_0}a_{j_0,n} > 0$, which implies that ($ii_2$) holds for $i = i_0$,
$j = n$ and $k = j_0$. If $j_0 = n$ then
$\alpha_{i_0,n}a_{i_0,n}a_{n,1} > 0$, which implies that ($ii_1$) holds for $i = i_0$,
$j = 1$ and $k = n$.

When ($ii_6$) holds, it gets that
$\alpha_{i_0,j_0}a_{i_0,j_0}a_{j_0,1} > 0$, which implies that ($ii_1$) holds for $i = i_0$,
$j = 1$ and $k = j_0$.

By the irreducibility of $A$, there exists $j \in \{1, \cdots, n-1\}$ such that
$a_{n,j} < 0$, which implies that ($iv^c$) or ($iv^d$) in Theorem \ref{thm3.7}
holds because $\alpha_{n,k}=0$, $k=1,\cdots, n-1$. By ($iv_1$) and ($iv_2$) in
Theorem \ref{thm3.7} we derive ($iv$).
\end{proof}

\begin{theorem}\label{thm3.48}
Theorem D is valid for $\nu = 12$,
provided one of the following conditions is satisfied:

\begin{itemize}
\item[(i)] One of the conditions ($i$)-($iv$) of Theorem \ref{thm3.47} holds.

\item[(ii)] The inequality (\ref{eqs3.36}) holds and one of the following conditions holds:

\begin{itemize}
\item[($ii_1$)] There exist $i\in\{1,\cdots,n-1\}$ and $j\in\{i+1,\cdots,n\}$ such that
$\alpha_{i,j}a_{i,j}a_{j,i} > 0$.

\item[($ii_2$)] $\gamma > 0$. There exist $i\in\{2,\cdots,n-1\}$, $j\in\{1,\cdots,i-1\}$
 and $k\in\{i+1,\cdots,n\}$ such that
$\alpha_{i,k}a_{i,k}a_{k,j} > 0$.
\end{itemize}\end{itemize}
\end{theorem}

\begin{proof}
We just need to prove ($ii$). In fact, since $\alpha_{i,j} = 0$ for $j \le i$,
then the result follows by ($ii$) of Theorem \ref{thm3.8} directly.
\end{proof}

\begin{corollary}\label{coro3-29}
Suppose that $0 \le \alpha_{i,j} \le 1$ and $\sum_{k = i+1}^n\alpha_{i,k}a_{i,k}a_{k,i} < 1$,
$i = 1,\cdots,n-1$, $j > i$.
Then Theorem A is valid for $\nu = 12$.
\end{corollary}

\begin{corollary}\label{coro3-30}
Suppose that $0 \le \alpha_{i,j} \le 1$, $i = 1,\cdots,n-1$, $j > i$.
Then Theorem B is valid for $\nu = 12$.
\end{corollary}

\begin{corollary}\label{coro3-31}
Suppose that $\sum_{k = i+1}^n\alpha_{i,k}a_{i,k}a_{k,i} < 1$, $i = 1,\cdots,n-1$.
Then Theorem C is valid for $\nu = 12$,
provided one of the following conditions is satisfied:

\begin{itemize}
\item[(i)]
For $i = 1,\cdots,n-1$, $j > i$, $0 \le \alpha_{i,j} \lesssim 1$.
One of the conditions $0\le\gamma<1$ and ($ii_1$)-($ii_6$) whenever $\gamma=1$ in
Theorem \ref{thm3.47} holds.

\item[(ii)] One of the conditions ($iii$) and ($iv$) of Theorem \ref{thm3.47} holds,
where the inequality (\ref{eqs3.36}) will be replaced by
$0 \le \alpha_{i,j} \le 1$, $i = 1,\cdots,n-1$, $j > i$.
\end{itemize}
\end{corollary}

\begin{corollary}\label{coro3-32}
Theorem D is valid for $\nu = 12$,
provided one of the following conditions is satisfied:

\begin{itemize}
\item[(i)] One of the conditions ($i$) and ($ii$) of Corollary \ref{coro3-31} holds.

\item[(ii)]
For $i = 1,\cdots,n-1$, $j > i$, $0 \le \alpha_{i,j} \le 1$.
One of the conditions ($ii_1$) and ($ii_2$) in Theorem \ref{thm3.48} holds.
\end{itemize}
\end{corollary}

Many known corresponding results about the preconditioned AOR method
proposed in the references are the special cases of Theorems \ref{thm3.45}-\ref{thm3.48} and
Corollaries \ref{coro3-29}-\ref{coro3-32}, i.e., they
can be derived from these theorems, immediately.

When $\alpha_{i,j} = \alpha > 0$, $i = 1, \cdots, n-1$, $j > i$,
the matrix $Q_{12}$ reduces to
\begin{eqnarray*}
Q_{13} = \alpha U,
\end{eqnarray*}
which is proposed in \cite{KNO96} for the preconditioned Gauss-Seidel method.
It is investigated in \cite{WWS07,YLK14} the preconditioned AOR method.
For $\alpha = 1$ it is proposed in \cite{UNK94}
for the preconditioned Gauss-Seidel method and in \cite{UKN95,MEY01} for the preconditioned SOR method.

Denote
\begin{eqnarray*}
\delta_{i,j}^{(13)}(\gamma) = \left\{\begin{array}{ll}
 \sum\limits_{k = i+1}^na_{i,k}a_{k,i}, & i = j = 1, \cdots, n-1;\\
 (\gamma-1)a_{i,j} + \gamma\sum\limits_{k = i+1}^{j-1} a_{i,k}a_{k,j}, &
 \begin{array}{l} i = 1,\cdots,n-1,\\
  j = i+1,\cdots,n; \end{array}\\
 \gamma\sum\limits_{k = i+1}^na_{i,k}a_{k,j}, &
  \begin{array}{l} i = 2,\cdots,n-1, \\
  j = 1,\cdots, i-1; \end{array}\\
  0, & i = n, j = 1,\cdots, n
\end{array}\right.
\end{eqnarray*}
and
\begin{eqnarray*}
\delta_{i,j}^{(13)}(1) = \left\{\begin{array}{ll}
\sum\limits_{k = i+1}^na_{i,k}a_{k,j}, & i = 1,\cdots,n-1, j = 1,\cdots, i; \\
 \sum\limits_{k = i+1}^{j-1}a_{i,k}a_{k,j}, & i = 1,\cdots,n-1, j = i+2,\cdots,n;\\
 0, & i = 1,\cdots,n-1, j = i+1;\\
 0, & i = n, j = 1,\cdots, n.
\end{array}\right.
\end{eqnarray*}

By Theorems \ref{thm3.45}-\ref{thm3.48} and
Corollaries \ref{coro3-29}-\ref{coro3-32}, the following results are obtained.

\begin{theorem}
Suppose that $\alpha\sum_{k = i+1}^na_{i,k}a_{k,i} < 1$, $i = 1,\cdots,n-1$, and
\begin{eqnarray}\label{eqs3.38}
 (1-\alpha)a_{i,j} - \alpha\sum\limits_{{k = i+1}\atop{k\not = j}}^n a_{i,k}a_{k,j} \le 0,
\; i = 1, \cdots, n-1, j > i.
\end{eqnarray}
Then Theorem A is valid for $\nu = 13$.
\end{theorem}

\begin{theorem}
Suppose that (\ref{eqs3.38}) holds.
Then Theorem B is valid for $\nu = 13$.
\end{theorem}

\begin{theorem} \label{thm3.51}
Suppose that $\alpha\sum_{k = i+1}^na_{i,k}a_{k,i} < 1$, $i = 1,\cdots,n-1$.
Then Theorem C is valid for $\nu = 13$,
provided one of the following conditions is satisfied:

\begin{itemize}
\item[(i)] $0 \le \gamma < 1$ and
\begin{eqnarray}\label{eqn3.555}\nonumber
&& (1 - \alpha)a_{i,j} - \alpha\sum\limits_{{k = i+1}\atop{k\not = j}}^n a_{i,k}a_{k,j}
 \lesssim 0 \;\; \hbox{whenever} \; a_{i,j} < 0, \\
 && i = 1, \cdots, n-1, j > i.
\end{eqnarray}

\item[(ii)] $\gamma = 1$. The inequality (\ref{eqn3.555}) holds and one of the following conditions holds:

\begin{itemize}
\item[($ii_1$)]
There exist $i\in\{1, \cdots, n-1\}$, $j\in\{1, \cdots, i\}$ and
$k\in\{i+1, \cdots, n\}$ such that
$a_{i,k}a_{k,j} > 0$.

\item[($ii_2$)]
There exist $i\in\{1, \cdots, n-1\}$, $j\in\{i+2, \cdots, n\}$ and $k\in\{i+1, \cdots, j-1\}$ such that
$a_{i,k}a_{k,j} > 0$.

\item[($ii_3$)]
There exists $k\in\{1, \cdots, n-1\}$ such that $a_{k,k+1} < 0$.

\item[($ii_4$)]
$a_{n,1} < 0$.
\end{itemize}

\item[(iii)] $0 \le \gamma < 1$ and the inequality (\ref{eqs3.38}) holds. For each $i\in\{1, \cdots, n-1\}$,
there exists $j(i)\in\{i+1, \cdots, n\}$ such that $a_{i,j(i)} < 0$.

\item[(iv)] $\gamma = 1$ and the inequality (\ref{eqs3.38}) holds. For each $i\in\{1, \cdots, n-1\}$
one of the following conditions holds:

\begin{itemize}
\item[($iv_1$)] There exist $j(i)\in\{1, \cdots, i\}$ and
$k(i)\in\{i+1, \cdots, n\}$ such that\\
$a_{i,k(i)}a_{k(i),j(i)} > 0$.

\item[($iv_2$)] There exist $j(i)\in\{i+2, \cdots, n\}$ and
$k(i)\in\{i+1, \cdots, j-1\}$ such that
$a_{i,k(i)}a_{k(i),j(i)} > 0$.

\item[($iv_3$)] $a_{i,i+1} < 0$.
\end{itemize}\end{itemize}
\end{theorem}

\begin{proof}
We just need to prove ($ii_4$).
From the irreducibility of $A$, there exists $i_0\in\{1, \cdots, n-1\}$ such that
$a_{i_0,n} < 0$ so that $a_{i_0,n}a_{n,1} > 0$, which implies that ($ii_1$) holds for $i = i_0$, $j = 1$ and
$k = n$.
 \end{proof}

\begin{theorem}\label{thm3.52}
Theorem D is valid for $\nu = 13$,
provided one of the following conditions is satisfied:

\begin{itemize}
\item[(i)] One of the conditions ($i$)-($iv$) of Theorem \ref{thm3.51} holds.

\item[(ii)] The inequality (\ref{eqs3.38}) holds and one of the following conditions holds:

\begin{itemize}
\item[($ii_1$)]
There exist $i\in\{1,\cdots,n-1\}$ and $j\in\{i+1,\cdots,n\}$ such that
$a_{i,j}a_{j,i} > 0$.

\item[($ii_1$)] $\gamma > 0$. There exist $i\in\{2,\cdots,n-1\}$, $j\in\{1,\cdots,i-1\}$
and $k\in\{i+1,\cdots,n\}$ such that $a_{i,k}a_{k,j} > 0$.
\end{itemize}\end{itemize}
\end{theorem}

\begin{corollary}
Suppose that $0 < \alpha \le 1$ and $\alpha\sum_{k = i+1}^na_{i,k}a_{k,i} < 1$, $i = 1,\cdots,n-1$.
Then Theorem A is valid for $\nu = 13$.
\end{corollary}

The result improves the corresponding one given by \cite[Theorem 3.6]{WWS07}.

\begin{corollary}
Suppose that $0 < \alpha \le 1$.
Then Theorem B is valid for $\nu = 13$.
\end{corollary}

\begin{corollary}\label{coro3-35}
Suppose that $\alpha\sum_{k = i+1}^na_{i,k}a_{k,i} < 1$, $i = 1,\cdots,n-1$.
Then Theorem C is valid for $\nu = 13$,
provided one of the following conditions is satisfied:

\begin{itemize}
\item[(i)] $0\le\gamma<1$ and $0 < \alpha \lesssim 1$.

\item[(ii)] $\gamma=1$ and $0 < \alpha \lesssim 1$. One of the conditions
($ii_1$)-($ii_4$) in Theorem \ref{thm3.51} is satisfied.

\item[(iii)] $0 \le \gamma < 1$, $0 < \alpha \le 1$ and for each $i\in\{1, \cdots, n-1\}$,
there exists $j(i)\in\{i+1, \cdots, n\}$ such that $a_{i,j(i)} < 0$.

\item[(iv)] $\gamma = 1$, $0 < \alpha \le 1$ and for each $i\in\{1, \cdots, n-1\}$
one of the conditions ($iv_1$)-($iiv_3$) in Theorem \ref{thm3.51} holds.
\end{itemize}
\end{corollary}

The result when ($i$) holds is better than the corresponding one given by \cite[Theorem 3.3]{YLK14}.

If $\sum_{k = i+1}^na_{i,k}a_{k,i} > 0, \; i = 1,\cdots,n-1$,
then for each $i\in\{1, \cdots, n-1\}$, there exists
$k(i)\in\{i+1, \cdots, n\}$ such that
$a_{i,k(i)}a_{k(i),i} > 0$ and $a_{i,k(i)} < 0$, which implies that ($iii$) of Theorem \ref{thm3.51}
holds for $j(i) = k(i)$ and ($iv_1$) in Theorem \ref{thm3.51}
holds for $j(i) = i$. Hence, Corollary \ref{coro3-35} when ($iii$) and ($iv$) hold is better than
\cite[Theorems 3.1, 3.2, 3.4]{YLK14}.

When $\alpha = 1$, \cite{MEY01} studies the preconditioned SOR method.
The main result \cite[Theorem 3.1]{MEY01} presents a Stein-Rosenberg type
comparison theorem. But it is incorrect. Where the authors
assume that $A$ is strictly diagonally dominant. Under this condition,
by \cite[Theorem 6-2.3]{BP94}, $A$ is a nonsingular M-matrix. Then, by
\cite[Theorem 7-5.24]{BP94}, it gets that $\rho({\mathscr L}_{\omega}) < 1$.
Hence, with our sign, \cite[Theorem 3.1]{MEY01} should be corrected as follows:
{\it ``If $A$ is a strictly diagonally dominant Z-matrix such that
$0 < a_{k,k+1}a_{k+1,k} < 1$, $k= 1,\cdots,n-1$ and $0 < \omega < 1$,
then $\rho({\mathscr L}^{(13)}_{\omega}) < \rho({\mathscr L}_{\omega}) < 1$"}.
While this result can be also derived from Corollary \ref{coro3-36}, directly.
In fact, the condition $a_{k,k+1}a_{k+1,k} > 0$, $k= 1,\cdots,n-1$,
implies that $A$ is irreducible and $a_{k,k+1} < 0$, $k= 1,\cdots,n-1$.
Therefore, the conditions ($iii$) and ($iv_2$) in Theorem \ref{thm3.51}
are satisfied. By ($i$) of Corollary \ref{coro3-36}, it follows that
$\rho({\mathscr L}^{(13)}_{\omega}) < \rho({\mathscr L}_{\omega}) < 1$
holds for $0 < \omega \le 1$.

\begin{corollary}\label{coro3-36}
Theorem D is valid for $\nu = 13$,
provided one of the following conditions is satisfied:

\begin{itemize}
\item[(i)] One of the conditions ($i$)-($iv$) of Corollary \ref{coro3-35} holds.

\item[(ii)] $0 \le \alpha \le 1$ and one of the conditions ($ii_1$) and ($ii_2$) in Theorem \ref{thm3.52} holds.
\end{itemize}
\end{corollary}

Corresponding to $Q_6$, for $r = 2, \cdots, n$, in \cite{Wa06} the matrix $Q$ is defined as
\begin{eqnarray*}
Q_{14} = \left(\begin{array}{cccccc}
0 & \quad \cdots\quad & 0 & 0 & \cdots  & 0\\
\vdots & \ddots & \vdots & \vdots & \vdots & \vdots\\
0 & \cdots & 0 & -\alpha_{r}a_{r-1,r} & \cdots & -\alpha_{n}a_{r-1,n} \\
0 & \cdots & 0 & 0 & \cdots & 0\\
\vdots & \vdots &\vdots & \vdots & \ddots & \vdots\\
0 & \cdots & 0 & 0 & \cdots & 0\\
\end{array}\right)
 \end{eqnarray*}
with $\alpha_k \ge 0$, $k = r, \cdots, n$, and
\begin{eqnarray*}
\sum\limits_{k = r}^{n}\alpha_ka_{r-1,k} \not = 0.
 \end{eqnarray*}

In this case, $\delta_{i,j}^{(12)}(\gamma)$ and $\delta_{i,j}^{(12)}(1)$ reduce respectively to
\begin{eqnarray*}
\delta_{i,j}^{(14)}(\gamma) = \left\{\begin{array}{ll}
 \sum\limits_{{k = r}}^n \alpha_ka_{r-1,k}a_{k,r-1}, & i = j = r-1;\\
 (\gamma-1)\alpha_ra_{r-1,j} + \gamma\sum\limits_{k = r}^{j-1} \alpha_ka_{r-1,k}a_{k,j}, &
  \begin{array}{l} i = r-1, \\
  j = r,\cdots,n; \end{array}\\
 \gamma\sum\limits_{k = r}^n\alpha_ka_{r-1,k}a_{k,j}, \; \hbox{whenever} \; r \ge 3, &
 \begin{array}{l} i = r-1,\\ j = 1,\cdots, r-2; \end{array}\\
  0, & otherwise
\end{array}\right.
\end{eqnarray*}
and
\begin{eqnarray*}
\delta_{i,j}^{(14)}(1) = \left\{\begin{array}{ll}
 \sum\limits_{k = r}^n\alpha_ka_{r-1,k}a_{k,j}, &
  i = r-1, j = 1,\cdots, r-1; \\
 \sum\limits_{k = r}^{j-1} \alpha_ka_{r-1,k}a_{k,j}, &
 i = r-1, j = r+1,\cdots,n;\\
  0, & otherwise.
\end{array}\right.
\end{eqnarray*}

Clearly, the conditions ($iii$) and ($iv$) of Theorem \ref{thm3.47} can be not satisfied.

From Theorems \ref{thm3.45}-\ref{thm3.48} and
Corollaries \ref{coro3-29}-\ref{coro3-32}, we have the following comparison results.

\begin{theorem}
Suppose that $\sum_{k = r}^{n}\alpha_ka_{r-1,k}a_{k,r-1} < 1$ and
\begin{eqnarray}\label{eqs3.40}
(1-\alpha_{j})a_{r-1,j} - \sum\limits_{{k = r}\atop{k\not = j}}^n \alpha_{k}a_{r-1,k}a_{k,j} \le 0,
\; j = r, \cdots, n.
\end{eqnarray}
Then Theorem A is valid for $\nu = 14$.
\end{theorem}

\begin{theorem}
Suppose that (\ref{eqs3.40}) holds.
Then Theorem B is valid for $\nu = 14$.
\end{theorem}

\begin{theorem}\label{thm3.55}
Suppose that $\sum_{k = r}^{n}\alpha_ka_{r-1,k}a_{k,r-1} < 1$ and
\begin{eqnarray}\label{eqs3.42}\nonumber
&& (1-\alpha_{j})a_{r-1,j} - \sum\limits_{{k = r}\atop{k\not = j}}^n \alpha_{k}a_{r-1,k}a_{k,j} \lesssim 0
\; \hbox{whenever} \;\; a_{r-1,j} < 0, \\
&& j = r, \cdots, n.
\end{eqnarray}
Then Theorem C is valid for $\nu = 14$,
provided one of the following conditions is satisfied:

\begin{itemize}
\item[(i)] $0 \le \gamma < 1$.

\item[(ii)]
$\gamma = 1$ and one of the following conditions holds:

\begin{itemize}
\item[($ii_1$)]
There exist $i\in\{1, \cdots, r-1\}$ and $j\in\{r, \cdots, n\}$
such that $\alpha_ja_{r-1,j}a_{j,i}$ $ > 0$.

\item[($ii_2$)]
There exist $i\in\{r+1, \cdots, n\}$ and $j\in\{r, \cdots, i-1\}$ such that $\alpha_ja_{r-1,j}a_{j,i} > 0$.

\item[($ii_3$)]
$a_{r-1,r} < 0$ and $\alpha_r > 0$.

\item[($ii_4$)]
$a_{k,1} < 0$, $k = r, \cdots, n$.

\item[($ii_5$)]
$a_{k,r-1} < 0$, $k = r, \cdots, n$.

\item[($ii_6$)]
$a_{n,1} < 0$ and $a_{k,n} < 0$, $k = r, \cdots, n-1$.

\item[($ii_7$)]
$a_{n,1} < 0$ and $a_{k,k+1} < 0$, $k = r, \cdots, n-1$.
\end{itemize}\end{itemize}
\end{theorem}

\begin{proof}
By Theorem \ref{thm3.47} we just need to prove ($ii_4$)-($ii_7$).

From the definition of $Q_{14}$, there exists $k_0\in\{r,\cdots,n\}$ such that
$\alpha_{k_0}a_{r-1,k_0}$ $ < 0$.

If ($ii_4$) holds then $\alpha_{k_0}a_{r-1,k_0}a_{k_0,1} > 0$, which implies that
($ii_1$) holds for $i = 1$ and $j = k_0$.

Similarly, if ($ii_5$) holds then we can prove that ($ii_1$) holds for $i = r-1$ and $j = k_0$.

When ($ii_6$) holds, if $k_0 < n$ then $\alpha_{k_0}a_{r-1,k_0}a_{k_0,n} > 0$, which implies that
($ii_2$) holds for $i = n$ and $j = k_0$. If $k_0 = n$ then $\alpha_{n}a_{r-1,n}a_{n,1} > 0$,
which implies that ($ii_1$) holds for $i = 1$ and $j = n$.

Similarly, when ($ii_7$) holds we can prove that ($ii_2$) holds.
\end{proof}

\begin{theorem}\label{thm3.56}
Theorem D is valid for $\nu = 14$,
provided one of the following conditions is satisfied:

\begin{itemize}
\item[(i)] The inequality (\ref{eqs3.42}) and
one of the conditions ($i$) and ($ii$) of Theorem \ref{thm3.55} holds.

\item[(ii)] The inequality (\ref{eqs3.40}) holds and
one of the following conditions holds:

\begin{itemize}
\item[($ii_1$)] There exists $k\in\{r,\cdots,n\}$ such that
$\alpha_{k}a_{r-1,k}a_{k,r-1} > 0$.

\item[($ii_2$)] $\gamma > 0$ and there exist $k\in\{r,\cdots,n\}$ and $j\in\{1,\cdots,r-2\}$ such that
$\alpha_{k}a_{r-1,k}a_{k,j} > 0$.
\end{itemize}\end{itemize}
\end{theorem}

\begin{corollary}
Suppose that $0 \le \alpha_k \le 1$,
$k = r, \cdots, n$, and\\ $\sum_{k = r}^{n}\alpha_ka_{r-1,k}a_{k,r-1}$ $ < 1$.
Then Theorem A is valid for $\nu = 14$.
\end{corollary}

\begin{corollary}
Suppose that $0 \le \alpha_k \le 1$, $k = r, \cdots, n$.
Then Theorem B is valid for $\nu = 14$.
\end{corollary}

The result includes the corresponding one given in \cite[Corollary 2.3]{Wa06}.

\begin{corollary}
Suppose that $0 \le \alpha_k \lesssim 1$,
$k = r, \cdots, n$, and\\ $\sum_{k = r}^{n}\alpha_ka_{r-1,k}a_{k,r-1}$ $ < 1$.
Then Theorem C is valid for $\nu = 14$,
provided one of the conditions ($i$) and ($ii$) of Theorem \ref{thm3.55} is satisfied.
\end{corollary}

\begin{corollary}
Theorem D is valid for $\nu = 14$,
provided one of the following conditions is satisfied:

\begin{itemize}
\item[(i)] $0 \le \alpha_k \lesssim 1$,
$k = r, \cdots, n$, and
one of the conditions ($i$) and ($ii$) of Theorem \ref{thm3.55} holds.

\item[(ii)] $0 \le \alpha_k \le 1$, $k = r, \cdots, n$, and
one of the conditions ($ii_1$) and ($ii_2$) in Theorem \ref{thm3.56} holds.
\end{itemize}
\end{corollary}

Similar to $Q_{14}$, for $r = 2, \cdots, n$, the matrix $Q$ is chosen as
\begin{eqnarray*}
Q_{15} = \left(\begin{array}{cccccc}
0 & \quad \cdots\quad & 0 & -\alpha_{1}a_{1,r} & \cdots  & 0\\
\vdots & \ddots & \vdots & \vdots & \vdots & \vdots\\
0 & \cdots & 0 & -\alpha_{r-1}a_{r-1,r} & \cdots & 0 \\
0 & \cdots & 0 & 0 & \cdots & 0\\
\vdots & \vdots &\vdots & \vdots & \ddots & \vdots\\
0 & \cdots & 0 & 0 & \cdots & 0\\
\end{array}\right)
 \end{eqnarray*}
with $\alpha_k \ge 0$, $k = 1, \cdots, r-1$, and
\begin{eqnarray*}
\sum\limits_{k = 1}^{r-1}\alpha_ka_{k,r} \not = 0.
 \end{eqnarray*}

It is proposed in \cite{Yu11} for $r = n$. When $r = n$ and $\alpha_i = 1$, $i = 2, \cdots, n$,
it is a special case in \cite{Mi87} for the preconditioned Gauss-Seidel and Jacobi methods.

In this case, $\delta_{i,j}^{(12)}(\gamma)$ and $\delta_{i,j}^{(12)}(1)$ reduce respectively to
\begin{eqnarray*}
\delta_{i,j}^{(15)}(\gamma) = \left\{\begin{array}{ll}
 \alpha_ia_{i,r}a_{r,i}, & i = j = 1, \cdots, r-1;\\
 (\gamma-1)\alpha_ia_{i,r}, & i = 1, \cdots, r-1, j = r;\\
 \gamma\alpha_ia_{i,r}a_{r,j}, & i = 1, \cdots, r-1,\\
 & j\in\{1, \cdots, i-1\}\cup\{r+1, \cdots, n\};\\
  0, & otherwise
\end{array}\right.
\end{eqnarray*}
and
\begin{eqnarray*}
\delta_{i,j}^{(15)}(1) = \left\{\begin{array}{ll}
 \alpha_ia_{i,r}a_{r,j}, & i = 1, \cdots, r-1, 1 \le j \le i, r+1 \le j \le n;\\
  0, & otherwise.
\end{array}\right.
\end{eqnarray*}

From Corollaries \ref{coro3-29} and \ref{coro3-30}, we have the following comparison results.

\begin{theorem}\label{thm3.57}
Suppose that $0 \le \alpha_k \le 1$ and $\alpha_ka_{k,r}a_{r,k} < 1$, $k = 1, \cdots, r-1$.
Then Theorem A is valid for $\nu = 15$.
\end{theorem}

\begin{theorem}\label{thm3.58}
Suppose that $0 \le \alpha_k \le 1$, $k = 1,\cdots, r-1$.
Then Theorem B is valid for $\nu = 15$.
\end{theorem}

In order to give the Stein-Rosenberg Type Theorem II,
completely similar to Lemma \ref{lem3-4}, we can prove the following lemma.

\begin{lemma}\label{lem3-8}
Let $A$ be an irreducible Z-matrix.
Assume that $r = n$, $0 < \alpha_k \le 1$,
$k = 1, \cdots, n-1$ and $A^{(15)}$ has the block form
\begin{eqnarray*}
A^{(15)} = \left(\begin{array}{ccc}
A^{(15)}_{1,1} & & \bar{a}^{(15)}_{1,2}\\
\bar{a}_{2,1} & & 1
\end{array}\right), \; A^{(15)}_{1,1}\in{\mathscr R}^{(n-1)\times(n-1)}.
 \end{eqnarray*}
 Then

\begin{itemize}
\item[(i)] $A^{(15)}_{1,1}$ is an irreducible Z-matrix.

\item[(ii)] $A^{(15)}$ is an irreducible Z-matrix if and only if there exists
$j_0\in\{1, \cdots, n-1\}$ such that $(1-\alpha_{j_0})a_{j_0,n} \not = 0$.
\end{itemize}
\end{lemma}

\begin{theorem}\label{thm3.59}
Suppose that $\alpha_ka_{k,r}a_{r,k} < 1$, $k = 1, \cdots, r-1$.
Then Theorem C is valid for $\nu = 15$,
provided one of the following conditions is satisfied:

\begin{itemize}
\item[(i)] $0 \le \gamma < 1$ and $0 \le \alpha_k \lesssim 1$, $k = 1,\cdots, r-1$.

\item[(ii)] $\gamma = 1$, $0 \le \alpha_k \lesssim 1$, $k = 1,\cdots, r-1$.
And one of the following conditions holds:

\begin{itemize}
\item[($ii_1$)] There exist $i\in\{1, \cdots, r-1\}$ and
$j\in\{1, \cdots, i\}\cup\{r+1, \cdots, n\}$ such that
$\alpha_ia_{i,r}a_{r,j} > 0$.

\item[($ii_2$)] $a_{r-1,r} < 0$ and $\alpha_{r-1} > 0$.

\item[($ii_3$)] $a_{r,1} < 0$.

\item[($ii_4$)] There exists $k\in\{r+1, \cdots, n\}$ such that $a_{r,k} < 0$.
\end{itemize}

\item[(iii)] $r = n$, $0 \le \gamma < 1$ and $0 < \alpha_k \le 1$, $k = 1,\cdots, n-1$.

\item[(iv)] $r = n$, $\gamma = 1$, $0 < \alpha_k \le 1$, $k = 1,\cdots, n-1$.
And one of the following conditions holds:

\begin{itemize}
\item[($iv_1$)] There exist $i\in\{1, \cdots, n-1\}$ and
$j\in\{1, \cdots, i\}$ such that $a_{i,n}a_{n,j} > 0$.

\item[($iv_2$)] $a_{n-1,n} < 0$.

\item[($iv_3$)] $a_{n,1} < 0$.
\end{itemize}
\end{itemize}
\end{theorem}

\begin{proof}
We just need to prove ($ii_3$), ($ii_4$) and ($iii$) and ($iv$).

From the definition of $Q_{15}$, there exists $i_0\in\{1, \cdots, r-1\}$ such that
$\alpha_{i_0}a_{i_0,r} < 0$. If ($ii_3$) holds then $\alpha_{i_0}a_{i_0,r}a_{r,1} > 0$, which
implies that ($ii_1$) holds for $i = i_0$ and $j = 1$. If ($ii_4$) holds then we can prove that
($ii_1$) holds for $i = i_0$ and $j = k$.

For ($iii$) and ($iv$), if $A^{(15)}$ is irreducible then the result is obvious.

Now, we consider the case when $A^{(15)}$ is reducible. Suppose that
$\rho = \rho({\mathscr L}_{\gamma,\omega})$ and
$x > 0$ is its associated eigenvector. By Theorem \ref{thm3.57},
$\rho =1$ if and only if $\rho({\mathscr L}^{(15)}_{\gamma,\omega}) =1$.

By Lemma \ref{lem3-1} we just need to consider the case when $\gamma \le \omega$.
Then, by Lemma \ref{lem1-8} it follows
that $\rho > 0$ and $x\gg 0$.

Let $A$ have the block form
\begin{eqnarray*}
A = \left(\begin{array}{ccc}
A_{1,1} & & \bar{a}_{1,2}\\
\bar{a}_{2,1} & & 1
\end{array}\right), \; A_{1,1}\in{\mathscr R}^{(n-1)\times(n-1)}.
 \end{eqnarray*}
Then, by Lemma \ref{lem3-8}, $A^{(15)}$ has the block form
\begin{eqnarray*}
A^{(15)} = \left(\begin{array}{ccc}
A^{(15)}_{1,1} & & 0\\
\bar{a}_{2,1} & & 1
\end{array}\right), \; A^{(15)}_{1,1}\in{\mathscr R}^{(n-1)\times(n-1)},
 \end{eqnarray*}
where $A^{(15)}_{1,1}$ is an irreducible L-matrix, since
$a^{(15)}_{k,k} = 1 - \alpha_ka_{k,n}a_{n,k} > 0$, $k = 1, \cdots, n-1$.
Let $A_{1,1} = \hat{M}_{\gamma,\omega} - \hat{N}_{\gamma,\omega}$ and
$A^{(15)}_{1,1} = \bar{M}_{\gamma,\omega} - \bar{N}_{\gamma,\omega}$ be
 the AOR splittings of $A_{1,1}$ and $A^{(15)}_{1,1}$, respectively.
Then they are regular splittings and
\begin{eqnarray*}
&& M_{\gamma,\omega} = \left(\begin{array}{ccc}
\hat{M}_{\gamma,\omega} & & 0\\
\frac{\gamma}{\omega}\bar{a}_{2,1} & & \frac{1}{\omega}
\end{array}\right), \;
N_{\gamma,\omega} = \left(\begin{array}{ccc}
\hat{N}_{\gamma,\omega} & & -\bar{a}_{1,2} \\
\frac{\gamma-\omega}{\omega}\bar{a}_{2,1} & & \frac{1-\omega}{\omega}
\end{array}\right), \\
&& M^{(15)}_{\gamma,\omega} = \left(\begin{array}{ccc}
\bar{M}_{\gamma,\omega} & & 0\\
\frac{\gamma}{\omega}\bar{a}_{2,1} & & \frac{1}{\omega}
\end{array}\right), \;
N^{(15)}_{\gamma,\omega} = \left(\begin{array}{ccc}
\bar{N}_{\gamma,\omega} & & 0 \\
\frac{\gamma-\omega}{\omega}\bar{a}_{2,1} & & \frac{1-\omega}{\omega}
\end{array}\right)
 \end{eqnarray*}
 and
\begin{eqnarray*}
&& [M^{(15)}_{\gamma,\omega}]^{-1} = \left(\begin{array}{ccc}
\bar{M}_{\gamma,\omega}^{-1} & & 0\\
-\gamma\bar{a}_{2,1}\bar{M}_{\gamma,\omega}^{-1} & & \omega
\end{array}\right), \\
&& {\mathscr L}^{(15)}_{\gamma,\omega}
 = \left(\begin{array}{ccc}
\bar{M}_{\gamma,\omega}^{-1}\bar{N}_{\gamma,\omega} & & 0 \\
\bar{a}_{2,1}[(\gamma-\omega)I-\gamma\bar{M}_{\gamma,\omega}^{-1}\bar{N}_{\gamma,\omega}] & & 1-\omega
\end{array}\right).
 \end{eqnarray*}
Let $E_2$ and $F_2$ be diagonal part and strictly lower triangular part of $Q_{15}L$
with block forms
\begin{eqnarray*}
E_2 = \left(\begin{array}{ccc}
E_{1,1} && 0 \\
0 && e_{2,2}
\end{array}\right), \;
F_2 = \left(\begin{array}{ccc}
F_{1,1} && 0 \\
f_{1,2} && f_{2,2}
\end{array}\right), \; E_{1,1}, F_{1,1}\in{\mathscr R}^{(n-1)\times(n-1)}.
\end{eqnarray*}
Then $e_{2,2} = f_{2,2} = 0$ and $f_{1,2} = 0$. Let
\begin{eqnarray*}
Q_{15} = \left(\begin{array}{cc}
Q_{1,1} & \bar{q}_{1,2}\\
\bar{q}_{2,1} & q^{(15)}_{2,2}
\end{array}\right), \; Q_{1,1}\in{\mathscr R}^{(n-1)\times (n-1)}
 \end{eqnarray*}
 and
\begin{eqnarray*}
 x = \left(\begin{array}{c}
\bar{x}_{1}\\
\bar{x}_{2}
\end{array}\right), \; \bar{x}_{1}\in{\mathscr R}^{n-1}, \; \bar{x}_{2}\in{\mathscr R}.
 \end{eqnarray*}
 Then $Q_{1,1} = 0$, $\bar{q}_{2,1} = 0$, $q^{(15)}_{2,2} = 0$, $\bar{q}_{1,2}>0$, $\bar{x}_{1} \gg 0$ and $\bar{x}_{2} > 0$.
 Now, by Lemma \ref{lem3-2}, we obtain
\begin{eqnarray*}
&& {\mathscr L}^{(15)}_{\gamma,\omega}x - \rho x\\
 && = \left(\begin{array}{c}
\bar{M}_{\gamma,\omega}^{-1}\bar{N}_{\gamma,\omega}\bar{x}_{1} -\rho\bar{x}_{1} \\
\bar{a}_{2,1}[(\gamma-\omega)I-\gamma\bar{M}_{\gamma,\omega}^{-1}\bar{N}_{\gamma,\omega}]\bar{x}_{1} + (1-\omega)\bar{x}_{2} - \rho\bar{x}_{2}
\end{array}\right) \\
 && = (\rho-1)[M^{(15)}_{\gamma,\omega}]^{-1}[E_2 +\gamma F_2 + \frac{\omega}{\rho}Q_{15}N_{\gamma,\omega}]x \\
 && = (\rho-1)\left(\begin{array}{ccc}
\bar{M}_{\gamma,\omega}^{-1} & & 0\\
-\gamma\bar{a}_{2,1}\bar{M}_{\gamma,\omega}^{-1} & & \omega
\end{array}\right)
\left[\left(\begin{array}{cc}
E_{1,1} & 0 \\
0 & 0
\end{array}\right)
+ \gamma\left(\begin{array}{cc}
F_{1,1} & 0 \\
0 & 0
\end{array}\right) \right.\\
&& \quad\left. + \frac{\omega}{\rho}\left(\begin{array}{cc}
0 & \bar{q}_{1,2}\\
0 & 0
\end{array}\right)
\left(\begin{array}{ccc}
\hat{N}_{\gamma,\omega} & & -\bar{a}_{1,2} \\
\frac{\gamma-\omega}{\omega}\bar{a}_{2,1} & & \frac{1-\omega}{\omega}
\end{array}\right)
\right]\left(\begin{array}{c}
\bar{x}_{1}\\
\bar{x}_{2}
\end{array}\right) \\
 && = (\rho-1) \left(\begin{array}{c}
\bar{M}_{\gamma,\omega}^{-1}[(E_{1,1} + \gamma F_{1,1} +
\frac{\gamma-\omega}{\rho}\bar{q}_{1,2}\bar{a}_{2,1})\bar{x}_{1} +
\frac{1-\omega}{\rho}\bar{q}_{1,2}\bar{x}_{2}] \\
-\gamma\bar{a}_{2,1}\bar{M}_{\gamma,\omega}^{-1}[(E_{1,1} + \gamma F_{1,1} +
\frac{\gamma-\omega}{\rho}\bar{q}_{1,2}\bar{a}_{2,1})\bar{x}_{1} +
\frac{1-\omega}{\rho}\bar{q}_{1,2}\bar{x}_{2}]
\end{array}\right).
\end{eqnarray*}
Hence, we have
\begin{eqnarray} \label{eqs3.43}
&& \bar{M}_{\gamma,\omega}^{-1}\bar{N}_{\gamma,\omega}\bar{x}_{1} - \rho\bar{x}_{1} \\\nonumber
 &&= (\rho-1)\bar{M}_{\gamma,\omega}^{-1}[(E_{1,1} + \gamma F_{1,1} +
\frac{\gamma-\omega}{\rho}\bar{q}_{1,2}\bar{a}_{2,1})\bar{x}_{1} +
\frac{1-\omega}{\rho}\bar{q}_{1,2}\bar{x}_{2}]
\end{eqnarray}
and
 \begin{eqnarray} \label{eqs3.44} \nonumber
&& (1-\omega)\bar{x}_{2} - \rho\bar{x}_{2} \\
&&= \bar{a}_{2,1}[(\omega-\gamma)I+\gamma\bar{M}_{\gamma,\omega}^{-1}\bar{N}_{\gamma,\omega}]\bar{x}_{1} \\ \nonumber
&& \quad -(\rho-1)\gamma\bar{a}_{2,1}\bar{M}_{\gamma,\omega}^{-1}[(E_{1,1} + \gamma F_{1,1} +
\frac{\gamma-\omega}{\rho}\bar{q}_{1,2}\bar{a}_{2,1})\bar{x}_{1} +
\frac{1-\omega}{\rho}\bar{q}_{1,2}\bar{x}_{2}].
\end{eqnarray}

For the case when (iii) holds, i.e., $\gamma < 1$, let $A^{(15)}_{1,1} = \breve{D} - \breve{L} - \breve{U}$,
where $\breve{D}$, $\breve{L}$ and $\breve{U}$
are respectively diagonal, strictly lower and upper triangular matrices.
Since $A^{(15)}_{1,1}$ is an irreducible L-matrix and
\begin{eqnarray*}
\bar{M}_{\gamma,\omega}^{-1}\bar{N}_{\gamma,\omega} =
\breve{D}^{-1}[(1-\omega)\breve{D} + \omega(1-\gamma)\breve{L} + \omega \breve{U}] + T \ge 0
\end{eqnarray*}
with
\begin{eqnarray*}
T = \omega\gamma \breve{D}^{-1}\breve{L}(\breve{D}-\gamma \breve{L})^{-1}[(1-\gamma)\breve{L} + \breve{U}]\ge 0,
\end{eqnarray*}
then it follows that $\bar{M}_{\gamma,\omega}^{-1}\bar{N}_{\gamma,\omega}$ is irreducible,
so that $\bar{M}_{\gamma,\omega}^{-1}\bar{N}_{\gamma,\omega}\bar{x}_{1} \gg 0$.
It is easy to see that
\begin{eqnarray*}
&& \frac{\gamma-\omega}{\rho}\bar{q}_{1,2}\bar{a}_{2,1}\bar{x}_{1} +
\frac{1-\omega}{\rho}\bar{q}_{1,2}\bar{x}_{2} > 0,\\
&& \bar{a}_{2,1}[(\omega-\gamma)I+\gamma\bar{M}_{\gamma,\omega}^{-1}\bar{N}_{\gamma,\omega}]\bar{x}_{1}<0.
\end{eqnarray*}
When $\rho < 1$, from (\ref{eqs3.43}) and (\ref{eqs3.44}) it gets that
\begin{eqnarray*}
&& \bar{M}_{\gamma,\omega}^{-1}\bar{N}_{\gamma,\omega}\bar{x}_{1} - \rho\bar{x}_{1} < 0, \;
(1-\omega)\bar{x}_{2} - \rho\bar{x}_{2} < 0,
\end{eqnarray*}
which shows that $\rho(\bar{M}_{\gamma,\omega}^{-1}\bar{N}_{\gamma,\omega}) < \rho$ and
$1-\omega < \rho$. Therefore we have $\rho({\mathscr L}^{(15)}_{\gamma,\omega}) = \max\{1-\omega, \;
\rho(\bar{M}_{\gamma,\omega}^{-1}\bar{N}_{\gamma,\omega})\} < \rho$. When
$\rho > 1$ then, from (\ref{eqs3.43}) it gets that
$\bar{M}_{\gamma,\omega}^{-1}\bar{N}_{\gamma,\omega}\bar{x}_{1} - \rho\bar{x}_{1} > 0$,
which shows that $\rho(\bar{M}_{\gamma,\omega}^{-1}\bar{N}_{\gamma,\omega}) > \rho >1$ and
hence, $\rho({\mathscr L}^{(15)}_{\gamma,\omega}) = \max\{1-\omega, \;
\rho(\bar{M}_{\gamma,\omega}^{-1}\bar{N}_{\gamma,\omega})\} =
\rho(\bar{M}_{\gamma,\omega}^{-1}\bar{N}_{\gamma,\omega}) > \rho$.

When (iv) holds, i.e., $\gamma = 1$, then $\omega = 1$ and the AOR method reduces to the Gauss-Seidel method.
In this case, we have
\begin{eqnarray*}
{\mathscr L}^{(15)}
 = \left(\begin{array}{ccc}
\bar{M}_{1,1}^{-1}\bar{N}_{1,1} & & 0 \\
-\bar{a}_{2,1}\bar{M}_{1,1}^{-1}\bar{N}_{1,1} & & 0
\end{array}\right)
 \end{eqnarray*}
and therefore $\rho({\mathscr L}^{(15)}) = \rho(\bar{M}_{1,1}^{-1}\bar{N}_{1,1})$.
Now, (\ref{eqs3.43}) reduces to
\begin{eqnarray*}
\bar{M}_{1,1}^{-1}\bar{N}_{1,1}\bar{x}_{1} - \rho\bar{x}_{1}
= (\rho-1)\bar{M}_{1,1}^{-1}(E_{1,1} + F_{1,1})\bar{x}_{1}.
\end{eqnarray*}
Since $A^{(15)}_{1,1}$ is irreducible, then the rest proof is completely similar to that of Theorem \ref{thm3.23}.
 \end{proof}

The result when ($iii$) and ($iv$) hold is better than \cite[Theorem 3.16, Corollary 3.17]{Yu11},
where the condition that $a_{k,n} \not= 0$ for $k = 1, \cdots, n-1$, is redundant. Again,
the result when ($i$) and ($ii$) hold includes \cite[Theorems 3.18, 3.19 and 3.20]{Yu11}.

\begin{theorem}\label{thm3.60}
Theorem D is valid for $\nu = 15$,
provided one of the following conditions is satisfied:

\begin{itemize}
\item[(i)] One of the conditions ($i$)-($iv$) of Theorem \ref{thm3.59} holds.

\item[(ii)] $0 \le \alpha_k \le 1$, $k = 1,\cdots, r-1$,
one of the following conditions holds:

\begin{itemize}
\item[($ii_1$)] There exists $i\in\{1,\cdots,r-1\}$ such that
$\alpha_{i}a_{i,r}a_{r,i} > 0$.

\item[($ii_2$)] $\gamma > 0$ and there exist $i\in\{2,\cdots,r-1\}$ and $j\in\{1,\cdots,i-1\}$ such that
$\alpha_{i}a_{i,r}a_{r,j} > 0$.
\end{itemize}
\end{itemize}
\end{theorem}

As a special case, for some $r < s$ with $a_{r,s} < 0$ and $\alpha > 0$, the matrix $Q_{15}$ reduces to
\begin{eqnarray*}
Q_{16} & = & \left(\begin{array}{ccccc}
0 & \quad \cdots\quad & 0 & \quad \cdots \quad & 0\\
\vdots & \ddots & \vdots & -\frac{a_{r,s}}{\alpha} & \vdots\\
0 & \cdots & 0 & \cdots & \vdots \\
\vdots & \vdots & \vdots & \ddots & \vdots\\
0 & \cdots & 0 & \cdots & 0\\
\end{array}\right),
 \end{eqnarray*}
which is given in \cite{LLW072} for $r = 1$ and $s = n$.
And for $r = 1$, $s = n$ and $\alpha = 1$ it is proposed in \cite{EMT01}.
When $r = 1$ and $s = n$, it is given in \cite{WH06} for the preconditioned Gauss-Seidel method.
When $\alpha = 1$ and $s = r+1$, it is
proposed in \cite{EMT96} for the preconditioned MSOR method.
It is given in \cite{ZHL05} to replace $-a_{r,s}/\alpha$ with a constant $\beta$.

In this case, $\delta_{i,j}^{(15)}(\gamma)$ and $\delta_{i,j}^{(15)}(1)$ reduce respectively to
\begin{eqnarray*}
\delta_{i,j}^{(16)}(\gamma) = \left\{\begin{array}{ll}
 \frac{1}{\alpha}a_{r,s}a_{s,r}, & i = j = r;\\
 \frac{\gamma-1}{\alpha}a_{r,s}, & i = r, j = s;\\
 \frac{\gamma}{\alpha}a_{r,s}a_{s,j}, & i = r, 1 \le j \le r-1, s+1 \le j \le n;\\
  0, & otherwise
\end{array}\right.
\end{eqnarray*}
and
\begin{eqnarray*}
\delta_{i,j}^{(16)}(1) = \left\{\begin{array}{ll}
 \frac{1}{\alpha}a_{r,s}a_{s,j}, & i = r, 1 \le j \le r, s+1 \le j \le n;\\
  0, & otherwise.
\end{array}\right.
\end{eqnarray*}

Clearly, the conditions ($iii$) and ($iv$) of Theorem \ref{thm3.59} can be not satisfied.

From Theorems \ref{thm3.57} and \ref{thm3.58}, the following comparison results are
immediately.

\begin{theorem}
Suppose that $\alpha \ge 1$ and $\alpha > a_{r,s}a_{s,r}$.
Then Theorem A is valid for $\nu = 16$.
\end{theorem}

\begin{theorem}
Suppose that $\alpha \ge 1$.
Then Theorem B is valid for $\nu = 16$.
\end{theorem}

In order to give the Stein-Rosenberg Type Theorem II,
completely similar to Lemma \ref{lem3-5}, we can prove the following lemma.

\begin{lemma}\label{lem3-9}
Let $A$ be an irreducible Z-matrix.
Assume that $r = 1$ and $s = n$, $\alpha \ge 1$ and $A^{(16)}$ has the block form
\begin{eqnarray*}
A^{(16)} = \left(\begin{array}{ccc}
A^{(16)}_{2,2} & & \bar{a}^{(16)}_{1,2}\\
\bar{a}_{2,1} & & 1
\end{array}\right), \; A^{(16)}_{1,1}\in{\mathscr R}^{(n-1)\times(n-1)}.
 \end{eqnarray*}
 Then one of the following two mutually exclusive relations holds:

\begin{itemize}
\item[(i)] $A^{(16)}$ is an irreducible Z-matrix.

\item[(ii)] $A^{(16)}$ is a reducible Z-matrix,
but $A^{(16)}_{1,1}$ is an irreducible Z-matrix and
$a_{1,k} = a^{(16)}_{1,k} = a^{(16)}_{1,n} = 0$, $k = 2,\cdots,n-1$.
\end{itemize}
\end{lemma}

Using this lemma, completely similar to the proof of Theorem \ref{thm3.59}, we can prove the following theorem.

\begin{theorem}\label{thm3.63}
Suppose that $\alpha > a_{r,s}a_{s,r}$.
Then Theorem C is valid for $\nu = 16$,
provided one of the following conditions is satisfied:

\begin{itemize}
\item[(i)] $\alpha \gtrsim 1$. And one of the following conditions holds:

\begin{itemize}
\item[($i_1$)] $0 \le \gamma < 1$.

\item[($i_2$)] $\gamma = 1$ and there exists
$k\in\{1, \cdots, r\}\cup\{s+1, \cdots, n\}$ such that
$a_{s,k} < 0$.

\item[($i_3$)] $\gamma = 1$ and $s = r+1$.
\end{itemize}

\item[(ii)] $r = 1$, $s = n$ and $\alpha \ge 1$.
\end{itemize}
\end{theorem}

The result when ($ii$) holds is better than \cite[Theorem 2]{LLW072}.
For $\alpha = 1$, it is also better than the corresponding one given by
\cite[Theorem 2.2, Corollaries 2.1, 2.2]{LW04},
where the condition $a_{k,k+1}a_{k+1,k} > 0$, $k = 1, \cdots, n-1$, implies that
$A$ is irreducible and so that the condition $a_{1,n}a_{n,1} > 0$ is unnecessary.

\begin{theorem}
Theorem D is valid for $\nu = 16$,
provided one of the following conditions is satisfied:

\begin{itemize}
\item[(i)] One of the conditions ($i$) and ($ii$) of Theorem \ref{thm3.63} holds.

\item[(ii)] $\alpha \ge 1$ and one of the following conditions holds:

\begin{itemize}
\item[($ii_1$)]
$a_{s,r} < 0$.

\item[($ii_2$)]
$\gamma > 0$, $r \ge 2$ and there exists $k\in\{1,\cdots,r-1\}$ such that
$a_{s,k} < 0$.
\end{itemize}
\end{itemize}
\end{theorem}

The result when ($ii$) holds for $r = 1$, $s = n$ and $\alpha = 1$, is better than \cite[Theorem 2.3]{EMT01}.

In \cite{KKNU97}, for the preconditioned Gauss-Seidel method,
$Q$ is chosen as
\begin{eqnarray*}
Q_{17} = \left(\begin{array}{ccccc}
0 & -\alpha_1a_{1,2} & 0 & \cdots & 0\\
0 & 0 & -\alpha_2a_{2,3} & \cdots & 0\\
\vdots & \vdots & \ddots & \ddots & \vdots\\
0 & 0 & 0 & \ddots & -\alpha_{n-1}a_{n-1,n} \\
0 & 0 & 0 & \cdots & 0
\end{array}\right)
 \end{eqnarray*}
with $\alpha_k \ge 0$, $k = 1, \cdots, n-1$, and
\begin{eqnarray*}
\sum\limits_{k = 1}^{n-1}\alpha_ka_{k,k+1} \not = 0,
 \end{eqnarray*}
which is used to the preconditioned AOR method in \cite{WWS07,Yu11},
to the preconditioned SOR method in \cite{Su05}
and to the preconditioned Gauss-Seidel and Jacobi methods in \cite{HNT03}.
For $\alpha_k = 1$, $k = 1, \cdots, n-1$, it is
proposed in \cite{GJS91} for the preconditioned Gauss-Seidel method,
in \cite{LE94} for the preconditioned SOR method,
and in \cite{EM95,LLW07} for the preconditioned AOR method.

In this case, $\delta_{i,j}^{(12)}(\gamma)$ and $\delta_{i,j}^{(12)}(1)$ reduce respectively to
\begin{eqnarray*}
\delta_{i,j}^{(17)}(\gamma) = \left\{\begin{array}{ll}
 \alpha_ia_{i,i+1}a_{i+1,i}, & i = j = 1, \cdots, n-1;\\
 (\gamma-1)\alpha_ia_{i,i+1}, & i = 1, \cdots, n-1, j = i+1;\\
 \gamma\alpha_ia_{i,i+1}a_{i+1,j}, & i = 1, \cdots, n-1, j = 1, \cdots, n, j\not = i, i+1; \\
  0, & i = n, j = 1,\cdots, n
\end{array}\right.
\end{eqnarray*}
and
\begin{eqnarray*}
\delta_{i,j}^{(17)}(1) = \left\{\begin{array}{ll}
\alpha_ia_{i,i+1}a_{i+1,j}, & i = 1, \cdots, n-1, j = 1, \cdots, n, j\not = i+1; \\
  0, & otherwise.
\end{array}\right.
\end{eqnarray*}

From Corollaries \ref{coro3-29} and \ref{coro3-30} we can obtain the following comparison
result, directly.

\begin{theorem}
Suppose that $0 \le \alpha_k \le 1$ and $\alpha_ka_{k,k+1}a_{k+1,k} < 1$, $k = 1, \cdots, n-1$.
Then Theorem A is valid for $\nu = 17$.
\end{theorem}

The result improves the corresponding ones given by
\cite[Theorem 2.1]{WWS07} and includes \cite[Theorem 4.1]{LS00} for the
preconditioned Gauss-Seidel method.
It is also better than \cite[Theorem 3.6]{Yu11}, where it is assumed that $A$
is irreducible.

\begin{theorem}
Suppose that $0 \le \alpha_k \le 1$, $k = 1, \cdots, n-1$.
Then Theorem B is valid for $\nu = 17$.
\end{theorem}

The result includes the corresponding ones given by
\cite[Corollary 2.3]{WWS07}.
For the Gauss-Seidel method it is better than \cite[Theorem 2]{KN09} where the condition that the
Gauss-Seidel methods are convergent is redundant, \cite[Theorem 28]{Li08}
where the assumption that $A$ is irreducible is redundant,
\cite[Theorem 3.5]{KHMN02} and \cite[Theorem 2.4]{NHMS04}
since an irreducibly diagonally dominant Z-matrix is a nonsingular M-matrix.

Completely similar to Lemma \ref{lem3-44} we can prove the following lemma.

\begin{lemma}\label{lem3-444}
Let $A$ be a Z-matrix. Assume that
$n \ge 3$, $a_{1,n} < 0$, $a_{k+1,k} < 0$, $0 \le \alpha_k \le 1$,
$k = 1, \cdots, n-1$.
Then $A$ and $A^{(17)}$ are irreducible Z-matrices.
\end{lemma}

\begin{theorem}\label{thm3.67}
Suppose that $\alpha_ka_{k,k+1}a_{k+1,k} < 1$, $k = 1, \cdots, n-1$.
Then Theorem C is valid for $\nu = 17$,
provided one of the following conditions is satisfied:

\begin{itemize}
\item[(i)] $0 \le \alpha_k \lesssim 1$, $k = 1,\cdots, n-1$.

\item[(ii)] $0 < \alpha_k \le 1$ and $a_{k,k+1} < 0$, $k = 1, \cdots, n-1$.

\item[(iii)] $n \ge 3$, $0 \le \alpha_k \le 1$, $a_{1,n} < 0$, $a_{k+1,k} < 0$,
$k = 1, \cdots, n-1$.
\end{itemize}
\end{theorem}

\begin{proof}
By the definition of $Q_{17}$,
there exists $k_0\in\{1, \cdots, n-1\}$ such that
$a_{k_0,k_0+1} < 0$

When ($i$) holds, the conditions ($i$) and ($ii_3$) in Theorem \ref{thm3.47} are satisfied,
so that ($i$) of Corollary \ref{coro3-31} holds.

Similarly, when ($ii$) holds, the conditions ($iii$) and ($iv_3$) in Theorem \ref{thm3.47}
are satisfied for $j(i) = i+1$,
so that ($ii$) of Corollary \ref{coro3-31} holds.

When ($iii$) holds, then, by Lemma \ref{lem3-444}, $A^{(17)}$ is an irreducible L-matrix.
From ($i$) we can prove ($iii$).

The proof is complete. \end{proof}

Obviously, from Lemma \ref{lem3-444}, if ($iii$) holds, then the assumption that
$A$ is irreducible is redundant.

The result when the condition ($i$) holds is better than \cite[Theorems 3.4, 3.5]{Yu11}
and \cite[Theorem 2]{WH09}, where $\gamma < 1$.

The result when the condition ($ii$) holds includes
\cite[Theorem 3.1, Corollaries 3.2, 3.3]{Yu11} and \cite[Theorem 4.2]{LS00}.
It is better than
\cite[Theorem 3.1, Corollary 3.1]{EM95},
\cite[Theorem 4.1]{GJS91} and \cite[Theorem 3]{LE94},
since the condition $a_{k,k+1}a_{k+1,k} > 0$, $k = 1, \cdots, n-1$,
implies that $A$ is irreducible and $a_{k,k+1} < 0$, $k = 1, \cdots, n-1$.

The corresponding result given in \cite[Theorem 2]{LLW07} is incorrect,
which is pointed out by \cite{YK08}. But the corresponding one given in
\cite[Theorem 3.5]{YK08} is also incorrect, which is pointed out by \cite{WH09}.

Similarly, from Corollary \ref{coro3-32}, we can obtain the following comparison
result, directly.

\begin{theorem}
Theorem D is valid for $\nu = 17$,
provided one of the following conditions is satisfied:

\begin{itemize}
\item[(i)] One of the conditions ($i$), ($ii$) and ($iii$) of Theorem \ref{thm3.67} holds.

\item[(ii)] $0 \le \alpha_k \le 1$, $k = 1,\cdots, n-1$ and one of the following conditions holds:

\begin{itemize}
\item[($ii_1$)] There exists $i\in\{1,\cdots,n-1\}$ such that
$\alpha_{i}a_{i,i+1}a_{i+1,i} > 0$.

\item[($ii_2$)] $\gamma > 0$ and there exist $i\in\{2,\cdots,n-1\}$ and $j\in\{1,\cdots,i-1\}$ such that
$\alpha_{i}a_{i,i+1}a_{i+1,j} > 0$.
\end{itemize}
\end{itemize}
\end{theorem}

Different from $Q_{15}$ and $Q_{17}$, we define $Q_{18} = (q^{(18)}_{i,j})$ as
\begin{eqnarray*}
q^{(18)}_{i,j} = \left\{\begin{array}{ll}
  -\alpha_ia_{i,s_i}, & i = 1, \cdots, n-1, j = s_i,\\
0, & otherwise,
\end{array}\right.
 \end{eqnarray*}
where
\begin{eqnarray*}
s_i = \min\{s: s \in \{k: |a_{i,k}| \; \hbox{is maximal for} \; i+1 \le k \le n\}\}
 \end{eqnarray*}
and $\sum_{i = 1}^{n-1}\alpha_ia_{i,s_i} \not = 0$,
which is proposed in \cite{KHMN02} for the preconditioned Gauss-Seidel method and
$\alpha_k = 1$, $k = 1, \cdots, n-1$. In \cite{KC03} its convergence for H-matrix
is discussed.

In this case, $\delta_{i,j}^{(12)}(\gamma)$ and $\delta_{i,j}^{(12)}(1)$ reduce respectively to
\begin{eqnarray*}
\delta_{i,j}^{(18)}(\gamma) = \left\{\begin{array}{ll}
 \alpha_ia_{i,s_i}a_{s_i,i}, & i = j = 1, \cdots, n-1;\\
 (\gamma-1)\alpha_ia_{i,s_i}, & i = 1, \cdots, n-1, j = s_i;\\
 \gamma\alpha_ia_{i,s_i}a_{s_i,j}, &
 i = 1, \cdots, n-1, \\
 & j \in \{1, \cdots, i-1\} \cup \{s_i+1, \cdots, n\};\\
  0, & otherwise
\end{array}\right.
\end{eqnarray*}
and
\begin{eqnarray*}
\delta_{i,j}^{(18)}(1) = \left\{\begin{array}{ll}
\alpha_ia_{i,s_i}a_{s_i,j}, &
 i = 1, \cdots, n-1, \\
& j \in \{1, \cdots, i\} \cup \{s_i+1, \cdots, n\}; \\
  0, & otherwise.
\end{array}\right.
\end{eqnarray*}

From Corollaries \ref{coro3-29}-\ref{coro3-32} we can obtain the following comparison
result, directly.

\begin{theorem}
Suppose that $0 \le \alpha_k \le 1$ and $\alpha_ka_{k,s_k}a_{s_k,k} < 1$, $k = 1, \cdots, n-1$.
Then Theorem A is valid for $\nu = 18$.
\end{theorem}

\begin{theorem}
Suppose that $0 \le \alpha_k \le 1$, $k = 1, \cdots, n-1$.
Then Theorem B is valid for $\nu = 18$.
\end{theorem}

The result for the case when $\omega = \gamma$ and $\alpha_k = 1$,
$k = 1, \cdots, n-1$, reduces to \cite[Theorem 4.2]{Li05}.

\begin{theorem}\label{thm3.71}
Suppose that $\alpha_ka_{k,s_k}a_{s_k,k} < 1$, $k = 1, \cdots, n-1$.
Then Theorem C is valid for $\nu = 18$,
provided one of the following conditions is satisfied:

\begin{itemize}
\item[(i)] $0 \le \gamma < 1$ and $0 \le \alpha_k \lesssim 1$, $k = 1,\cdots, n-1$.

\item[(ii)] $\gamma = 1$, $0 \le \alpha_k \lesssim 1$, $k = 1,\cdots, n-1$.
And one of the following conditions holds:

\begin{itemize}
\item[($ii_1$)] There exist $i\in\{1, \cdots, n-1\}$,
$j\in\{1, \cdots, i\}\cup\{s_i+1, \cdots, n\}$ such that
$\alpha_{i}a_{i,s_i}a_{s_i,j} > 0$.

\item[($ii_2$)]
$a_{n,1} < 0$ and $a_{k,k+1} < 0$, $k = 1, \cdots, n-1$.

\item[($ii_3$)]
$a_{n,1} < 0$ and $a_{k,n} < 0$, $k = 2, \cdots, n-1$.

\item[($ii_4$)]
$a_{k,1} < 0$, $k = 2, \cdots, n$.
\end{itemize}

\item[(iii)] $0 < \alpha_k \le 1$ and $a_{k,s_k} < 0$, $k = 1, \cdots, n-1$.
For each $i\in\{1, \cdots, n-1\}$
one of the following conditions holds:

\begin{itemize}
\item[($iii_1$)] $0 \le \gamma < 1$.

\item[($iii_2$)] $\gamma = 1$ and there exists
$j(i)\in\{1, \cdots, i\}\cup\{s_i+1, \cdots, n\}$ such that
$a_{s_i,j(i)} > 0$.
\end{itemize}
\end{itemize}
\end{theorem}

The result when ($iii_2$) holds for $\alpha_k = 1$, $k = 1, \cdots, n-1$, is better than
\cite[Theorem 4.3]{Li05}, since its condition insures that $A$ is irreducible.

\begin{theorem}
Theorem D is valid for $\nu = 18$,
provided one of the following conditions is satisfied:

\begin{itemize}
\item[(i)] One of the conditions ($i$), ($ii$) and ($iii$) of Theorem \ref{thm3.71} holds.

\item[(ii)] $0 \le \alpha_k \le 1$, $k = 1,\cdots, n-1$ and one of the following conditions holds:

\begin{itemize}
\item[($ii_1$)]
There exists $i\in\{1,\cdots,n-1\}$ such that
$\alpha_{i}a_{i,s_i}a_{s_i,i} > 0$.

\item[($ii_2$)]
$\gamma > 0$ and there exist $i\in\{2,\cdots,n-1\}$ and $j\in\{1,\cdots,i-1\}$ such that
$\alpha_{i}a_{i,s_i}a_{s_i,j} > 0$.
\end{itemize}
\end{itemize}
\end{theorem}

Similar to $Q_{10}$, $Q$ can be defined as
\begin{eqnarray*}
Q_{19} & = & \left(\begin{array}{cccc}
0 & \quad \cdots\quad & 0 &  -\alpha_1 a_{1,n}+\beta_1 \\
0 & \cdots & 0 & -\alpha_2 a_{2,n}+\beta_2 \\
\vdots & \vdots & \vdots & \vdots\\
0 & \cdots & 0 & -\alpha_{n-1} a_{n-1,n}+\beta_{n-1} \\
0 & \cdots & 0 & 0\\
\end{array}\right)
 \end{eqnarray*}
with $\alpha_k \ge 0$, $-\alpha_k a_{k,n}+\beta_k \ge 0$, $k = 1, \cdots, n-1$, and
 \begin{eqnarray*}
 \sum\limits_{k = 1}^{n-1}(-\alpha_k a_{k,n}+\beta_k) \not = 0.
 \end{eqnarray*}

For $\alpha_k = 1$, it is given in \cite{DH14} for the preconditioned AOR method,
and in \cite{DH11} for the preconditioned SOR method.

In this case, $\delta_{i,j}^{(1)}(\gamma)$ and $\delta_{i,j}^{(1)}(1)$ reduce respectively to
\begin{eqnarray*}
\delta_{i,j}^{(19)}(\gamma) = \left\{\begin{array}{ll}
 (\alpha_ia_{i,n}-\beta_i)a_{n,i}, & i = j = 1, \cdots, n-1;\\
 (\gamma-1)(\alpha_ia_{i,n}-\beta_i), & i = 1, \cdots, n-1, j = n;\\
 \gamma(\alpha_ia_{i,n}-\beta_i)a_{n,j}, & i = 1, \cdots, n-1, j = 1, \cdots, i-1;\\
  0, & otherwise
\end{array}\right.
\end{eqnarray*}
and
\begin{eqnarray*}
\delta_{i,j}^{(19)}(1) = \left\{\begin{array}{ll}
(\alpha_ia_{i,n}-\beta_i)a_{n,j}, & i = 1, \cdots, n-1, j =1, \cdots, i; \\
  0, & otherwise.
\end{array}\right.
\end{eqnarray*}

Similar to Lemma \ref{lem3-8}, we have the following lemma.

\begin{lemma}\label{lem3-199}
Let $A$ be an irreducible Z-matrix.
Assume that $0 < -\alpha_ka_{k,n} + \beta_k \le -a_{k,n}$, $k = 1, \cdots, n-1$ and
$A^{(19)}$ has the block form
\begin{eqnarray*}
A^{(19)} = \left(\begin{array}{ccc}
A^{(19)}_{1,1} & & \bar{a}^{(19)}_{1,2}\\
\bar{a}_{2,1} & & 1
\end{array}\right), \; A^{(19)}_{1,1}\in{\mathscr R}^{(n-1)\times(n-1)}.
 \end{eqnarray*}
 Then

\begin{itemize}
\item[(i)] $A^{(19)}_{1,1}$ is an irreducible Z-matrix.

\item[(ii)] $A^{(19)}$ is an irreducible Z-matrix if and only if there exists
$j_0\in\{1, \cdots, n-1\}$ such that $(1-\alpha_{j_0})a_{j_0,n} + \beta_{j_0} \not = 0$.
\end{itemize}
\end{lemma}

Similar to Theorems \ref{thm3.57}-\ref{thm3.60}, we can prove the following comparison theorems.

\begin{theorem}
Suppose that $0 \le -\alpha_ka_{k,n} + \beta_k \le -a_{k,n}$ and $(\alpha_ka_{k,n} - \beta_k)a_{n,k} < 1$,
$k = 1, \cdots, n-1$.
Then Theorem A is valid for $\nu = 19$.
\end{theorem}

\begin{theorem}\label{thm3.589}
Suppose that $0 \le -\alpha_ka_{k,n} + \beta_k \le -a_{k,n}$,
$k = 1, \cdots, n-1$.
Then Theorem B is valid for $\nu = 19$.
\end{theorem}

\begin{theorem}\label{thm3.599}
Suppose that $(1-\alpha_k)a_{k,n} + \beta_k \le 0$ and $(\alpha_ka_{k,n} - \beta_k)a_{n,k} < 1$,
$k = 1, \cdots, n-1$.
Then Theorem C is valid for $\nu = 19$,
provided one of the following conditions is satisfied:

\begin{itemize}
\item[(i)] $0 \le \gamma < 1$ and $(1-\alpha_k)a_{k,n} + \beta_k \lesssim 0$ whenever $a_{k,n} < 0$,
$k = 1, \cdots, n-1$.

\item[(ii)] $\gamma = 1$ and $(1-\alpha_k)a_{k,n} + \beta_k \lesssim 0$ whenever $a_{k,n} < 0$,
$k = 1, \cdots, n-1$.
And one of the following conditions holds:

\begin{itemize}
\item[($ii_1$)] There exist $i\in\{1, \cdots, n-1\}$ and
$j\in\{1, \cdots, i\}$ such that
$(\alpha_ia_{i,n}-\beta_i)a_{n,j} > 0$.

\item[($ii_2$)] $-\alpha_{n-1}a_{n-1,n} + \beta_{n-1} > 0$.

\item[($ii_3$)] $a_{n,1} < 0$.
\end{itemize}

\item[(iii)]
$0 < -\alpha_ka_{k,n} + \beta_k$, $k = 1, \cdots, n-1$.
\end{itemize}
\end{theorem}

\begin{proof}
By Theorem \ref{thm3.3} and referring to the proof of Theorem \ref{thm3.59},
we just need to prove the case when $\gamma = 1$ in ($iii$).

In fact, using Lemma \ref{lem3-199}, it is easy to prove that
a sufficient condition similar with ($iv$) of Theorem \ref{thm3.59} is that
there exist $i\in\{1, \cdots, n-1\}$ and
$j\in\{1, \cdots, i\}$ such that $(\alpha_ia_{i,n}-\beta_i)a_{n,j} > 0$,
which is equivalent to that there exists
$j\in\{1, \cdots, n-1\}$ such that $a_{n,j} < 0$,
since $-\alpha_ia_{i,n} + \beta_i > 0$. While, the irreducibility of $A$ ensures
that it is true.
\end{proof}

When ($iii$) holds, the result is better that the corresponding ones
given by \cite[Theorem 3.3, Corollary 3.3]{DH11} and \cite[Theorem 3.2]{DH14}.

\begin{theorem}\label{thm3.609}
Theorem D is valid for $\nu = 19$,
provided one of the following conditions is satisfied:

\begin{itemize}
\item[(i)] One of the conditions ($i$), ($ii$) and ($iii$) of Theorem \ref{thm3.599} holds.

\item[(ii)] $0 \le \alpha_k \le 1$, $k = 1,\cdots, n-1$ and one of the following conditions holds:

\begin{itemize}
\item[($ii_1$)]
There exists $i\in\{1,\cdots,n-1\}$ such that
$(\alpha_ia_{i,n}-\beta_i)a_{n,i} > 0$.

\item[($ii_2$)]
$\gamma > 0$ and there exist $i\in\{2,\cdots,n-1\}$ and $j\in\{1,\cdots,i-1\}$
such that $(\alpha_ia_{i,n}-\beta_i)a_{n,j} > 0$.
\end{itemize}
\end{itemize}
\end{theorem}

As a special of $Q_{19}$, $Q$ is proposed in \cite{Li11} as
\begin{eqnarray*}
Q_{20} & = & \left(\begin{array}{cccc}
0 & \quad \cdots\quad & 0 & \quad -\frac{a_{1,n}}{\alpha} - \beta\\
0 & \cdots & 0 & 0\\
\vdots & \vdots & \vdots & \vdots\\
0 & \cdots & 0 & 0\\
\end{array}\right)
 \end{eqnarray*}
with $\alpha > 0$ and $a_{1,n}/\alpha + \beta < 0$, which is given in \cite{LW042} for $\alpha = 1$.

In this case, similar to $\delta_{i,j}^{(16)}(1)$, we can derive
$\delta_{i,j}^{(20)}(1)$, whose all elements are zero except
$\delta_{1,1}^{(20)}(1) = (a_{1,n}/\alpha + \beta)a_{n,1}$.

\begin{theorem}
Suppose that
\begin{eqnarray}\label{eqs3.45}
 \left(1 - \frac{1}{\alpha}\right)a_{1,n} \le \beta < -\frac{a_{1,n}}{\alpha}
 \end{eqnarray}
and
\begin{eqnarray}\label{eqs3.46}
1 - \left(\frac{a_{1,n}}{\alpha} + \beta\right)a_{n,1} > 0.
\end{eqnarray}
Then Theorem A is valid for $\nu = 20$.
\end{theorem}

\begin{proof}
Since
\begin{eqnarray*}
 \left(1 - \frac{1}{\alpha}\right)a_{1,n} \le \beta \;\; \hbox{iff} \;\;
-\frac{a_{1,n}}{\alpha} - \beta \le -a_{1,n},
 \end{eqnarray*}
and the inequality (\ref{eqs3.11}) reduces to (\ref{eqs3.46}), then
the condition of Corollary \ref{coro3-5} is satisfied so that Theorem A is valid.  \end{proof}

It is easy to prove the following theorem.

\begin{theorem}
Suppose that (\ref{eqs3.45}) holds.
Then Theorem B is valid for $\nu = 20$.
\end{theorem}

In order to give the Stein-Rosenberg Type Theorem II,
completely similar to Lemma \ref{lem3-9}, we can prove the following lemma.

\begin{lemma} \label{lem3-10}
Let $A$ be an irreducible Z-matrix.
Assume that $\beta \ge (1 - 1/\alpha)a_{1,n}$ and $A^{(20)}$ has the block form
\begin{eqnarray*}
A^{(20)} = \left(\begin{array}{ccc}
A^{(20)}_{1,1} & & \bar{a}^{(20)}_{1,2}\\
\bar{a}_{2,1} & & 1
\end{array}\right), \; A^{(20)}_{1,1}\in{\mathscr R}^{(n-1)\times(n-1)}.
 \end{eqnarray*}
 Then one of the following two mutually exclusive relations holds:

\begin{itemize}
\item[(i)] $A^{(20)}$ is an irreducible Z-matrix.

\item[(ii)] $A^{(20)}$ is a reducible Z-matrix,
but $A^{(20)}_{1,1}$ is an irreducible Z-matrix and
$a_{1,k} = a^{(20)}_{1,k} = a^{(20)}_{1,n} = 0$, $k = 2,\cdots,n-1$.
\end{itemize}
\end{lemma}

Using this lemma, completely similar to ($ii$) of Theorem \ref{thm3.63}, we can prove the following theorem.

\begin{theorem}
Suppose that
$\beta \ge (1 - 1/\alpha)a_{1,n}$ and $(a_{1,n}/\alpha + \beta)a_{n,1} < 1$.
Then Theorem C is valid for $\nu = 20$.
\end{theorem}

The result is better than \cite[Theorem 6]{LW14},
where the condition $0 < a_{1,n}a_{n,1} < \alpha (\alpha > 1)$ is unnecessary.

\begin{theorem}
Suppose that
$\beta \ge (1 - 1/\alpha)a_{1,n}$.
Then Theorem D is valid for $\nu = 20$.
\end{theorem}

By the definition of $Q_{20}$, $a_{1,n}/\alpha + \beta < 0$. While
in the comparison theorems above we need the condition
$\beta \ge (1 - 1/\alpha)a_{1,n}$. Hence it implies $a_{1,n}<0$.

For $\alpha = 1$, the result is better than  the corresponding one given by
\cite[Theorem 9, Corollaries 10, 11]{LW042},
where the condition $a_{1,n}a_{n,1} > 0$ is unnecessary.

\subsection{Combination preconditioners}

In this subsection, when the matrix $Q$ is composed of two different combinations of
$Q_i$ and $Q_j$, we always assume $Q_i > 0$ and $Q_j > 0$. Otherwise,
the corresponding situation has been discussed in the above two subsections.

First, the matrix $Q$ is chosen as
 \begin{eqnarray*}
Q_{21} = Q_{5} + Q_{12},
 \end{eqnarray*}
 i.e.,
 \begin{eqnarray*}
Q_{21} = \left(\begin{array}{ccccc}
0      & -\beta_{1,2}a_{1,2}    & \cdots & -\beta_{1,n-1}a_{1,n-1} & -\beta_{1,n}a_{1,n}\\
0      & 0                  & \ddots & -\beta_{2,n-1}a_{2,n-1}     & -\beta_{2,n}a_{2,n}\\
\vdots & \vdots             & \ddots & \ddots & \vdots\\
0      & 0                  & \cdots & 0      & -\beta_{n-1,n}a_{n-1,n} \\
-\alpha_{1}a_{n,1} & -\alpha_{2}a_{n,2} & \cdots & -\alpha_{n-1}a_{n,n-1} & 0
\end{array}\right)
 \end{eqnarray*}
with $\alpha_i \ge 0$, $\beta_{i,j} \ge 0$, $i = 1, \cdots, n-1$, $j = i+1, \cdots, n$, and
\begin{eqnarray*}
\sum\limits_{k = 1}^{n-1}\alpha_ka_{n,k} \not = 0 \;\;
\hbox{and} \;\; \sum\limits_{i = 1}^{n-1}\sum\limits_{j=i+1}^n\beta_{i,j}a_{i,j} \not = 0,
 \end{eqnarray*}
 where for simplicity we set $r = n$ for $Q_5$.

When $\alpha_i = \beta_{i,j} = 1$, $i = 1, \cdots, n-1$, $j = i+1, \cdots, n$,
it is proposed in \cite{YZ12} for the preconditioned Gauss-Seidel method.

By Corollaries \ref{coro3-9} and \ref{coro3-10}, the following two comparison theorems are directly.

\begin{theorem}\label{thm3.77}
Suppose that $0 \le \alpha_i \le 1$, $0 \le \beta_{i,j} \le 1$,
$i = 1, \cdots, n-1$, $j = i+1, \cdots, n$ and
\begin{eqnarray} \label{eqs3.48}
\sum\limits_{k = 1}^{n-1}\alpha_{k}a_{n,k}a_{k,n} < 1, \;
\sum\limits_{k = i+1}^{n}\beta_{i,k}a_{i,k}a_{k,i} < 1, \; i = 1, \cdots, n-1.
 \end{eqnarray}
Then Theorem A is valid for $\nu = 21$.
\end{theorem}

\begin{theorem}
Suppose that $0 \le \alpha_i \le 1$, $0 \le \beta_{i,j} \le 1$,
$i = 1, \cdots, n-1$, $j = i+1, \cdots, n$.
Then Theorem B is valid for $\nu = 21$.
\end{theorem}

In this case, $\delta_{i,j}^{(2)}(\gamma)$ and $\delta_{i,j}^{(2)}(1)$ reduce
respectively to
\begin{eqnarray*}
\delta_{i,j}^{(21)}(\gamma) = \left\{\begin{array}{ll}
\sum\limits_{k = i+1}^{n} \beta_{i,k}a_{i,k}a_{k,i}, & i = j = 1, \cdots, n-1;\\
\sum\limits_{k = 1}^{n-1} \alpha_{k}a_{n,k}a_{k,n}, & i = j = n;\\
(\gamma-1)\beta_{i,j}a_{i,j} + \gamma\sum\limits_{k = i+1}^{j-1}\beta_{i,k}a_{i,k}a_{k,j},
 & \begin{array}{l} i = 1, \cdots, n-1, \\
 j = i+1,\cdots, n;  \end{array}\\
 \gamma\sum\limits_{k = i+1}^{n}\beta_{i,k}a_{i,k}a_{k,j}, &
  \begin{array}{l} i = 2,\cdots,n-1, \\
  j = 1,\cdots, i-1; \end{array} \\
 (\gamma-1)\alpha_{j}a_{n,j} + \gamma\sum\limits_{k = 1}^{j-1}\alpha_{k}a_{n,k}a_{k,j}, &
 \begin{array}{l}  i = n, \\
 j = 1,\cdots, n-1 \end{array}
\end{array}\right.
\end{eqnarray*}
and
\begin{eqnarray*}
\delta_{i,j}^{(21)}(1) = \left\{\begin{array}{ll}
 \sum\limits_{k = i+1}^n\beta_{i,k}a_{i,k}a_{k,j}, &
  i = 1,\cdots,n-1, j = 1,\cdots, i; \\
 \sum\limits_{k = i+1}^{j-1} \beta_{i,k}a_{i,k}a_{k,j}, &
 i = 1,\cdots,n-1, j = i+2,\cdots,n;\\
\sum\limits_{k = 1}^{j-1}\alpha_{k}a_{n,k}a_{k,j}, &
  i = n, j = 2,\cdots, n; \\
  0, & otherwise.
\end{array}\right.
\end{eqnarray*}

Using Corollary \ref{coro3-11}, we prove Stein-Rosenberg type comparison theorem.

\begin{theorem}\label{thm3.79}
Suppose that (\ref{eqs3.48}) holds.
Then Theorem C is valid for $\nu = 21$,
provided one of the following conditions is satisfied:

\begin{itemize}
\item[(i)] $0 \le \gamma < 1$ and $0 \le \alpha_i \lesssim 1$, $0 \le \beta_{i,j} \lesssim 1$,
$i = 1, \cdots, n-1$, $j = i+1, \cdots, n$.

\item[(ii)] $\gamma = 1$ and $0 \le \alpha_i \lesssim 1$, $0 \le \beta_{i,j} \lesssim 1$,
$i = 1, \cdots, n-1$, $j = i+1, \cdots, n$. And one of the following conditions holds:

\begin{itemize}
\item[($ii_1$)]
There exist $i\in\{1, \cdots, n-1\}$, $j\in\{1, \cdots, i\}$ and
$k\in\{i+1, \cdots, n\}$ such that
$\beta_{i,k}a_{i,k}a_{k,j} > 0$.

\item[($ii_2$)]
There exist $i\in\{1, \cdots, n-1\}$, $j\in\{i+2, \cdots, n\}$ and
$k\in\{i+1, \cdots, j-1\}$ such that
$\beta_{i,k}a_{i,k}a_{k,j} > 0$.

\item[($ii_3$)] There exist $j\in\{2, \cdots, n\}$ and
$k\in\{1, \cdots, j-1\}$ such that $\alpha_{k}a_{n,k}a_{k,j}$ $ > 0$.

\item[($ii_4$)]
There exists $k\in\{1, \cdots, n-1\}$ such that
$a_{k,k+1} < 0$ and $\beta_{k,k+1} > 0$.

\item[($ii_5$)] $a_{n,1} < 0$ and $\alpha_1 > 0$.

\item[($ii_6$)]
$a_{k,1} < 0$, $k = 2, \cdots, n$.

\item[($ii_7$)]
$a_{n,1} < 0$ and $a_{k,n} < 0$, $k = 2, \cdots, n-1$.

\item[($ii_8$)] $a_{k,n} < 0$, $k = 1, \cdots, n-1$.

\item[($ii_9$)] $a_{k,k+1} < 0$, $k = 1, \cdots, n-1$.
\end{itemize}

\item[(iii)] $0 \le \gamma < 1$ and $0 \le \alpha_i \le 1$, $0 \le \beta_{i,j} \le 1$,
$i = 1, \cdots, n-1$, $j = i+1, \cdots, n$. And for each $i\in\{1, \cdots, n-1\}$,
there exists $j(i)\in\{i+1, \cdots, n\}$ such that $\beta_{i,j(i)}a_{i,j(i)} < 0$.

\item[(iv)] $\gamma = 1$ and $0 \le \alpha_i \le 1$, $0 \le \beta_{i,j} \le 1$,
$i = 1, \cdots, n-1$, $j = i+1, \cdots, n$. For each $i\in\{2, \cdots, n-1\}$,
one of the following conditions holds:

\begin{itemize}
\item[($iv_1$)] There exist $j(i)\in\{1, \cdots, i\}$ and
$k(i)\in\{i+1, \cdots, n\}$ such that\\
$\beta_{i,k(i)}a_{i,k(i)}a_{k(i),j(i)} > 0$.

\item[($iv_2$)] There exist $j(i)\in\{i+2, \cdots, n\}$ and $k(i)\in\{i+1, \cdots, j-1\}$ such that
$\beta_{i,k(i)}a_{i,k(i)}a_{k(i),j(i)} > 0$.

\item[($iv_3$)] $a_{i,i+1} < 0$ and $\beta_{i,i+1} > 0$.
\end{itemize}

 At the same time, one of the following conditions also holds:

\begin{itemize}
\item[($iv^a$)]
There exist $j\in\{2, \cdots, n\}$ and $k\in\{1, \cdots, j-1\}$ such that
$\alpha_{k}a_{n,k}a_{k,j}$ $ > 0$.

\item[($iv^b$)] $a_{n,1} < 0$ and $\alpha_{1} > 0$.

\item[($iv^c$)] There exists $j\in\{2, \cdots, n-1\}$ such that
\begin{eqnarray*}
(1 - \alpha_{j})a_{n,j} - \sum\limits_{{k = 1}\atop{k\not = j}}^{n-1}\alpha_{k}a_{n,k}a_{k,j} < 0.
\end{eqnarray*}

\item[($iv^d$)] One of the conditions ($iv_1$)-($iv_3$) holds for $i=1$ and
\begin{eqnarray*}
(1 - \alpha_{1})a_{n,1} - \sum\limits_{k = 2}^{n-1}\alpha_{k}a_{n,k}a_{k,1} < 0.
\end{eqnarray*}

\item[($iv^e$)] One of the conditions ($iv_1$)-($iv_3$) holds for $i=1$ and $a_{n,1} < 0$.
\end{itemize}
\end{itemize}
\end{theorem}

\begin{proof}
When ($i$) or ($iii$) holds, then the condition ($i$) or ($iii$) of Theorem \ref{thm3.7} is satisfied.

When one of ($ii_1$), ($ii_2$) and ($ii_3$) holds, then $\delta_{i,j}^{(21)}(1) > 0$ or
$\delta_{n,j}^{(21)}(1) > 0$, which implies that ($ii_1$) in Theorem \ref{thm3.7} is satisfied.

When ($ii_4$) or ($ii_5$) holds, then ($ii_3$) or ($ii_2$) in Theorem \ref{thm3.7} is satisfied.

When one of ($ii_6$) and ($ii_7$) holds, then by the proof of Theorem \ref{thm3.47} it
follows that ($ii_1$) is satisfied.

Similarly, when one of ($ii_8$) and ($ii_9$) holds, then by the proof of Theorem \ref{thm3.19} it
follows that ($ii_3$) is satisfied.

Exactly the same, we can prove that if ($iv_4$) holds, then
the condition ($iv_4$) in Theorem \ref{thm3.7} is satisfied.

By Corollary \ref{coro3-11} the proof is complete.
\end{proof}

\begin{theorem}\label{thm3.80}
Theorem D is valid for $\nu = 21$,
provided one of the following conditions is satisfied:

\begin{itemize}
\item[(i)] One of the conditions ($i$)-($iv$) of Theorem \ref{thm3.79} holds.

\item[(ii)] $0 \le \alpha_i \le 1$, $0 \le \beta_{i,j} \le 1$,
$i = 1, \cdots, n-1$, $j = i+1, \cdots, n$. For $j = 1,\cdots,n-1$,
\begin{eqnarray*}
 \alpha_{j}a_{n,j} + \sum\limits_{{k = 1}\atop{k\not = j}}^{n-1} \alpha_{k}a_{n,k}a_{k,j}
\ge 0.
\end{eqnarray*}
And one of the following conditions holds:

\begin{itemize}
\item[($ii_1$)] There exists $i_0\in\{1,\cdots,n-1\}$ such that
$\sum\limits_{k = i_0+1}^{n} \beta_{i_0,k}a_{i_0,k}a_{k,i_0} > 0$.

\item[($ii_2$)]
$\sum_{k = 1}^{n-1} \alpha_{k}a_{n,k}a_{k,n} > 0$.

\item[($ii_3$)] $\gamma > 0$ and there exist $i_0\in\{2,\cdots,n-1\}$ and $j_0\in\{1,\cdots,i_0-1\}$ such that
$\sum_{k = i_0+1}^n\beta_{i_0,k}a_{i_0,k}a_{k,j_0} > 0$.

\item[($ii_4$)] $\gamma > 0$ and there exists $j_0\in\{1,\cdots,n-1\}$ such that
$\alpha_{j_0}a_{n,j_0} + $\\ $\sum_{k = 1,k\not =j_0}^{n-1}\alpha_{k}a_{n,k}a_{k,j_0} > 0$.
\end{itemize}\end{itemize}
\end{theorem}

\begin{proof}
We just need to prove ($ii$). It is easy to obtain that
\begin{eqnarray*}
a_{i,j}^{(21)} = \left\{\begin{array}{ll}
1- \sum\limits_{k = i+1}^{n}\alpha_{i,k}a_{i,k}a_{k,i}, &
  i = j = 1,\cdots,n-1, \\
  1- \sum\limits_{k = 1}^{n-1}\alpha_{k}a_{n,k}a_{k,n}, &
  i = j = n, \\
 a_{i,j} - \sum\limits_{k = i+1}^{n}\beta_{k}a_{i,k}a_{k,j}, &
  \begin{array}{l} i = 2,\cdots,n-1, \\
  j = 1,\cdots, i-1, \end{array}\\
  (1-\alpha_{j})a_{n,j} - \sum\limits_{k = 1}^{n-1}\alpha_{k}a_{n,k}a_{k,j}, &
  i = n,  j = 1,\cdots, n-1.
\end{array}\right.
\end{eqnarray*}
By Corollary \ref{coro3-12} we can derive ($ii_1$)-($ii_4$).
\end{proof}

As a special case of $Q_{21}$ for $\alpha_k = 0$, $k = 2, \cdots, n-1$,
it gets that
 \begin{eqnarray*}
Q_{22} = Q_{7} + Q_{12},
 \end{eqnarray*}
 i.e.,
 \begin{eqnarray*}
Q_{22} = \left(\begin{array}{ccccc}
0 & -\alpha_{1,2}a_{1,2} & -\alpha_{1,3}a_{1,3} & \cdots & -\alpha_{1,n}a_{1,n}\\
0 & 0 & -\alpha_{2,3}a_{2,3} & \cdots & -\alpha_{2,n}a_{2,n}\\
\vdots & \vdots & \ddots & \ddots & \vdots\\
0 & 0 & 0 & \cdots & -\alpha_{n-1,n}a_{n-1,n} \\
-\alpha a_{n,1} & 0 & 0 & \cdots & 0
\end{array}\right)
 \end{eqnarray*}
with $a_{n,1} < 0$, $\alpha > 0$, $\beta_{i,j} \ge 0$, $i = 1, \cdots, n-1$, $j = i+1, \cdots, n$, and
$\sum_{i = 1}^{n-1}\sum_{j=i+1}^n\beta_{i,j}a_{i,j}$ $ \not = 0$.

In this case, ($ii_5$) and ($iv^b$) in Theorem \ref{thm3.79} are satisfied. While,
($ii$) of Theorem \ref{thm3.80} can be not satisfied. Hence,
form Theorems \ref{thm3.77}-\ref{thm3.80}, we can derive the following theorems, directly.

\begin{theorem}
Suppose that $0 < \alpha \le 1$, $0 \le \beta_{i,j} \le 1$,
$i = 1, \cdots, n-1$, $j = i+1, \cdots, n$, and
\begin{eqnarray} \label{eqs3.49}
\alpha a_{n,1}a_{1,n} < 1, \;
\sum\limits_{k = i+1}^{n}\beta_{i,k}a_{i,k}a_{k,i} < 1, \; i = 1, \cdots, n-1.
 \end{eqnarray}
Then Theorem A is valid for $\nu = 22$.
\end{theorem}

\begin{theorem}
Suppose that $0 < \alpha \le 1$, $0 \le \beta_{i,j} \le 1$,
$i = 1, \cdots, n-1$, $j = i+1, \cdots, n$.
Then Theorem B is valid for $\nu = 22$.
\end{theorem}

\begin{theorem}\label{thm3.83}
Suppose that (\ref{eqs3.49}) holds.
Then Theorem C is valid for $\nu = 22$,
provided one of the following conditions is satisfied:

\begin{itemize}
\item[(i)]
$0 < \alpha \lesssim 1$, $0 \le \beta_{i,j} \lesssim 1$,
 $i = 1, \cdots, n-1$, $j = i+1, \cdots, n$.

\item[(ii)] $0 \le \gamma < 1$ and $0 < \alpha \le 1$, $0 \le \beta_{i,j} \le 1$,
$i = 1, \cdots, n-1$, $j = i+1, \cdots, n$. And for each $i\in\{1, \cdots, n-1\}$,
there exists $j(i)\in\{i+1, \cdots, n\}$ such that $\beta_{i,j(i)}a_{i,j(i)} < 0$.

\item[(iii)] $\gamma = 1$ and $0 < \alpha \le 1$, $0 \le \beta_{i,j} \le 1$,
$i = 1, \cdots, n-1$, $j = i+1, \cdots, n$. For each $i\in\{2, \cdots, n-1\}$,
one of the following conditions holds:

\begin{itemize}
\item[($iii_1$)] There exist $j(i)\in\{1, \cdots, i\}$ and
$k(i)\in\{i+1, \cdots, n\}$ such that\\
$\beta_{i,k(i)}a_{i,k(i)}a_{k(i),j(i)} > 0$.

\item[($iii_2$)] There exist $j(i)\in\{i+2, \cdots, n\}$ and $k(i)\in\{i+1, \cdots, j-1\}$ such that
$\beta_{i,k(i)}a_{i,k(i)}a_{k(i),j(i)} > 0$.

\item[($iii_3$)] $a_{i,i+1} < 0$ and $\beta_{i,i+1} > 0$.
\end{itemize}
\end{itemize}
\end{theorem}

\begin{theorem}
Theorem D is valid for $\nu = 22$,
provided one of the conditions ($i$), ($ii$) and ($iii$) of Theorem \ref{thm3.83} is satisfied.
\end{theorem}

Similar to $Q_{21}$, the matrix $Q$ is chosen as
 \begin{eqnarray*}
Q_{23} = Q_{3} + Q_{17},
 \end{eqnarray*}
 i.e.,
 \begin{eqnarray*}\small
Q_{23} = \left(\begin{array}{ccccc}
0 & -\beta_1a_{1,2}  & \cdots & 0 & 0\\
-\alpha_{2,1}a_{2,1} & 0  & \cdots & 0 & 0\\
\vdots & \ddots & \ddots & \ddots & \vdots\\
-\alpha_{n-1,1}a_{n-1,1} & -\alpha_{n-1,2}a_{n-1,2} & \cdots & 0 & -\beta_{n-1}a_{n-1,n} \\
-\alpha_{n,1}a_{n,1} & -\alpha_{n,2}a_{n,2} &  \cdots &
-\alpha_{n,n-1}a_{n,n-1} & 0
\end{array}\right)
 \end{eqnarray*}
with $\alpha_{i,j} \ge 0$, $i = 2, \cdots, n$, $j < i$ and $\beta_i \ge 0$, $i = 1, \cdots, n-1$, and
\begin{eqnarray*}
\sum\limits_{i = 2}^n\sum\limits_{j = 1}^{i-1}\alpha_{i,j}a_{i,j} \not = 0 \;\;
\hbox{and} \;\;  \sum\limits_{i = 1}^{n-1}\beta_ia_{i,i+1} \not = 0.
\end{eqnarray*}

By Corollaries \ref{coro3-9} and \ref{coro3-10}, the following two comparison theorems are directly.

\begin{theorem}\label{thm3.85}
Suppose that $0 \le \alpha_{i+1,j} \le 1$, $0 \le \beta_i \le 1$,
$i = 1, \cdots, n-1$, $j = 1, \cdots, i$, and $\beta_i a_{i,i+1}a_{i+1,i} +
\sum_{k = 1}^{i-1}\alpha_{i,k}a_{i,k}a_{k,i} < 1$, $i = 1, \cdots, n$.
Then Theorem A is valid for $\nu = 23$.
\end{theorem}

\begin{theorem}\label{thm3.86}
Suppose that $0 \le \alpha_{i+1,j} \le 1$, $0 \le \beta_i \le 1$,
$i = 1, \cdots, n-1$, $j = 1, \cdots, i$.
Then Theorem B is valid for $\nu = 23$.
\end{theorem}

In this case, $\delta_{i,j}^{(2)}(\gamma)$ and $\delta_{i,j}^{(2)}(1)$ reduce
respectively to
\begin{eqnarray*}
\delta_{i,j}^{(23)}(\gamma) = \left\{\begin{array}{ll}
\sum\limits_{k = 1}^{i-1} \alpha_{i,k}a_{i,k}a_{k,i} + \beta_ia_{i,i+1}a_{i+1,i}, & i = j = 1, \cdots, n;\\
 (\gamma-1)\beta_ia_{i,i+1}, & i = 1, \cdots, n-1, \\& j = i+1;\\
 \gamma\beta_ia_{i,i+1}a_{i+1,j}, & i = 1, \cdots, n-1, \\
 & j = i+2, \cdots, n;\\
 (\gamma-1)\alpha_{i,j}a_{i,j} + \gamma\sum\limits_{k = 1}^{j-1}\alpha_{i,k}a_{i,k}a_{k,j} & \\
 + \gamma\beta_ia_{i,i+1}a_{i+1,j}, &
  i = 2,\cdots,n, \\
  & j = 1,\cdots, i-1
\end{array}\right.
\end{eqnarray*}
and
\begin{eqnarray*}
\delta_{i,j}^{(23)}(1) = \left\{\begin{array}{ll}
\sum\limits_{k = 1}^{j-1}\alpha_{i,k}a_{i,k}a_{k,j} + \beta_ia_{i,i+1}a_{i+1,j}, &
  i = 1,\cdots,n, j = 1,\cdots, i; \\
\beta_ia_{i,i+1}a_{i+1,j}, & i = 1, \cdots, n-1,\\
& j = i+2, \cdots, n;\\
  0, & otherwise.
\end{array}\right.
\end{eqnarray*}

Using Corollary \ref{coro3-11}, we prove Stein-Rosenberg type comparison theorem.

\begin{theorem}\label{thm3.87}
Suppose that $\beta_i a_{i,i+1}a_{i+1,i} +
\sum_{k = 1}^{i-1}\alpha_{i,k}a_{i,k}a_{k,i} < 1$, $i = 1, \cdots, n$.
Then Theorem C is valid for $\nu = 23$,
provided one of the following conditions is satisfied:

\begin{itemize}
\item[(i)] For $i = 1, \cdots, n-1$, $j = 1, \cdots, i$,
$0 \le \alpha_{i+1,j} \lesssim 1$,
$0 \le \beta_i \lesssim 1$.

\item[(ii)] $0 \le \gamma < 1$ and
$0 \le \alpha_{i+1,j} \le 1$, $0 \le \beta_i \le 1$,
$i = 1, \cdots, n-1$, $j = 1, \cdots, i$.
For each $i\in\{1, \cdots, n-1\}$,
$\beta_ia_{i,i+1} < 0$ or there exists $j(i)\in\{1, \cdots, i-1\}$ such that $\alpha_{i,j(i)}a_{i,j(i)} < 0$.

\item[(iii)] $\gamma = 1$ and
$0 \le \alpha_{i+1,j} \le 1$, $0 \le \beta_i \le 1$,
$i = 1, \cdots, n-1$, $j = 1, \cdots, i$.
For each $i\in\{2, \cdots, n-1\}$, $\beta_ia_{i,i+1} < 0$
or there exist $j(i)\in\{1, \cdots, i\}$ and $k(i)\in\{1, \cdots, j(i)-1\}$
such that $\alpha_{i,k(i)}a_{i,k(i)}a_{k(i),j(i)} > 0$.
At the same time, one of the following conditions holds:

\begin{itemize}
\item[($iii_1$)]
There exist $j\in\{2, \cdots, n\}$ and $k\in\{1, \cdots, j-1\}$ such that
$\alpha_{n,k}a_{n,k}a_{k,j}$ $ > 0$.

\item[($iii_2$)] $a_{n,1} < 0$ and $\alpha_{n,1} > 0$.

\item[($iii_3$)] There exists $j\in\{2, \cdots, n-1\}$ such that
\begin{eqnarray*}
(1 - \alpha_{n,j})a_{n,j} - \sum\limits_{{k = 1}\atop{k\not = j}}^{n-1}\alpha_{n,k}a_{n,k}a_{k,j} < 0.
\end{eqnarray*}

\item[($iii_4$)] $a_{1,2} < 0$, $\beta_1 > 0$ and
\begin{eqnarray*}
(1 - \alpha_{n,1})a_{n,1} - \sum\limits_{k = 2}^{n-1}\alpha_{n,k}a_{n,k}a_{k,1} < 0.
\end{eqnarray*}

\item[($iii_5$)] $a_{1,2} < 0$,  $a_{n,1} < 0$ and $\beta_1 > 0$.
\end{itemize}
\end{itemize}
\end{theorem}

\begin{proof}
Since $\sum_{k = 1}^{n-1}\beta_ka_{k,k+1} \not = 0$, then there exists
$k\in\{1, \cdots, n-1\}$ such that $a_{k,k+1} < 0$ and $\beta_k > 0$, which shows that the condition
($ii_3$) in Theorem \ref{thm3.7} is satisfied, so that the condition ($i$) of Corollary \ref{coro3-11}
is satisfied. This shows ($i$).

When ($ii$) holds, then the condition ($iii$) of Theorem \ref{thm3.7} is satisfied,
so that the condition ($ii$) of Corollary \ref{coro3-11} is satisfied.

When ($iii$) holds, it gets that $\delta_{i,j(i)}^{(23)}(1) > 0$. In this case
$\delta_{1,j}^{(23)}(1) = \beta_1a_{1,2}a_{2,j}$, $j = 1, 3, \cdots, n$. Since $A$
is irreducible, then there exists $j \in \{1\}\cup\{3, \cdots, n\}$ such that $a_{2,j}<0$,
so that $\delta_{1,j}^{(23)}(1) = \beta_1a_{1,2}a_{2,j} > 0$ whenever
$\beta_1>0$ and $a_{1,2}<0$. Hence, the conditions
($iv^a$)-($iv^e$) in Theorem \ref{thm3.7} reduce to ($iii_1$)-($iii_5$), respectively.
By ($ii$) of Corollary \ref{coro3-11} the proof of ($iii$) is complete.
\end{proof}

\begin{theorem}\label{thm3.88}
Theorem D is valid for $\nu = 23$,
provided one of the following conditions is satisfied:

\begin{itemize}
\item[(i)] One of the conditions ($i$), ($ii$) and ($iii$) of Theorem \ref{thm3.87} holds.

\item[(ii)] $0 \le \alpha_{i+1,j} \le 1$, $0 \le \beta_i \le 1$,
$i = 1, \cdots, n-1$, $j = 1, \cdots, i$.
For $i = 2,\cdots,n$, $j = 1,\cdots,i-1$,
\begin{eqnarray*}
 \alpha_{i,j}a_{i,j} + \sum\limits_{{k = 1}\atop{k\not = j}}^{i-1} \alpha_{i,k}a_{i,k}a_{k,j}
 + \beta_ia_{i,i+1}a_{i+1,j} \ge 0.
\end{eqnarray*}
And one of the following conditions holds:

\begin{itemize}
\item[($ii_1$)] There exists $i_0\in\{1,\cdots,n-1\}$ such that
$\beta_{i_0}a_{i_0,i_0+1}a_{i_0+1,i_0} > 0$.

\item[($ii_2$)] There exist $i_0\in\{2,\cdots,n\}$ and $j_0\in\{1,\cdots,i_0-1\}$ such that\\
$\alpha_{i_0,j_0}a_{i_0,j_0}a_{j_0,i_0}$ $ > 0$.

\item[($ii_3$)] $\gamma > 0$ and $\alpha_{2,1}a_{2,1} + \beta_{2}a_{2,3}a_{3,1} > 0$.

\item[($ii_4$)] $\gamma > 0$ and there exist $i_0\in\{3,\cdots,n\}$ and $j_0\in\{1,\cdots,i_0-1\}$ such that
 \begin{eqnarray*}
\alpha_{i_0,j_0}a_{i_0,j_0} + \sum\limits_{{k = 1}\atop{k\not =j_0}}^{i_{0}-1}\alpha_{i_0,k}a_{i_0,k}a_{k,j_0}
+ \beta_{i_0}a_{i_0,i_0+1}a_{i_0+1,j_0} > 0.
\end{eqnarray*}
\end{itemize}\end{itemize}
\end{theorem}

\begin{proof}
We just need to prove ($ii$). It is easy to obtain
\begin{eqnarray*}
a_{i,j}^{(23)} = \left\{\begin{array}{ll}
1- \sum\limits_{k = 1}^{i-1}\alpha_{i,k}a_{i,k}a_{k,i} - \beta_ia_{i,i+1}a_{i+1,i}, &
  i = j = 1,\cdots,n, \\
(1-\alpha_{i,j})a_{i,j} - \sum\limits_{k = 1}^{i-1}\alpha_{i,k}a_{i,k}a_{k,j} - \beta_ia_{i,i+1}a_{i+1,j}, &
  i = 2,\cdots,n, \\
  & j = 1,\cdots, i-1.
\end{array}\right.
\end{eqnarray*}
By Corollary \ref{coro3-12} we can derive ($ii_1$)-($ii_4$).
\end{proof}

As a special case, $Q_{23}$ and $Q_{21}$ reduce to
 \begin{eqnarray*}
Q_{24} & = & Q_5 + Q_{17}  \\
& = & \left(\begin{array}{ccccc}
0      & -\beta_1a_{1,2}    & \cdots & 0      & 0\\
0      & 0                  & \ddots & 0      & 0\\
\vdots & \vdots             & \ddots & \ddots & \vdots\\
0      & 0                  & \cdots & 0      & -\beta_{n-1}a_{n-1,n} \\
-\alpha_{1}a_{n,1} & -\alpha_{2}a_{n,2} & \cdots & -\alpha_{n-1}a_{n,n-1} & 0
\end{array}\right)
 \end{eqnarray*}
with $\alpha_k \ge 0$, $\beta_k \ge 0$, $k = 1, \cdots, n-1$, and
\begin{eqnarray*}
\sum\limits_{k = 1}^{n-1}\alpha_ka_{n,k} \not = 0 \;\;
\hbox{and} \;\; \sum\limits_{k = 1}^{n-1}\beta_ka_{k,k+1} \not = 0,
 \end{eqnarray*}
 where for simplicity we set $r = n$ for $Q_5$.

It is proposed in \cite{NHMS04} for the preconditioned Gauss-Seidel method,
where $\alpha_k = \beta_k = 1$, $k = 1, \cdots, n-1$.

In this case, $\delta_{i,j}^{(23)}(\gamma)$ and $\delta_{i,j}^{(23)}(1)$ reduce respectively to
\begin{eqnarray*}
\delta_{i,j}^{(24)}(\gamma) = \left\{\begin{array}{ll}
 \beta_ia_{i,i+1}a_{i+1,i}, & i = j = 1, \cdots, n-1;\\
\sum\limits_{k = 1}^{n-1} \alpha_ka_{n,k}a_{k,n}, &
 i = j = n;\\
 (\gamma-1)\beta_ia_{i,i+1}, & i = 1, \cdots, n-1, j = i+1;\\
 \gamma\beta_ia_{i,i+1}a_{i+1,j}, & i = 1, \cdots, n-1, \\
 & j = 1, \cdots, n, j\not = i, i+1; \\
 (\gamma-1)\alpha_ja_{n,j} + \gamma\sum\limits_{k = 1}^{j-1} \alpha_ka_{n,k}a_{k,j},
 & i = n, j = 1,\cdots, n-1
\end{array}\right.
\end{eqnarray*}
and
\begin{eqnarray*}
\delta_{i,j}^{(24)}(1) = \left\{\begin{array}{ll}
\beta_ia_{i,i+1}a_{i+1,j}, & i = 1, \cdots, n-1,
 j = 1, \cdots, n, j\not = i+1; \\
 \sum\limits_{k = 1}^{j-1} \alpha_ka_{n,k}a_{k,j},
 & i = n, j = 1,\cdots, n.
\end{array}\right.
\end{eqnarray*}

From Theorems \ref{thm3.85}-\ref{thm3.88} the following comparison results are immediately.

\begin{theorem}\label{thm3.89}
Suppose that $0 \le \alpha_k, \beta_k \le 1$, $k = 1, \cdots, n-1$, and
 \begin{eqnarray} \label{eqs3.52}
\sum\limits_{k = 1}^{n-1}\alpha_ka_{n,k}a_{k,n} < 1, \;
 \beta_ia_{i,i+1}a_{i+1,i} < 1, \; i = 1, \cdots, n-1.
 \end{eqnarray}
Then Theorem A is valid for $\nu = 24$.
\end{theorem}

\begin{theorem}\label{thm3.90}
Suppose that $0 \le \alpha_k, \beta_k \le 1$, $k = 1, \cdots, n-1$.
Then Theorem B is valid for $\nu = 24$.
\end{theorem}

\begin{theorem}\label{thm3.91}
Suppose that (\ref{eqs3.52}) holds.
Then Theorem C is valid for $\nu = 24$,
provided one of the following conditions is satisfied:

\begin{itemize}
\item[(i)] For $k = 1, \cdots, n-1$,
$0 \le \alpha_k, \beta_k \lesssim 1$.

\item[(ii)] $0 \le \gamma < 1$, $a_{k,k+1} < 0$, $0 \le \alpha_k \le 1$
and $0 < \beta_k \le 1$, $k = 1, \cdots, n-1$.

\item[(iii)] $\gamma = 1$, $a_{k,k+1} < 0$, $0 \le \alpha_k \le 1$
and $0 < \beta_k \le 1$, $k = 1, \cdots, n-1$. One of the following conditions holds:

\begin{itemize}
\item[($iii_1$)]
There exist $j\in\{2, \cdots, n\}$ and $k\in\{1, \cdots, j-1\}$ such that
$\alpha_{k}a_{n,k}a_{k,j}$ $ > 0$.

\item[($iii_2$)] $a_{n,1} < 0$ and $\alpha_{1} > 0$.

\item[($iii_3$)] There exists $j\in\{2, \cdots, n-1\}$ such that
\begin{eqnarray*}
(1 - \alpha_{j})a_{n,j} - \sum\limits_{{k = 1}\atop{k\not = j}}^{n-1}\alpha_{k}a_{n,k}a_{k,j} < 0.
\end{eqnarray*}

\item[($iii_4$)] $a_{1,2} < 0$, $\beta_1 > 0$ and
\begin{eqnarray*}
(1 - \alpha_{1})a_{n,1} - \sum\limits_{k = 2}^{n-1}\alpha_{k}a_{n,k}a_{k,1} < 0.
\end{eqnarray*}

\item[($iii_5$)] $a_{1,2} < 0$,  $a_{n,1} < 0$ and $\beta_1 > 0$.
\end{itemize}
\end{itemize}
\end{theorem}

\begin{theorem}\label{thm3.92}
Theorem D is valid for $\nu = 24$,
provided one of the following conditions is satisfied:

\begin{itemize}
\item[(i)] One of the conditions ($i$), ($ii$) and ($iii$) of Theorem \ref{thm3.91} holds.

\item[(ii)] For $j = 1,\cdots,n-1$, $0 \le \alpha_j, \beta_j \le 1$ and
$\alpha_{j}a_{n,j} + \sum_{k = 1,k\not=j}^{n-1}\alpha_{k}a_{n,k}a_{k,j} \ge 0$.
One of the following conditions holds:

\begin{itemize}
\item[($ii_1$)] There exists $i_0\in\{1,\cdots,n-1\}$ such that
$\beta_{i_0}a_{i_0,i_0+1}a_{i_0+1,i_0} > 0$.

\item[($ii_2$)] There exists $i_0\in\{1,\cdots,n-1\}$ such that
$\alpha_{i_0}a_{n,i_0}a_{i_0,n} > 0$.

\item[($ii_3$)] $\gamma > 0$ and there exist $i_0\in\{2,\cdots,n-1\}$ and $j_0\in\{1,\cdots,i_0-1\}$ such that
$\beta_{i_0}a_{i_0,i_0+1}a_{i_0+1,j_0} > 0$.

\item[($ii_4$)] $\gamma > 0$ and there exists $j_0\in\{1,\cdots,n-1\}$ such that
 \begin{eqnarray*}
\alpha_{j_0}a_{n,j_0} + \sum\limits_{{k = 1}\atop{k\not =j_0}}^{n-1}\alpha_{k}a_{n,k}a_{k,j_0}
 > 0.
\end{eqnarray*}
\end{itemize}\end{itemize}
\end{theorem}

Similarly, as a special case of $Q_{23}$, $Q$ is proposed in \cite{SSHL18} as
 \begin{eqnarray*}
Q_{25} & = & Q_6 + Q_{17} \\
& = & \left(\begin{array}{ccccc}
0                      & -\beta_1a_{1,2} & 0                 & \cdots & 0\\
-\alpha_{2}a_{2,1}     & 0               & -\beta_{2}a_{2,3} & \cdots & 0\\
\vdots                 & \vdots          & \ddots            & \ddots & \vdots\\
-\alpha_{n-1}a_{n-1,1} & 0               & 0                 & \ddots & -\beta_{n-1}a_{n-1,n} \\
-\alpha_{n}a_{n,1}     & 0               & 0                 & \cdots & 0\\
\end{array}\right)
 \end{eqnarray*}
with $\alpha_{k+1} \ge 0$, $\beta_k \ge 0$, $k = 1, \cdots, n-1$, and
\begin{eqnarray*}
\sum\limits_{k = 2}^{n}\alpha_{k}a_{k,1} \not = 0 \;\;
\hbox{and} \;\; \sum\limits_{k = 1}^{n-1}\beta_ka_{k,k+1} \not = 0,
 \end{eqnarray*}
 where for simplicity we set $r = 2$ for $Q_6$.

It is proposed in \cite{NA12} for the preconditioned SOR method, where
$\alpha_{k+1} = \beta_k = 1$, $k = 1, \cdots, n-1$.

In this case, $\delta_{i,j}^{(23)}(\gamma)$ and $\delta_{i,j}^{(23)}(1)$ reduce respectively to
\begin{eqnarray*}
\delta_{i,j}^{(25)}(\gamma) = \left\{\begin{array}{ll}
\beta_1a_{1,2}a_{2,1}, & i = j = 1;\\
 \alpha_{i}a_{i,1}a_{1,i} + \beta_ia_{i,i+1}a_{i+1,i}, & i = j = 2, \cdots, n;\\
 (\gamma-1)\alpha_{i}a_{i,1} + \gamma\beta_ia_{i,i+1}a_{i+1,1}, &
  i = 2,\cdots,n, j = 1;\\
\gamma\alpha_{i}a_{i,1}a_{1,j} + \gamma\beta_ia_{i,i+1}a_{i+1,j}, &
  i = 3,\cdots,n, j = 2, \cdots, i-1;\\
 (\gamma-1)\beta_ia_{i,i+1}, & i = 1, \cdots, n-1, j = i+1;\\
 \gamma\beta_ia_{i,i+1}a_{i+1,j}, & i = 1, \cdots, n-1,\\& j = i+2, \cdots, n
\end{array}\right.
\end{eqnarray*}
and
\begin{eqnarray*}
\delta_{i,j}^{(25)}(1) = \left\{\begin{array}{ll}
 \beta_ia_{i,i+1}a_{i+1,1}, &
  i = 1,\cdots,n-1, \\
  & j\in\{1\}\cup\{i+2, \cdots, n\};\\
\alpha_{i}a_{i,1}a_{1,j} + \beta_ia_{i,i+1}a_{i+1,j}, &
  i = 2,\cdots,n, j = 2, \cdots, i;\\
0, & otherwise.
\end{array}\right.
\end{eqnarray*}

From Theorems \ref{thm3.85}-\ref{thm3.88}, the following comparison results are directly.

\begin{theorem}
Suppose that $0 \le \alpha_{k+1}, \beta_k \le 1$, $k = 1, \cdots, n-1$, and
$\beta_1a_{1,2}a_{2,1} < 1$, $\alpha_n a_{n,1}a_{1,n} < 1$,
$\alpha_ka_{k,1}a_{1,k} + \beta_k a_{k,k+1}a_{k+1,k} < 1$, $k = 2, \cdots, n-1$.
Then Theorem A is valid for $\nu = 25$.
\end{theorem}

\begin{theorem}
Suppose that $0 \le \alpha_{k+1}, \beta_k \le 1$, $k = 1, \cdots, n-1$.
Then Theorem B is valid for $\nu = 25$.
\end{theorem}

\begin{theorem}\label{thm3.95}
Suppose that $\beta_1a_{1,2}a_{2,1} < 1$, $\alpha_n a_{n,1}a_{1,n} < 1$,
$\alpha_ka_{k,1}a_{1,k} + $\\ $\beta_k a_{k,k+1}a_{k+1,k} < 1$, $k = 2, \cdots, n-1$.
Then Theorem C is valid for $\nu = 25$,
provided one of the following conditions is satisfied:

\begin{itemize}
\item[(i)] For $k = 1, \cdots, n-1$,
$0 \le \alpha_{k+1}, \beta_k \lesssim 1$.

\item[(ii)] $0 \le \gamma < 1$, $0 \le \alpha_{k+1}, \beta_k \le 1$,
$k = 1, \cdots, n-1$. $\beta_1 > 0$,
$a_{1,2} < 0$ and for each $i\in\{2, \cdots, n-1\}$,
$\beta_ia_{i,i+1} < 0$ or $\alpha_{i}a_{i,1} < 0$.

\item[(iii)] $\gamma = 1$, $0 \le \alpha_{k+1}, \beta_k \le 1$, $k = 1, \cdots, n-1$.
For each $i\in\{2, \cdots, n-1\}$, $\beta_ia_{i,i+1} < 0$
or there exists $j(i)\in\{2, \cdots, i\}$
such that $\alpha_{i}a_{i,1}a_{1,j(i)} > 0$.
At the same time, one of the following conditions holds:

\begin{itemize}
\item[($iii_1$)] $a_{n,1} < 0$ and $\alpha_{n} > 0$.

\item[($iii_2$)] There exists $j\in\{2, \cdots, n-1\}$ such that
$a_{n,j} - \alpha_{n}a_{n,1}a_{1,j} < 0$.

\item[($iii_3$)]  $a_{1,2} < 0$,  $a_{n,1} < 0$ and $\beta_1 > 0$.
\end{itemize}
\end{itemize}
\end{theorem}

It can be proved that if $0 < a_{1,2}a_{2,1} < 1$, $0 < a_{n,1}a_{1,n} < 1$ and
$0 < a_{k,1}a_{1,k} + a_{k,k+1}a_{k+1,k} < 1$, $k = 2, \cdots, n-1$,
then $A$ is irreducible. Hence, for the Gauss-Seidel method the result when ($iii$) holds is better than
\cite[Theorems 2,3,4]{SSHL18}, where in Theorem 4 it should be that $\alpha_{k+1}>0$ and $\beta_k>0$,
$k=1,\cdots,n-1$.

All the corresponding results given in \cite{NA12} are problematic, because
\cite[Theorem 3.1]{NA12} is wrong. In fact, let
 \begin{eqnarray}\label{eqs3.66}
A = \left(\begin{array}{rrrrrr}
1    & 0    & 0    & 0    & 0    & -0.5\\
-0.5 & 1    & -0.5 & 0    & 0    & 0   \\
0    & -0.5 & 1    & 0    & 0    & 0   \\
0    & -0.5 & 0    & 1    & 0    & 0   \\
0    & 0    & 0    & 0    & 1    & -0.5\\
-0.5 & -0.5 & 0    & -0.5 & -0.5 & 1
\end{array}\right).
 \end{eqnarray}
Then it is easy to prove that $A$ is an irreducible L-matrix and
the assumption of \cite[Theorem 3.1]{NA12} is satisfied. But
the iteration matrix of the preconditioned SOR method
is reducible.

\begin{theorem}
Theorem D is valid for $\nu = 25$,
provided one of the conditions ($i$), ($ii$) and ($iii$) of Theorem \ref{thm3.95} is satisfied.
\end{theorem}

As a special case of $Q_{24}$ and $Q_{25}$, $Q$ is defined in \cite{WH07} as
 \begin{eqnarray*}
Q_{26} & = & Q_7 + Q_{17} \\
& = & \left(\begin{array}{ccccc}
0     & -\beta_1a_{1,2} & 0                 & \cdots & 0\\
0     & 0               & -\beta_{2}a_{2,3} & \cdots & 0\\
\vdots                 & \vdots          & \ddots            & \ddots & \vdots\\
0 & 0               & 0                 & \ddots & -\beta_{n-1}a_{n-1,n} \\
-\alpha a_{n,1}     & 0               & 0                 & \cdots & 0
\end{array}\right)
 \end{eqnarray*}
with $a_{n,1} < 0$, $\alpha > 0$, $\beta_k \ge 0$, $k = 1, \cdots, n-1$, and
$\sum_{k = 1}^{n-1}\beta_ka_{k,k+1} \not = 0$.

It is continued to study in \cite{HWXC16}.
It is proposed in \cite{YZ12,Yu12} for the preconditioned Gauss-Seidel method, where
$\alpha_1 = \beta_k = 1$, $k = 1, \cdots, n-1$ and $\alpha_k = 0$, $k = 2, \cdots, n-1$.

In this case, $\delta_{i,j}^{(23)}(\gamma)$ and $\delta_{i,j}^{(23)}(1)$ reduce respectively to
\begin{eqnarray*}
\delta_{i,j}^{(26)}(\gamma) = \left\{\begin{array}{ll}
 \beta_ia_{i,i+1}a_{i+1,i}, & i = j = 1, \cdots, n-1;\\
 \alpha a_{n,1}a_{1,n},  & i =  j = n;\\
 (\gamma-1)\beta_ia_{i,i+1}, & i = 1, \cdots, n-1, j = i+1;\\
 \gamma\beta_ia_{i,i+1}a_{i+1,j}, & i = 1, \cdots, n-1, j = 1, \cdots, n, j\not = i, i+1; \\
 \gamma\alpha a_{n,1}a_{1,j},  & i = n, j = 2, \cdots, n-1;\\
 (\gamma-1)\alpha a_{n,1} & i = n, j = 1
\end{array}\right.
\end{eqnarray*}
and
\begin{eqnarray*}
\delta_{i,j}^{(26)}(1) = \left\{\begin{array}{ll}
\beta_ia_{i,i+1}a_{i+1,j}, & i = 1, \cdots, n-1, j = 1, \cdots, n, j\not = i+1; \\
 \alpha a_{n,1}a_{1,j},  & i = n, j = 2, \cdots, n;\\
  0, & otherwise.
\end{array}\right.
\end{eqnarray*}

From Theorems \ref{thm3.89} and \ref{thm3.90} the following comparison results are immediately.

\begin{theorem}
Suppose that $\alpha \le 1$, $0 \le \beta_k \le 1$, $\alpha a_{n,1}a_{1,n} < 1$,
$\beta_ka_{k,k+1}a_{k+1,k}$ $ < 1$, $k = 1, \cdots, n-1$.
Then Theorem A is valid for $\nu = 26$.
\end{theorem}

\begin{theorem}
Suppose that $\alpha \le 1$, $0 \le \beta_k \le 1$, $k = 1, \cdots, n-1$.
Then Theorem B is valid for $\nu = 26$.
\end{theorem}

Completely similar to Lemma \ref{lem3-444} we can prove the following lemma.

\begin{lemma}\label{lem3-4444}
Let $A$ be a Z-matrix. Assume that
$n \ge 3$, $a_{1,n} < 0$, $a_{k+1,k} < 0$, $\alpha \le 1$, $0 \le \beta_k \le 1$,
$k = 1, \cdots, n-1$.
Then $A^{(26)}$ is an irreducible Z-matrix.
\end{lemma}

\begin{theorem}\label{thm3.99}
Suppose that $\alpha a_{n,1}a_{1,n} < 1$,
$\beta_ka_{k,k+1}a_{k+1,k} < 1$, $k = 1, \cdots,$ $ n-1$.
Then Theorem C is valid for $\nu = 26$,
provided one of the following conditions is satisfied:

\begin{itemize}
\item[(i)] $\alpha \lesssim 1$ and $0 \le \beta_k \lesssim 1$, $k = 1, \cdots, n-1$.

\item[(ii)] $\alpha \le 1$,
$0 < \beta_k \le 1$, $a_{k,k+1} < 0$, $k = 1, \cdots, n-1$.

\item[(iii)] $n \ge 3$, $\alpha \le 1$,
$0 \le \beta_k \le 1$, $a_{1,n} < 0$, $a_{k+1,k} < 0$, $k = 1, \cdots, n-1$.
\end{itemize}
\end{theorem}

\begin{proof}
By Theorem \ref{thm3.91}, we just need to prove ($ii$) and ($iii$).

When ($ii$) holds, since $a_{n,1} < 0$ and $\alpha > 0$, the condition ($iii_2$)
in Theorem \ref{thm3.91} is satisfied.

When ($iii$) holds, by Lemma \ref{lem3-4444}, $A^{(26)}$ is an irreducible Z-matrix.
From ($i$) we can prove ($iii$).
\end{proof}

For ($ii$) and ($iii$) the assumption that $A$ is irreducible is redundant.
Hence, \cite[Theorem 3.1]{Yu12} can be derived, directly.

All the corresponding results given in \cite{WH07} are problematic, because
\cite[Lemmas 4.1, 4.3]{WH07} are wrong. In fact, let $A$ be defined by (\ref{eqs3.66}).
Then it is easy to prove that $A$ is an irreducible L-matrix and
the assumptions of \cite[Lemmas 4.1, 4.3]{WH07} are satisfied. But
the iteration matrices of the preconditioned AOR methods
are reducible when we choose $\beta_2=1$.

In this case, ($ii$) of Theorem \ref{thm3.92} can be not satisfied. Hence,
form Theorems \ref{thm3.92}, the following theorem is derived, directly.

\begin{theorem}
Theorem D is valid for $\nu = 26$,
provided one of the conditions ($i$), ($ii$) and ($iii$) of Theorem \ref{thm3.99} is satisfied.
\end{theorem}

Corresponding to $Q_{26}$, in \cite{HWXC16} $Q$ is defined as
\begin{eqnarray*}
Q_{27} = \left(\begin{array}{ccccc}
0 & 0 & \cdots & 0 & -\beta a_{1,n} \\
-\alpha_1a_{2,1} & 0 & \cdots & 0 & 0\\
0 & -\alpha_2a_{3,2} &\ddots & 0 & 0 \\
\vdots & \vdots & \ddots & \ddots & \vdots\\
0 & 0 & \cdots & -\alpha_{n-1}a_{n,n-1} & 0
\end{array}\right)
 \end{eqnarray*}
with $a_{1,n} < 0$, $\beta > 0$, $\alpha_k \ge 0$, $k = 1, \cdots, n-1$, and
\begin{eqnarray*}
\sum\limits_{k = 1}^{n-1}\alpha_ka_{k+1,k} \not = 0.
 \end{eqnarray*}

In this case, $\delta_{i,j}^{(2)}(\gamma)$ and $\delta_{i,j}^{(2)}(1)$ reduce respectively to
\begin{eqnarray*}
\delta_{i,j}^{(27)}(\gamma) = \left\{\begin{array}{ll}
\beta a_{1,n}a_{n,1}, & i = j = 1; \\
\alpha_{i-1}a_{i,i-1}a_{i-1,i}, & i = j = 2,\cdots,n;\\
 (\gamma-1)\beta a_{1,n}, & i = 1, j = n; \\
   (\gamma-1)\alpha_{i-1}a_{i,i-1}, & i = 2,\cdots,n, j = i-1;\\
   0, & otherwise
\end{array}\right.
\end{eqnarray*}
and
\begin{eqnarray*}
\delta_{i,j}^{(27)}(\gamma) = \left\{\begin{array}{ll}
\beta a_{1,n}a_{n,1}, & i = j = 1; \\
 \alpha_{i-1}a_{i,i-1}a_{i-1,i}, & i = j = 2,\cdots,n; \\
 0, & otherwise.
\end{array}\right.
\end{eqnarray*}

By Corollaries \ref{coro3-9} and \ref{coro3-10}, the following comparison theorems are directly.

\begin{theorem}
Suppose that $\beta \le 1$, $0 \le \alpha_k \le 1$, $\beta a_{1,n}a_{n,1} < 1$,
$\alpha_k a_{k+1,k}a_{k,k+1}$ $ < 1$, $k = 1, \cdots, n-1$.
Then Theorem A is valid for $\nu = 27$.
\end{theorem}

\begin{theorem}
Suppose that $\beta \le 1$, $0 \le \alpha_k \le 1$, $k = 1, \cdots, n-1$.
Then Theorem B is valid for $\nu = 27$.
\end{theorem}

Completely similar to Lemma \ref{lem3-4444} we can prove the following lemma.

\begin{lemma}\label{lem3-4445}
Let $A$ be a Z-matrix. Assume that
$n \ge 3$, $a_{n,1} < 0$, $\beta \le 1$, $a_{k,k+1} < 0$, $0 \le \alpha_k \le 1$,
$k = 1, \cdots, n-1$.
Then $A^{(27)}$ is an irreducible Z-matrix.
\end{lemma}

\begin{theorem}\label{thm3.999}
Suppose that $\beta a_{1,n}a_{n,1} < 1$,
$\alpha_k a_{k+1,k}a_{k,k+1} < 1$, $k = 1, \cdots,$ $ n-1$.
Then Theorem C is valid for $\nu = 27$,
provided one of the following conditions is satisfied:

\begin{itemize}
\item[(i)] $0 \le \gamma < 1$, $\beta \lesssim 1$ and $0 \le \alpha_k \lesssim 1$, $k = 1, \cdots, n-1$.

\item[(ii)] $\gamma = 1$, $\beta \lesssim 1$ and $0 \le \alpha_k \lesssim 1$, $k = 1, \cdots, n-1$.
One of the following conditions holds:

\begin{itemize}
\item[($ii_1$)] $a_{n,1} < 0$.

\item[($ii_2$)] There exists $k\in\{1, \cdots, n-1\}$ such that
$\alpha_{k}a_{k+1,k}a_{k,k+1} > 0$.

\item[($ii_3$)] $a_{k,k+1} < 0$, $k = 1, \cdots, n-1$.
\end{itemize}

\item[(iii)] $0 \le \gamma < 1$, $\beta \le 1$, $0 \ge \alpha_{n-1} \le 1$,
$0 < \alpha_k \le 1$ and $a_{k+1,k} < 0$, $k = 1, \cdots, n-2$.

\item[(iv)] $\gamma = 1$, $\beta \le 1$, $0 \le \alpha_{n-1} \le 1$,
$0 < \alpha_k \le 1$ and $a_{k+1,k}a_{k,k+1} > 0$, $k = 1, \cdots, n-2$.
One of the following conditions holds:

\begin{itemize}
\item[($iv_1$)] $a_{n,1} < 0$.

\item[($iv_2$)] $\alpha_{n-1} > 0$ and $a_{n,n-1}a_{n-1,n} > 0$.
\end{itemize}

\item[(v)] $n \ge 3$, $a_{n,1} < 0$, $a_{k,k+1} < 0$, $\beta \le 1$, $0 \le \alpha_k \le 1$,
$k = 1, \cdots, n-1$.
\end{itemize}
\end{theorem}

\begin{proof}
($i$), ($ii_1$) and ($ii_2$) satisfy respectively ($i$), ($ii_1$) in Theorem \ref{thm3.7}.
Hence they satisfy the condition ($i$) of Corollaries \ref{coro3-11}.

By the definition of $Q_{27}$, there exists $k_0 \in \{1, \cdots, n-1\}$ such that
$\alpha_{k_0}a_{k_0+1,k_0}$ $ < 0$ so that $\alpha_{k_0}a_{k_0+1,k_0}a_{k_0,k_0+1} > 0$,
which implies that ($ii_2$) holds for $k=k_0$.

($iii$) can be derived by ($ii$) of Corollaries \ref{coro3-11}.

($iv$) satisfies the conditions ($iv_1$), ($iv^e$) and ($iv^a$) in Theorem \ref{thm3.7}, so that
it satisfies condition ($ii$) of Corollaries \ref{coro3-11}.

When ($v$) holds, by Lemma \ref{lem3-4445}, $A^{(27)}$ is an irreducible L-matrix.
From ($i$) and ($ii$) we can prove ($v$), where ($ii_1$) is satisfied.
\end{proof}

Obviously, for ($v$) the assumption that $A$ is irreducible is redundant.

The following result is easy to prove.

\begin{theorem}
Theorem D is valid for $\nu = 27$,
provided one of the conditions ($i$)-($v$) of Theorem \ref{thm3.99} is satisfied.
\end{theorem}

As a special case of $Q_{23}$, $Q$ is defined in \cite{LH05} for the preconditioned Gauss-Seidel method as
 \begin{eqnarray*}
Q_{28} & = & Q_8 + Q_{17}\\
& = & \left(\begin{array}{ccccc}
0 & -\beta_1a_{1,2} & \cdots & 0 & 0\\
-\alpha_1a_{2,1} & 0 & \ddots & 0 & 0\\
\vdots & \ddots & \ddots & \ddots& \vdots\\
0 & 0 & \ddots & 0 & -\beta_{n-1}a_{n-1,n} \\
0 & 0 & \cdots & -\alpha_{n-1}a_{n,n-1} & 0
\end{array}\right)
 \end{eqnarray*}
with $\beta_k, \alpha_k \ge 0$, $k = 1, \cdots, n-1$, and
\begin{eqnarray*}
\sum\limits_{k = 1}^{n-1}\beta_ka_{k,k+1} \not = 0 \;\;
\hbox{and} \;\; \sum\limits_{k = 1}^{n-1}\alpha_ka_{k+1,k} \not = 0.
 \end{eqnarray*}

It is continued to be studied in \cite{LJ10} for the preconditioned Gauss-Seidel and Jacobi methods.
It is given in \cite{WDT12,WWT13} for the preconditioned GAOR method. In \cite{WS16}, it is used to
preconditioned parallel multisplitting USAOR method.

In this case, $\delta_{i,j}^{(23)}(\gamma)$ and $\delta_{i,j}^{(23)}(1)$ reduce respectively to
\begin{eqnarray*}
\delta_{i,j}^{(28)}(\gamma) = \left\{\begin{array}{ll}
 \alpha_{i-1}a_{i,i-1}a_{i-1,i} + \beta_ia_{i,i+1}a_{i+1,i}, & i = j = 1, \cdots, n;\\
(\gamma-1)\alpha_{i-1}a_{i,i-1} + \gamma\beta_ia_{i,i+1}a_{i+1,j},
 & i = 2,\cdots, n, j = i-1; \\
 (\gamma-1)\beta_ia_{i,i+1}, & i = 1, \cdots, n-1, j = i+1;\\
 \gamma\beta_ia_{i,i+1}a_{i+1,j}, & i = 1, \cdots, n-1,\\& j = 1, \cdots, n, \\
 & j\not = i-1, i, i+1; \\
 0, & i = n, j = 1, \cdots, n-2
 \end{array}\right.
\end{eqnarray*}
and
\begin{eqnarray*}
\delta_{i,j}^{(28)}(1) = \left\{\begin{array}{ll}
\beta_1a_{1,2}a_{2,1} < 1, & i = j = 1; \\
 \alpha_{i-1}a_{i,i-1}a_{i-1,i} + \beta_ia_{i,i+1}a_{i+1,i}, & i = j = 2, \cdots, n-1;\\
\alpha_{n-1}a_{n,n-1}a_{n-1,n} < 1, & i = j = n;\\
\beta_ia_{i,i+1}a_{i+1,j}, & i = 1, \cdots, n-1,\\& j = 1, \cdots, n, j\not = i, i+1; \\
 0, & otherwise.
 \end{array}\right.
\end{eqnarray*}

From Theorems \ref{thm3.85} and \ref{thm3.86}, the following two comparison results are immediately.

\begin{theorem}
Suppose that $0 \le \alpha_k, \beta_k \le 1$, $k = 1, \cdots, n-1$, and
$\beta_1a_{1,2}a_{2,1} < 1$, $\alpha_{n-1}a_{n,n-1}a_{n-1,n} < 1$,
$\alpha_{k-1}a_{k,k-1}a_{k-1,k} + \beta_ka_{k,k+1}a_{k+1,k} < 1$,
$k = 2, \cdots, n-1$.
Then Theorem A is valid for $\nu = 28$.
\end{theorem}

\begin{theorem}
Suppose that $0 \le \alpha_k, \beta_k \le 1$, $k = 1, \cdots, n-1$.
Then Theorem B is valid for $\nu = 28$.
\end{theorem}

\begin{theorem}\label{thm3.103}
Suppose that $\beta_1a_{1,2}a_{2,1} < 1$, $\alpha_{n-1}a_{n,n-1}a_{n-1,n} < 1$,\\
$\alpha_{k-1}a_{k,k-1}a_{k-1,k} + \beta_ka_{k,k+1}a_{k+1,k} < 1$,
$k = 2, \cdots, n-1$.
Then Theorem C is valid for $\nu = 28$,
provided one of the following conditions is satisfied:

\begin{itemize}
\item[(i)] For $k = 1, \cdots, n-1$,
$0 \le \alpha_{k}, \beta_k \lesssim 1$.

\item[(ii)] $0 \le \gamma < 1$, $0 \le \alpha_k, \beta_k \le 1$, $k = 1, \cdots, n-1$.
$a_{1,2} < 0$, $\beta_1 > 0$ and for each $i\in\{2, \cdots, n-1\}$,
$\beta_ia_{i,i+1} < 0$ or $\alpha_{i-1}a_{i,i-1} < 0$.

\item[(iii)] $\gamma = 1$, $0 \le \alpha_k, \beta_k \le 1$, $k = 1, \cdots, n-1$.
For each $i\in\{2, \cdots, n-1\}$, $\beta_ia_{i,i+1} < 0$
or $\alpha_{i-1}a_{i,i-1}a_{i-1,i} > 0$.
At the same time, one of the following conditions holds:

\begin{itemize}
\item[($iii_1$)]
$\alpha_{n-1}a_{n,n-1}a_{n-1,n} > 0$.

\item[($iii_2$)] $(1-\alpha_{n-1})a_{n,n-1} < 0$.

\item[($iii_3$)] There exists $j\in\{2, \cdots, n-2\}$ such that
$a_{n,j} - \alpha_{n-1}a_{n,n-1}a_{n-1,j} < 0$.

\item[($iii_4$)] $a_{1,2} < 0$, $\beta_1 > 0$ and
$a_{n,1} - \alpha_{n-1}a_{n,n-1}a_{n-1,1} < 0$.

\item[($iii_5$)] $a_{1,2} < 0$,  $a_{n,1} < 0$ and $\beta_1 > 0$.
\end{itemize}
\end{itemize}
\end{theorem}

\begin{proof}
We just need to prove ($iii$).

When $\alpha_{i-1}a_{i,i-1}a_{i-1,i} > 0$, the inequality $\alpha_{i,k(i)}a_{i,k(i)}a_{k(i),j(i)} > 0$
holds for $j(i)=i$ and $k(i)= j(i)-1$.

The conditions ($iii_1$), ($iii_4$) and ($iii_5$) can be derived from corresponding one in Theorem \ref{thm3.87}.
While, the conditions ($iii_2$) and ($iii_3$) can be derived by ($iii_3$) in Theorem \ref{thm3.87}, directly.
\end{proof}

The result for the Gauss-Seidel method is better than \cite[Theorems 3.2, 3.3, 3.4, 3.5]{LH05},
where in Theorem 3.5 it should be that $\alpha_k>0$ and $\beta_k>0$, $k = 1, \cdots, n-1$.

\begin{theorem}
Theorem D is valid for $\nu = 28$,
provided one of the conditions ($i$), ($ii$) and ($iii$) of Theorem \ref{thm3.103} is satisfied.
\end{theorem}

Another combination is given as
 \begin{eqnarray*}
Q_{29} = Q_5 + Q_{15}
 = \left(\begin{array}{cccc}
0                    & \cdots & 0             & -\beta_1a_{1,n}\\
\vdots               & \ddots & \vdots        & \vdots\\
0                    & \cdots & 0             & -\beta_{n-1}a_{n-1,n} \\
-\alpha_{1}a_{n,1}   & \cdots & -\alpha_{n-1}a_{n,n-1} & 0
\end{array}\right)
 \end{eqnarray*}
with $\alpha_k \ge 0$, $\beta_k \ge 0$, $k = 1, \cdots, n-1$, and
\begin{eqnarray*}
\sum\limits_{k = 1}^{n-1}\alpha_ka_{n,k} \not = 0 \;\;
\hbox{and} \;\; \sum\limits_{k = 1}^{n-1}\beta_ka_{k,n} \not = 0,
 \end{eqnarray*}
 where for simplicity we set $r = n$ for $Q_5$ and $Q_{15}$.

In this case, $\delta_{i,j}^{(2)}(\gamma)$ and $\delta_{i,j}^{(2)}(1)$ reduce respectively to
\begin{eqnarray*}
\delta_{i,j}^{(29)}(\gamma) = \left\{\begin{array}{ll}
\beta_ia_{i,n}a_{n,i}, & i = j = 1, \cdots, n-1;\\
 \sum\limits_{k = 1}^{n-1} \alpha_ka_{n,k}a_{k,n}, & i = j = n;\\
 (\gamma-1)\beta_ia_{i,n}, & i = 1, \cdots, n-1, j = n;\\
 \gamma\beta_ia_{i,n}a_{n,j}, & i = 2, \cdots, n-1, \\
 & j = 1, \cdots, i-1;\\
 (\gamma-1)\alpha_ja_{n,j} + \gamma\sum\limits_{k = 1}^{j-1} \alpha_ka_{n,k}a_{k,j},
 & i = n, j = 1,\cdots, n-1;\\
  0, & otherwise
\end{array}\right.
\end{eqnarray*}
and
\begin{eqnarray*}
\delta_{i,j}^{(29)}(1) = \left\{\begin{array}{ll}
 \beta_ia_{i,n}a_{n,j}, & i = 1, \cdots, n-1, j = 1, \cdots, i;\\
 \sum\limits_{k = 1}^{j-1} \alpha_ka_{n,k}a_{k,j}, &
 i = n, j = 2,\cdots,n;\\
 0, & otherwise.
\end{array}\right.
\end{eqnarray*}

Using Corollaries \ref{coro3-9} and \ref{coro3-10}, we can prove the following theorems, directly.

\begin{theorem}
Suppose that $0 \le \alpha_{k}, \beta_{k} \le 1$, $k = 1,\cdots,n-1$,
and
\begin{eqnarray}\label{eqs3.57}
 \sum\limits_{k = 1}^{n-1} \alpha_ka_{n,k}a_{k,n} < 1, \;
 \beta_ia_{i,n}a_{n,i} < 1, i = 1, \cdots, n-1.
\end{eqnarray}
Then Theorem A is valid for $\nu = 29$.
\end{theorem}

\begin{theorem}
Suppose that $0 \le \alpha_{k}, \beta_{k} \le 1$, $k = 1,\cdots,n-1$.
Then Theorem B is valid for $\nu = 29$.
\end{theorem}

\begin{theorem}\label{thm3.107}
Suppose that (\ref{eqs3.57}) holds.
Then Theorem C is valid for $\nu = 29$,
provided one of the following conditions is satisfied:

\begin{itemize}
\item[(i)] $0 \le \gamma < 1$. For $k = 1, \cdots, n-1$,
$0 \le \alpha_{k}, \beta_k \lesssim 1$.

\item[(ii)] $\gamma = 1$. For $k = 1, \cdots, n-1$,
$0 \le \alpha_{k}, \beta_k \lesssim 1$.
And one of the following conditions holds:

\begin{itemize}
\item[($ii_1$)]
There exist $i\in\{1, \cdots, n-1\}$ and $j\in\{1, \cdots, i\}$
 such that $\beta_ia_{i,n}a_{n,j} > 0$.

\item[($ii_2$)]
There exist $j\in\{2, \cdots, n\}$ and
 $k\in\{1, \cdots, j-1\}$ such that $\alpha_ka_{n,k}a_{k,j}$ $ > 0$.

\item[($ii_3$)]
$a_{n-1,n} < 0$ and $\beta_{n-1} > 0$.

\item[($ii_4$)]
$a_{n,1} < 0$.

\item[($ii_5$)]
$a_{k,n} < 0$, $k = 1, \cdots, n-1$.

\item[($ii_6$)]
$a_{k,k+1} < 0$, $k = 1, \cdots, n-1$.
\end{itemize}

\item[(iii)] $a_{k,n} < 0$, $0 \le \alpha_{k} \le 1$, $0 < \beta_{k} \le 1$, $k = 1,\cdots,n-1$.
And one of the following conditions holds:

\begin{itemize}
\item[($iii_1$)] $0 \le \gamma < 1$

\item[($iii_2$)] $\gamma = 1$ and for each $i\in\{1, \cdots, n-1\}$ there exists
$j(i)\in\{1, \cdots, i\}$ such that $a_{n,j(i)} < 0$.
\end{itemize}
\end{itemize}
\end{theorem}

\begin{proof}
Since $A$ is irreducible, then there exists $j_0\in\{1, \cdots, n-1\}$ such that $a_{n,j_0} < 0$.

By ($i$) of Corollary \ref{coro3-11}, ($i$) is obvious.

If ($ii_1$) holds, then $\delta_{i,j}^{(29)}(1) > 0$ for
$i = 1, \cdots, n-1$ and $j = 1, \cdots, i$. While if
($ii_2$) holds, then $\delta_{n,j}^{(29)}(1) \ge \alpha_ka_{n,k}a_{k,j} > 0$
for $j\in\{2, \cdots, n\}$. This shows that ($ii_1$) in Theorem \ref{thm3.7}
holds, so that the required result follows by ($i$) of Corollary \ref{coro3-11}, directly.

If ($ii_3$) holds, then $\beta_{n-1}a_{n-1,n}a_{n,j_0} > 0$, which implies that
($ii_1$) holds for $i = n-1$ and $j = j_0$.

If ($ii_4$) holds, by the definition of $Q_{29}$, there exists $i_0\in\{1, \cdots, n-1\}$ such that
$\beta_{i_0}a_{i_0,n} < 0$, so that $\beta_{i_0}a_{i_0,n}a_{n,1} > 0$, which implies that
($ii_1$) holds for $i = i_0$ and $j = 1$.

If ($ii_5$) holds, by the definition of $Q_{29}$, there exists $k_0\in\{1, \cdots, n-1\}$ such that
$\alpha_{k_0}a_{n,k_0} < 0$, so that $\alpha_{k_0}a_{n,k_0}a_{k_0,n} > 0$, which implies that
($ii_2$) holds for $k = k_0$ and $j = n$. While when ($ii_6$) holds, it is easy to see that
($ii_2$) holds for $k = k_0$ and $j = k_0+1$.

When ($iii$) holds, for each $i\in\{1,\cdots,n-1\}$, if $\gamma < 1$ then
$\beta_ia_{i,n} < 0$, i.e., ($iii$) of Theorem \ref{thm3.7} is satisfied.

While, if $\gamma = 1$ then
$\delta_{i,j(i)}^{(29)}(1) = \beta_ia_{i,n}a_{n,j(i)} > 0$, i.e.,
($iv_1$) in Theorem \ref{thm3.7} holds.
If $\alpha_{j_0} = 0$ then $(1-\alpha_{j_0})a_{n,j_0} = a_{n,j_0} < 0$, which implies that
(\ref{eqs3.15}) or (\ref{eqs3.16}) holds, so that ($iv^c$) or ($iv^d$) in Theorem \ref{thm3.7}
is satisfied. If $\alpha_{j_0} > 0$ then $\alpha_{j_0}a_{n,j_0}a_{j_0,n} > 0$,
which implies that ($iv^a$) in Theorem \ref{thm3.7} is satisfied for $k = j_0$ and $j = n$.
This has proved that the condition ($iv$) of Theorem \ref{thm3.7} is satisfied.

By Corollary \ref{coro3-11}, ($iii$) is proved.
\end{proof}

Similarly, by Corollary  \ref{coro3-12}, we can prove the following result.

\begin{theorem}
Theorem D is valid for $\nu = 29$,
provided one of the following conditions is satisfied:

\begin{itemize}
\item[(i)] One of the conditions ($i$), ($ii$) and ($iii$) of Theorem \ref{thm3.107} holds.

\item[(ii)] For $j = 1,\cdots,n-1$,
$0 \le \alpha_{j}, \beta_{j} \le 1$ and
\begin{eqnarray*}
\alpha_{j}a_{n,j} + \sum\limits_{{k = 1}\atop{k\not = j}}^{n-1} \alpha_{k}a_{n,k}a_{k,j} > 0.
 \end{eqnarray*}
One of the following conditions holds:

\begin{itemize}
\item[($ii_1$)] There exists $i_0\in\{1,\cdots,n-1\}$ such that
 $\beta_{i_0}a_{i_0,n}a_{n,i_0} > 0$.

\item[($ii_2$)]
There exists $i_0\in\{1,\cdots,n-1\}$ such that
$\alpha_{i_0}a_{n,i_0}a_{i_0,n} > 0$.

\item[($ii_3$)] $\gamma > 0$.
There exist $i_0\in\{2,\cdots,n-1\}$, $j_0\in\{1,\cdots,i_0-1\}$ such that
$\beta_{i_0}a_{i_0,n}a_{n,j_0} > 0$.

\item[($ii_4$)] $\gamma > 0$.
There exists $j_0\in\{1,\cdots,n-1\}$ such that
$\alpha_{j_0}a_{n,j_0} + $\\ $\sum_{k = 1,k\not = j_0}^{n-1} \alpha_{k}a_{n,k}a_{k,j_0} > 0$.
\end{itemize}\end{itemize}
\end{theorem}

Similar to $Q_{29}$, we give a new combination preconditioner as
 \begin{eqnarray*}
Q_{30} = Q_6 + Q_{14}
 = \left(\begin{array}{cccc}
0                  & -\beta_2a_{1,2}  & \cdots & -\beta_{n}a_{1,n}\\
-\alpha_{2}a_{2,1} & 0                & \cdots & 0\\
\vdots             & \vdots           & \ddots & \vdots\\
-\alpha_{n}a_{n,1} & 0                & \cdots & 0\\
\end{array}\right)
 \end{eqnarray*}
with $\alpha_{k} \ge 0$, $\beta_k \ge 0$, $k = 2, \cdots, n$, and
\begin{eqnarray*}
\sum\limits_{k = 2}^{n}\alpha_{k}a_{k,1} \not = 0 \;\;
\hbox{and} \;\; \sum\limits_{k = 2}^{n}\beta_ka_{1,k} \not = 0,
 \end{eqnarray*}
where for simplicity we set $r = 2$ for $Q_6$ and $Q_{14}$.

It is proposed in \cite{WDT12} for the preconditioned GAOR method
for weighted linear least squares problems.

In this case, $\delta_{i,j}^{(2)}(\gamma)$ and $\delta_{i,j}^{(2)}(1)$ reduce respectively to
\begin{eqnarray*}
\delta_{i,j}^{(30)}(\gamma) = \left\{\begin{array}{ll}
 \sum\limits_{k = 2}^n \beta_ka_{1,k}a_{k,1}, & i = j = 1;\\
 \alpha_ia_{i,1}a_{1,i}, & i = j = 2, \cdots, n;\\
 (\gamma-1)\beta_ja_{1,j} + \gamma\sum\limits_{k = 2}^{j-1} \beta_ka_{1,k}a_{k,j}, &
 i = 1, j = 2,\cdots,n;\\
 (\gamma-1)\alpha_ia_{i,1}, & i = 2, \cdots, n, j = 1;\\
\gamma\alpha_ia_{i,1}a_{1,j}, &
  i = 3, \cdots, n, j = 2,\cdots, i-1; \\
  0, & otherwise
\end{array}\right.
\end{eqnarray*}
and
\begin{eqnarray*}
\delta_{i,j}^{(30)}(1) = \left\{\begin{array}{ll}
 \sum\limits_{k = 2}^n\beta_k a_{1,k}a_{k,1}, & i = j = 1; \\
 \sum\limits_{k = 2}^{j-1} \beta_ka_{1,k}a_{k,j}, & i = 1, j = 3,\cdots,n;\\
\alpha_ia_{i,1}a_{1,j}, & i = 2, \cdots, n, j = 2, \cdots, i;\\
  0, & otherwise.
\end{array}\right.
\end{eqnarray*}

Using Corollaries \ref{coro3-9} and \ref{coro3-10}, we can prove the following theorems.

\begin{theorem}\label{thm3.109}
Suppose that $0 \le \alpha_{i}, \beta_{i} \le 1$, $\alpha_ia_{i,1}a_{1,i} < 1$,
$i = 2,\cdots,n$, $\sum_{k = 2}^n\beta_k a_{1,k}a_{k,1} < 1$.
Then Theorem A is valid for $\nu = 30$.
\end{theorem}

\begin{theorem}
Suppose that $0 \le \alpha_{k}, \beta_{k} \le 1$, $k = 2,\cdots,n$.
Then Theorem B is valid for $\nu = 30$.
\end{theorem}

\begin{theorem}\label{thm3.111}
Suppose that $\sum_{k = 2}^n\beta_k a_{1,k}a_{k,1} < 1$,
$\alpha_ia_{i,1}a_{1,i} < 1$, $i = 2, \cdots, n$.
Then Theorem C is valid for $\nu = 30$,
provided one of the following conditions is satisfied:

\begin{itemize}
\item[(i)] $0 \le \gamma < 1$. For $k = 2,\cdots,n$, $0 \le \alpha_{k}, \beta_{k} \lesssim 1$.

\item[(ii)] $\gamma = 1$. For $k = 2,\cdots,n$, $0 \le \alpha_{k}, \beta_{k} \lesssim 1$.
One of the following conditions holds:

\begin{itemize}
\item[($ii_1$)]
There exists $k\in\{2, \cdots, n\}$ such that $\beta_ka_{1,k}a_{k,1} > 0$.

\item[($ii_2$)]
There exist $j\in\{3, \cdots, n\}$ and $k\in\{2, \cdots, j-1\}$ such that
$\beta_ka_{1,k}a_{k,j}$ $ > 0$.

\item[($ii_3$)]
There exist $i\in\{2, \cdots, n\}$ and $j\in\{2, \cdots, i\}$ such that
$\alpha_ia_{i,1}a_{1,j} > 0$.

\item[($ii_4$)]
$a_{1,2} < 0$.

\item[($ii_5$)]
$a_{n,1} < 0$ and $\alpha_{n} > 0$.

\item[($ii_6$)]
$a_{k,1} < 0$, $k = 2, \cdots, n$.

\item[($ii_7$)]
$a_{n,1} < 0$ and $a_{k,n} < 0$, $k = 2, \cdots, n-1$.

\item[($ii_8$)]
$a_{n,1} < 0$ and $a_{k,k+1} < 0$, $k = 2, \cdots, n-1$.
\end{itemize}

\item[(iii)] $0 \le \gamma < 1$. For $k = 2,\cdots,n-1$, $a_{k,1} < 0$,
$0 < \alpha_{k} \le 1$, $0 \le \alpha_{n} \le 1$.
$k = 2,\cdots,n$, $0 \le \beta_{k} \le 1$.

\item[($iv$)]$\gamma = 1$.
$0 < \alpha_{k} \le 1$, $0 \le \beta_{k} \le 1$, $k = 2,\cdots,n$.
For each $i\in\{2, \cdots, n-1\}$, $a_{i,1} < 0$ and there exists
$j(i)\in\{2, \cdots, i\}$ such that $a_{1,j(i)} < 0$.
\end{itemize}
\end{theorem}

\begin{proof}
From ($i$) and ($ii_2$) in Theorem \ref{thm3.7}, we can derive ($i$) and ($ii_5$) directly.

If one of ($ii_1$), ($ii_2$) and ($ii_3$) holds, then there exist $i,j\in\{1, \cdots, n\}$
such that $\delta_{i,j}^{(30)}(1) > 0$, which implies that ($ii_1$) in Theorem \ref{thm3.7}
holds.

If ($ii_4$) holds, then from the definition of $Q_{30}$, there exists $i_0\in\{2, \cdots, n\}$
such that
$\alpha_{i_0}a_{i_0,1} < 0$ so that $\alpha_{i_0}a_{i_0,1}a_{1,2} > 0$. This has shown that
the condition ($ii_3$) is satisfied for $i = i_0$ and $j = 2$.

From
\begin{eqnarray*}
&& \sum\limits_{k = 2}^n\beta_k a_{1,k}a_{k,1}
 \ge \max\limits_{2 \le k \le n}\{a_{k,1}\}\sum\limits_{k = 2}^n\beta_k a_{1,k} > 0,\\
&& \sum\limits_{k = 2}^{n-1}\beta_k a_{1,k}a_{k,n} + \beta_n a_{1,n}a_{n,1} \\
&& \ge \max\{a_{n,1}; \; a_{k,n}:k = 2, \cdots, n-1\}\sum\limits_{k = 2}^{n}\beta_k a_{1,k} > 0
\end{eqnarray*}
and
\begin{eqnarray*}
&& \sum\limits_{j = 3}^{n}\beta_{j-1}a_{1,j-1}a_{j-1,j} +
\beta_n a_{1,n}a_{n,1} \\
 && \ge \max\{a_{n,1}; \; a_{k,k+1}:k = 2, \cdots, n-1\}\sum\limits_{j = 2}^{n}\beta_{j}a_{1,j} \\
 && > 0,
\end{eqnarray*}
it is easy to see that if one of ($ii_6$), ($ii_7$) and ($ii_8$) holds, then
($ii_1$) or ($ii_2$) is satisfied.

For ($iii$),
by the definition of $Q_{30}$ again, there exists $j(1)\in\{2, \cdots, n\}$ such that
$\beta_{j(1)}a_{1,j(1)} < 0$. For $i\in\{2, \cdots, n-1\}$, $\alpha_ia_{i,1} < 0$.
The condition ($iii$) follows by ($iii$) of Theorem \ref{thm3.7}, immediately.

If ($iv$) holds, then $\delta_{i,j(i)}^{(30)}(1) > 0$, which implies that
($iv_1$) in Theorem \ref{thm3.7} holds.
From the irreducibility of $A$, there exists $j_0\in\{1, \cdots, n-1\}$ such that
$a_{n,j_0} < 0$. If $j_0 = 1$ then ($iv^b$) in Theorem \ref{thm3.7} holds.
If $j_0 \ge 2$ then $a_{n,j_0} - \alpha_{n}a_{n,1}a_{1,j_0} \le a_{n,j_0} < 0$,
which shows that (\ref{eqs3.15}) holds for $j = j_0$, i.e., the condition
($iv^c$) in Theorem \ref{thm3.7} holds. We have proved ($iv$).
\end{proof}

\begin{theorem}\label{thm3.112}
Theorem D is valid for $\nu = 30$,
provided one of the conditions ($i$)-($iv$) of Theorem \ref{thm3.111} is satisfied.
\end{theorem}

Unlike $Q_{29}$ and $Q_{30}$, we give $Q$ as
 \begin{eqnarray*}
Q_{31} &=& Q_5 + Q_{14} \\
&=& \left(\begin{array}{ccccc}
0                  & -\beta_2a_{1,2}    & \cdots & -\beta_{n-1}a_{1,n-1}  & -\beta_{n}a_{1,n}\\
0                  & 0                  & \cdots & 0                     & 0\\
\vdots             & \vdots             & \ddots & \vdots                & \vdots\\
0                  & 0                  & \cdots & 0                     & 0\\
-\alpha_{1}a_{n,1} & -\alpha_{2}a_{n,2} & \cdots & -\alpha_{n-1}a_{n,n-1} &  0\\
\end{array}\right)
 \end{eqnarray*}
with $\alpha_{k} \ge 0$, $\beta_{k=1} \ge 0$, $k = 1, \cdots, n-1$, and
\begin{eqnarray*}
\sum\limits_{k = 1}^{n-1}\alpha_{k}a_{n,k} \not = 0 \;\;
\hbox{and} \;\; \sum\limits_{k = 2}^{n}\beta_ka_{1,k} \not = 0.
 \end{eqnarray*}

It is proposed in \cite{Wa062} with $\alpha_1 = 0$.

In this case, $\delta_{i,j}^{(2)}(\gamma)$ and $\delta_{i,j}^{(2)}(1)$ reduce respectively to
\begin{eqnarray*}
\delta_{i,j}^{(31)}(\gamma) = \left\{\begin{array}{ll}
\sum\limits_{k = 2}^n \beta_ka_{1,k}a_{k,1}, & i = j = 1;\\
\sum\limits_{k = 1}^{n-1} \alpha_ka_{n,k}a_{k,n}, & i = j = n;\\
 (\gamma-1)\beta_ja_{1,j} + \gamma\sum\limits_{k = 2}^{j-1} \beta_ka_{1,k}a_{k,j}, &
 i = 1, j = 2,\cdots,n;\\
 (\gamma-1)\alpha_ja_{n,j} + \gamma\sum\limits_{k = 1}^{j-1} \alpha_ka_{n,k}a_{k,j}, & i = n, j = 1, \cdots, n - 1;\\
  0, & otherwise
\end{array}\right.
\end{eqnarray*}
and
\begin{eqnarray*}
\delta_{i,j}^{(31)}(1) = \left\{\begin{array}{ll}
 \sum\limits_{k = 2}^n\beta_k a_{1,k}a_{k,1}, & i = j = 1; \\
 \sum\limits_{k = 2}^{j-1} \beta_ka_{1,k}a_{k,j}, & i = 1, j = 3,\cdots,n;\\
\sum\limits_{k = 1}^{j-1} \alpha_ka_{n,k}a_{k,j}, & i = n, j = 2, \cdots, n;\\
  0, & otherwise.
\end{array}\right.
\end{eqnarray*}

Using Corollaries \ref{coro3-9}-\ref{coro3-12}, similar to Theorems \ref{thm3.109}-\ref{thm3.112},
we can prove the following results.

\begin{theorem}\label{thm3.113}
Suppose that $0 \le \alpha_{k}, \beta_{k+1} \le 1$, $k = 1,\cdots,n-1$, and
\begin{eqnarray} \label{eqs3.59}
\sum\limits_{k = 2}^n\beta_k a_{1,k}a_{k,1} < 1, \;
\sum\limits_{k = 1}^{n-1} \alpha_ka_{n,k}a_{k,n} < 1.
\end{eqnarray}
Then Theorem A is valid for $\nu = 31$.
\end{theorem}

\begin{theorem}\label{thm3.114}
Suppose that $0 \le \alpha_{k}, \beta_{k+1} \le 1$, $k = 1,\cdots,n-1$.
Then Theorem B is valid for $\nu = 31$.
\end{theorem}

The results given in Theorems \ref{thm3.113} and \ref{thm3.114}
are better than the corresponding one given in \cite[Theorem 2.2]{Wa062}.

\begin{theorem}\label{thm3.115}
Suppose that (\ref{eqs3.59}) holds and $0 \le \alpha_{k}, \beta_{k+1} \lesssim 1$, $k = 1,\cdots,n-1$.
Then Theorem C is valid for $\nu = 31$,
provided one of the following conditions is satisfied:

\begin{itemize}
\item[(i)] $0 \le \gamma < 1$.

\item[(ii)] $\gamma = 1$ and one of the following conditions holds:

\begin{itemize}
\item[($ii_1$)]
There exists $k\in\{2, \cdots, n\}$ such that $\beta_ka_{1,k}a_{k,1} > 0$.

\item[($ii_2$)]
There exist $j\in\{3, \cdots, n\}$ and $k\in\{2, \cdots, j-1\}$ such that
$\beta_ka_{1,k}a_{k,j}$ $ > 0$.

\item[($ii_3$)]
There exist $j\in\{2, \cdots, n\}$ and $k\in\{1, \cdots, j-1\}$ such that
$\alpha_ka_{n,k}a_{k,j}$ $ > 0$.

\item[($ii_4$)]
$a_{n,1} < 0$ and $\alpha_{n} > 0$.

\item[($ii_5$)]
$a_{k,1} < 0$, $k = 2, \cdots, n$.

\item[($ii_6$)]
$a_{k,n} < 0$, $k = 1, \cdots, n-1$.

\item[($ii_7$)]
$a_{n,1} < 0$ and $a_{k,n} < 0$, $k = 2, \cdots, n-1$.

\item[($ii_8$)]
$a_{k,k+1} < 0$, $k = 1, \cdots, n-1$.

\item[($ii_9$)]
$a_{n,1} < 0$ and $a_{k,k+1} < 0$, $k = 2, \cdots, n-1$.
\end{itemize}
\end{itemize}
\end{theorem}

\begin{theorem}
Suppose that $0 \le \alpha_{k}, \beta_{k+1} \lesssim 1$, $k = 1,\cdots,n-1$.
Then Theorem D is valid for $\nu = 31$,
provided one of ($i$) and ($ii$) of Theorem \ref{thm3.115} is satisfied.
\end{theorem}

A combination is proposed in \cite{Wa06} as $Q_{32} = Q_5 + Q_{14}$, i.e.,
\begin{eqnarray*}
Q_{32} = \left(\begin{array}{ccccccc}
0      & \cdots  & 0 &  0 &  0 &  \cdots  & 0\\
\vdots & \ddots  & \vdots & \vdots & \vdots & \vdots & \vdots\\
0      & \cdots  & 0 & 0 &  0 & \cdots & 0\\
-\alpha_1a_{r,1} & \cdots & -\alpha_{r-1}a_{r,r-1} & 0 &  -\alpha_{r+1}a_{r,r+1} & \cdots &  -\alpha_{n}a_{r,n} \\
0      & \cdots  &      0  & 0 &  0 & \cdots & 0\\
\vdots &  \vdots &\vdots & \vdots & \vdots & \ddots & \vdots\\
0      & \cdots  &      0  & 0 &  0 & \cdots & 0
\end{array}\right)
 \end{eqnarray*}
with $2 \le r \le n-1$, $\alpha_k \ge 0$, $k = 1, \cdots, n$, $k\not = r$ and
\begin{eqnarray*}
\sum\limits_{k = 1}^{r-1}\alpha_ka_{r,k} \not = 0 \;\;
\hbox{and} \;\; \sum\limits_{k = r+1}^{n}\alpha_ka_{r,k} \not = 0.
 \end{eqnarray*}

In this case, $\delta_{i,j}^{(2)}(\gamma)$ and $\delta_{i,j}^{(2)}(1)$ reduce respectively to
\begin{eqnarray*}
\delta_{i,j}^{(32)}(\gamma) = \left\{\begin{array}{ll}
 (\gamma-1)\alpha_{j}a_{r,j} + \gamma\sum\limits_{{1 \le k \le j-1}\atop{r+1 \le k \le n}}\alpha_{k}a_{r,k}a_{k,j}, &
 i = r, j = 1,\cdots, r-1; \\
 \sum\limits_{{k = 1}\atop{k\not = r}}^n \alpha_ka_{r,k}a_{k,r}, & i = j = r;\\
 (\gamma-1)\alpha_ja_{r,j} + \gamma\sum\limits_{k = r+1}^{j-1} \alpha_ka_{r,k}a_{k,j}, &
 i = r, j = r+1,\cdots,n;\\
  0, & otherwise
\end{array}\right.
\end{eqnarray*}
and
\begin{eqnarray*}
\delta_{i,j}^{(32)}(1) = \left\{\begin{array}{ll}
\sum\limits_{{1 \le k \le j-1}\atop{r+1 \le k \le n}}\alpha_{k}a_{r,k}a_{k,j}, &
 i = r, j = 1,\cdots, r; \\
 \sum\limits_{k = r+1}^{j-1} \alpha_ka_{r,k}a_{k,j}, &
 i = r, j = r+2,\cdots,n;\\
  0, & otherwise.
\end{array}\right.
\end{eqnarray*}

Similar to the proof of Theorems \ref{thm3.19} and \ref{thm3.55},
we can prove corresponding comparison results, directly.

\begin{theorem}\label{thm3.117}
Suppose that
$0 \le \alpha_{k} \le 1$, $k = 1,\cdots,n$, $k\not = r$, and\\
$\sum_{k = 1,k\not = r}^n \alpha_ka_{r,k}a_{k,r} < 1$.
Then Theorem A is valid for $\nu = 32$.
\end{theorem}

\begin{theorem}\label{thm3.118}
Suppose that $0 \le \alpha_{k} \le 1$, $k = 1,\cdots,n$, $k\not = r$.
Then Theorem B is valid for $\nu = 32$.
\end{theorem}

The results given in Theorems \ref{thm3.117} and \ref{thm3.118}
are better than the corresponding one given in \cite[Theorem 2.2]{Wa06}.

\begin{theorem} \label{thm3.119}
Suppose that
$0 \le \alpha_{k} \lesssim 1$, $k = 1,\cdots,n$, $k\not = r$, and \\
$\sum_{k = 1,k\not = r}^n \alpha_ka_{r,k}a_{k,r} < 1$.
Then Theorem C is valid for $\nu = 32$,
provided one of the following conditions is satisfied:

\begin{itemize}
\item[(i)] $0 \le \gamma < 1$.

\item[(ii)]
$\gamma = 1$ and one of the following conditions holds:

\begin{itemize}
\item[($ii_1$)]
There exist $i\in\{1, \cdots, r\}$ and
$j\in\{1, \cdots, i-1\}\cup\{r+1, \cdots, n\}$ such that
$\alpha_ja_{r,j}a_{j,i} > 0$.

\item[($ii_2$)]
There exist $i\in\{r+2, \cdots, n\}$ and $j\in\{r+1, \cdots, i-1\}$ such that
$\alpha_ja_{r,j}a_{j,i} > 0$.

\item[($ii_3$)]
$a_{r,r+1} < 0$ and $\alpha_{r+1} > 0$.

\item[($ii_4$)]
$a_{k,k+1} < 0$, $k = 1, \cdots, r-1$.

\item[($ii_5$)]
$a_{n,1} < 0$ and $a_{k,k+1} < 0$, $k = r+1, \cdots, n-1$.

\item[($ii_6$)]
$a_{k,r} < 0$, $k = 1, \cdots, r-1$.

\item[($ii_7$)]  $a_{k,r} < 0$, $k = r+1, \cdots, n$.

\item[($ii_8$)]
$a_{n,1} < 0$ and $a_{k,n} < 0$, $k = r+1, \cdots, n-1$.

\item[($ii_9$)]  $a_{k,1} < 0$, $k = r+1, \cdots, n$.
\end{itemize}\end{itemize}
\end{theorem}

\begin{theorem}
Suppose that $0 \le \alpha_{k} \lesssim 1$, $k = 1,\cdots,n$, $k\not = r$.
Then Theorem D is valid for $\nu = 32$,
provided one of the conditions ($i$) and ($ii$) of Theorem \ref{thm3.119} is satisfied.
\end{theorem}

Similarly, let $Q_{33} = Q_6 + Q_{15}$. Then we can propose a combination.
For simplicity we set
 \begin{eqnarray*}
Q_{33} = \left(\begin{array}{ccccccc}
0      & \quad\cdots\quad  & 0     & -\alpha_{1}a_{1,r}   & 0      & \quad\cdots \quad& 0\\
\vdots & \ddots  &\vdots &\vdots                & \vdots & \vdots & \vdots\\
0      & \cdots  & 0     & -\alpha_{r-1}a_{r-1,r} & 0      & \cdots & 0\\
0      & \cdots  & 0     & 0                    & 0      & \cdots & 0\\
0      & \cdots  & 0     & -\alpha_{r+1}a_{r+1,r} & 0      & \cdots & 0\\
\vdots & \vdots  &\vdots & \vdots               &\vdots  & \ddots & \vdots\\
0      & \cdots  & 0     &-\alpha_{n}a_{n,r}  & 0      & \cdots & 0
\end{array}\right)
 \end{eqnarray*}
with $2 \le r \le n-1$, $\alpha_k \ge 0$, $k = 1, \cdots, n$, $k\not = r$ and
\begin{eqnarray*}
\sum\limits_{k = 1}^{r-1}\alpha_ka_{k,r} \not = 0 \;\;
\hbox{and} \;\; \sum\limits_{k = r+1}^{n}\alpha_ka_{k,r} \not = 0.
 \end{eqnarray*}

It is proposed for the preconditioned
Jacobi and Gauss-Seidel methods in \cite{Mi87}.

In this case, $\delta_{i,j}^{(2)}(\gamma)$ and $\delta_{i,j}^{(2)}(1)$ reduce respectively to
\begin{eqnarray*}
\delta_{i,j}^{(33)}(\gamma) = \left\{\begin{array}{ll}
 \alpha_ia_{i,r}a_{r,i}, & i = j\in\{1, \cdots, n\}\setminus\{r\};\\
 (\gamma-1)\alpha_ia_{i,r}, & i\in\{1, \cdots, n\}\setminus\{r\}, j = r;\\
 \gamma\alpha_ia_{i,r}a_{r,j}, & i = 1, \cdots, r-1, \\
 & j\in\{1, \cdots, i-1\}\cup\{r+1, \cdots, n\};\\
 \gamma\alpha_ia_{i,r}a_{r,j}, & i = r+2, \cdots, n, j = r+1, \cdots, i-1;\\
  0, & otherwise
\end{array}\right.
\end{eqnarray*}
and
\begin{eqnarray*}
\delta_{i,j}^{(33)}(1) = \left\{\begin{array}{ll}
 \alpha_ia_{i,r}a_{r,j}, & i = 1, \cdots, r-1, \\
 & j\in\{1, \cdots, i\}\cup\{r+1, \cdots, n\};\\
\alpha_ia_{i,r}a_{r,j}, & i = r+1, \cdots, n, j = r+1, \cdots, i;\\
  0, & otherwise.
\end{array}\right.
\end{eqnarray*}

Since $\delta_{r,j}^{(33)}(\gamma) = 0$, $j = 1,\cdots,n$, then the condition ($ii$) of Theorem \ref{thm3.7}
can be not satisfied.

By Corollaries \ref{coro3-9}-\ref{coro3-12}, we can prove the following results, directly.

\begin{theorem}
Suppose that $0 \le \alpha_{k} \le 1$ and $\alpha_ka_{k,r}a_{r,k} < 1$,
$k = 1,\cdots,n$, $k\not = r$.
Then Theorem A is valid for $\nu = 33$.
\end{theorem}

\begin{theorem}
Suppose that $0 \le \alpha_{k} \le 1$, $k = 1,\cdots,n$, $k\not = r$.
Then Theorem B is valid for $\nu = 33$.
\end{theorem}

\begin{theorem}\label{thm3.123}
Suppose that $0 \le \alpha_{k} \lesssim 1$
and $\alpha_ka_{k,r}a_{r,k} < 1$, $k = 1,\cdots,n$, $k\not = r$.
Then Theorem C is valid for $\nu = 33$,
provided one of the following conditions is satisfied:

\begin{itemize}
\item[(i)] $0 \le \gamma < 1$.

\item[(ii)]
$\gamma = 1$ and one of the following conditions holds:

\begin{itemize}
\item[($ii_1$)]
There exist $i\in\{1, \cdots, r-1\}$ and $j\in\{1, \cdots, i\}\cup\{r+1, \cdots, n\}$
 such that $\alpha_ia_{i,r}a_{r,j} > 0$.

\item[($ii_2$)]
There exist $i\in\{r+1, \cdots, n\}$ and $j\in\{r+1, \cdots, i\}$
 such that $\alpha_ia_{i,r}a_{r,j} > 0$.

\item[($ii_3$)] $a_{r-1,r} < 0$ and $\alpha_{r-1} > 0$.

\item[($ii_4$)] $a_{r,1} < 0$.

\item[($ii_5$)]
There exists $k\in\{r+1, \cdots, n\}$ such that $a_{r,k} < 0$.
\end{itemize}\end{itemize}
\end{theorem}

\begin{proof}
We just need to prove ($ii_3$), ($ii_4$) and ($ii_v$).

When ($ii_3$) holds, from the irreducibility of $A$, there exists $j_0\in\{1, \cdots, n\}\setminus\{r\}$
such that $a_{r,j_0} < 0$, so that $\alpha_{r-1}a_{r-1,r}a_{r,j_0} > 0$, which implies
that ($ii_1$) holds for $i = r-1$ and $j = j_0$.

If ($ii_4$) holds, then by the definition of $Q_{33}$ there exists $k_0 \{1, \cdots, r-1\}$
such that $\alpha_{k_0}a_{k_0,r} < 0$, so that $\alpha_{k_0}a_{k_0,r}a_{r,1} > 0$, which implies
that ($ii_1$) holds for $i = k_0$ and $j = 1$.

For the case when ($ii_5$) holds, the proof is completely same.
\end{proof}

\begin{theorem}\label{thm3.124}
Suppose that $0 \le \alpha_{k} \lesssim 1$, $k = 1,\cdots,n$, $k\not = r$.
Then Theorem D is valid for $\nu = 33$,
provided one of the conditions ($i$) and ($ii$) of Theorem \ref{thm3.123} is satisfied.
\end{theorem}

Clearly, there exists a permutation matrix $V$ such that $V^TQ_{33}V = Q_6$
with $q^{(6)}_{r,1} = \alpha_1a_{1,r}$, $q^{(6)}_{k.1} = \alpha_ka_{k,r}$,
$k=2,\cdots,n$, $k\not=r$. It is easy to see that the matrices $A$ and $V^TAV$
have the same irreducibility and $\rho({\mathscr J}(A)) = \rho({\mathscr J}(V^TAV))$,
$\rho({\mathscr J}(A^{(33)})) = \rho({\mathscr J}(V^TP_{33}VV^TAV))$.
Hence, by ($iii$) of Theorem \ref{thm3.23},
Theorems \ref{thm3.123} and \ref{thm3.124} for $\gamma = 0$
are valid whenever we set $\alpha_k=1$, $k=1,\cdots,n$, $k\not=r$.
So from Theorem \ref{thm3.124} it can derive \cite[Theorem 2.2]{Mi87}.

At last, in \cite{HWXC16}, $Q$ is chosen as
 \begin{eqnarray*}
Q_{34} = Q_{11} + Q_{20} = \left(\begin{array}{ccccccc}
0      && 0      && \cdots & 0 & -\frac{a_{1,n}}{\alpha_2} - \beta_2  \\
\vdots && \vdots &&  \cdots & \vdots & \vdots  \\
0      && 0      &&\cdots & 0 & 0   \\
-\frac{a_{n,1}}{\alpha_1} - \beta_1  && 0 && \cdots & 0 & 0
\end{array}\right)
 \end{eqnarray*}
with $a_{1,n} < 0$, $a_{n,1} < 0$, $\alpha_i > 0$, $i=1,2$,
$a_{n,1}/\alpha_1 + \beta_1 < 0$ and $a_{1,n}/\alpha_2 + \beta_2 < 0$.

It is proposed in \cite{LW14} for $\alpha_i = \alpha$,
$\beta_i = \beta$, $i=1,2$.

In this case, for $n \ge 3$, the condition ($ii$) of Theorem \ref{thm3.3} can be not satisfied.

By Corollaries \ref{coro3-5}-\ref{coro3-8},
it can be prove the following comparison results.

\begin{theorem}
Suppose that
\begin{eqnarray}\label{eqs3.61}
\beta_1 \ge \left(1 - \frac{1}{\alpha_1}\right)a_{n,1} , \;\;
\beta_2 \ge \left(1 - \frac{1}{\alpha_2}\right)a_{1,n}
 \end{eqnarray}
 and
 \begin{eqnarray}\label{eqs3.619}
\beta_1 > \frac{1}{a_{1,n}} - \frac{a_{n,1}}{\alpha_1}, \;\;
\beta_2 > \frac{1}{a_{n,1}} - \frac{a_{1,n}}{\alpha_2}.
 \end{eqnarray}
Then Theorem A is valid for $\nu = 34$.
\end{theorem}

\begin{theorem}
Suppose that (\ref{eqs3.61}) holds.
Then Theorem B is valid for $\nu = 34$.
\end{theorem}

\begin{theorem}
Suppose that (\ref{eqs3.619}) holds and
\begin{eqnarray}\label{eqs3.62}
\beta_1 \gtrsim \left(1 - \frac{1}{\alpha_1}\right)a_{n,1}, \;\;
\beta_2 \gtrsim \left(1 - \frac{1}{\alpha_2}\right)a_{1,n}.
 \end{eqnarray}
Then Theorem C is valid for $\nu = 34$.
\end{theorem}

\begin{proof}
In this case, for $i\not = j$, the condition $q^{(1)}_{i,j} \lesssim -a_{i,j}$ reduces to
(\ref{eqs3.62}).

On the other hand, since $a_{n,1}/\alpha_1 + \beta_1 < 0$, then the condition ($ii_2$) in
Theorem \ref{thm3.3} is satisfied. It follows by Corollary \ref{coro3-7} that Theorem C is valid.
\end{proof}

\begin{theorem}
Suppose that (\ref{eqs3.62}) holds.
Then Theorem D is valid for $\nu = 34$.
\end{theorem}

\section{Conclusions} \label{se4}

In this paper, we have investigated the preconditioned AOR method for solving linear systems.
We have studied two general preconditioners and proposed some lower triangular,
upper triangular and combination preconditioners. For $A$ being
an L-matrix, a nonsingular M-matrix, an irreducible L-matrix
and an irreducible nonsingular M-matrix, four types of comparison theorems are presented,
respectively. They contain a general comparison result, a strict comparison result and two
Stein-Rosenberg type comparison results. Our theorems include and are better than
almost all known corresponding results. We also pointed out some incorrect results.

When $(\gamma,\omega)$ is equal to $(\omega, \omega)$, (1, 1) and (0, 1),
from the results above, we can derive respectively the corresponding comparison results about
the preconditioned SOR method, Gauss-Seidel method and Jacobi method directly.

Similar to \cite{AH04,DHS11,DY12,EM95,EMT96} for the block preconditioned Jacobi, Gauss-Seidel, SOR methods and
the block preconditioned AOR method respectively,
when $A$ is partitioned by block, then $Q$ can be chosen as a block matrix,
so that we can derive the same comparison results for the block preconditioned AOR method.

Similar to \cite{BW15,HWLC15,HWXC18,MLW18,WDT12,Yu122,ZLWL14,ZSWL09} for the preconditioned AOR methods for solving linear least
squares problems, we can derive the corresponding comparison results as above.

The comparisons between either different preconditioners
or different parameters of a same type preconditioner
have investigated by many authors in
\cite{AH04,BW15,HNT03,HS14,HCC06,HCEC05,KHMN02,KN09,KN10,Li11,LW14,LJ10,Li05,Li052,LL07,LC09,LHZ15,Mi87,
Mo10,NHMS04,NK10,NKA09,NKM08,SE13,SER14,Su05,WZ11,WL09,Wa06,Wa062,WH06,YZ12,Yu07,Yu072,Yu12,YLK14}.
This is an important and interesting research subject.
Because of the length, this paper does not cover this topic.

In \cite{HDS10,HWF07,KC03,KNO97,LY08,Li08,LC10,LCC08,SE15,Su06,WS16,WS09,WWS07,Yu07} and some related literatures,
the preconditioned method for H-matrix is studied.

Recently, in \cite{LLV18}, the preconditioned tensor splitting method for solving multi-linear systems is proposed.
This is a new subject to be studied.

\section*{Acknowledgements}

This work is supported by the National Natural Science Foundation
of China, grant No. 11771213 and 11971242, the Priority Academic Program
Development of Jiangsu Higher Education Institutions.

\end{document}